\documentclass{mcom-l}
\usepackage{graphicx}
\usepackage{leftidx}
\usepackage{multirow}
\usepackage{booktabs}
\usepackage{threeparttable}
\usepackage{amsmath}
\usepackage{amssymb}
\usepackage{latexsym}
\usepackage{mathdots}
\usepackage{epstopdf}
\usepackage{color}
\usepackage[font=normalsize]{subfig}
\usepackage{geometry}
\usepackage[hidelinks]{hyperref} 
\newtheorem{theorem}{Theorem}[section]

\theoremstyle{definition}

\newtheorem{example}[theorem]{Example}

\theoremstyle{remark}
\newtheorem{remark}[theorem]{Remark}

\theoremstyle{conjecture}

\theoremstyle{corollary}
\newtheorem{corollary}[theorem]{Corollary}

\numberwithin{equation}{section}

\begin{document}

\title[Convergence rates of LP with uniform exponentially clustered poles]{Exact convergence rates of lightning plus polynomial approximation for branch singularities with uniform exponentially clustered poles}

\author{Shuhuang Xiang}
\address{School of Mathematics and Statistics, HNP-LAMA, Central South University, Changsha, Hunan 410083, P. R. China.}
\email{xiangsh@mail.csu.edu.cn.}
\author{Yanghao Wu}
\address{School of Mathematics and Statistics, HNP-LAMA, Central South University, Changsha, Hunan 410083, P. R. China.}
\email{wyanghao96@163.com.}

\author{Shunfeng Yang}
\address{College of Science, Southwest Forestry University, Kunming, Yunnan 650224, P. R. China.}
\email{yangshunfeng@163.com.  Corresponding author.}


\date{}

\dedicatory{}

\begin{abstract}
This paper builds rigorous analysis on the root-exponential convergence for the lightning schemes via rational functions in approximating corner (branch) singularity problems with uniform exponentially clustered poles proposed by Gopal and Trefethen.
The start point is to set up the integral representations of $z^\alpha$ and $z^\alpha\log z$ in the slit disk and develop results akin
to Paley-Wiener theorem,  from which, together with the Poisson summation formula,  the root-exponential convergence of the lightning plus polynomial scheme with  an exact order for each clustered parameter is established in approximation of  prototype functions  $z^{\alpha}$
or $z^\alpha\log z$ on a sector-shaped domain, which includes $[0,1]$ as a special case. In addition, the fastest convergence rate
is confirmed based upon the best choice of the clustered parameter.
Furthermore, the optimal selection of the clustered parameter is employed in conformal mappings through solving Laplace problems on corner domains, building upon Lehman and Wasow's analysis of corner singularities and incorporating the domain decomposition method proposed by Gopal and Trefethen.
\end{abstract}

\maketitle

\section{Introduction}
\label{sec:Int}
In the study of partial differential equations in corner domains, solutions may exhibit isolated branch points at the corners \cite{Lehman1954DevelopmentsIT,Lehman1957DevelopmentOT,Wasow}. Standard techniques for solving these problems face significant challenges in achieving accurate solutions \cite{Gopal2019}. However, recent advancements have led to the development of efficient and powerful lightning schemes, particularly lightning plus polynomial schemes, which utilize rational functions to address corner singularities \cite{Brubeck2022,costa2020solvinglaplaceproblemsaaa,costa2023aaa,Gopal20192,Gopal2019,Herremans2023,Trefethen2024,Trefethen2025,TNWNM2021,Xue2024}. These methods have demonstrated root-exponential convergence through extensive numerical experiments in solving Laplace \cite{Gopal20192,Gopal2019}, Helmholtz \cite{Gopal20192}, and biharmonic equations (Stokes flow) \cite{Brubeck2022,Xue2024}.

For singularity problems, rational functions can achieve much faster convergence rates than polynomials. A fundamental result of rational approximation is due to Newman \cite{Newman1964} concerning the approximation of the absolute value function $f(x)=|x|$ by
\begin{align}\label{newman}
  r_N(x)=x\frac{p_N(x)-p_N(-x)}{p_N(x)+p_N(-x)},\quad p_N(x)=\prod_{k=0}^{N-1} (x+\xi^k),\quad \xi=\exp(-\sqrt{N})
\end{align}
for $x \in [-1,1]$,  which attains a root-exponential convergence rate \cite{XieZhou2004}
\begin{align*}
  \lim_{N\rightarrow \infty}\sqrt{N}e^{\sqrt{N}} \max_{x\in [-1,1]}|f(x)-r_N(x)| =\max_{ x\in[0,+\infty)}\frac{x}{1+e^x}= 0. 27846\ldots.
\end{align*}
This result demonstrates a substantially superior convergence rate compared to the first-order convergence with polynomial approximation $\|f-p_N^*\|_{C[-1,1]}=\mathcal{O}(N^{-1})$ \cite{Bernstein1914,Poussin1908}, where $p_N^*$ is the best approximation polynomial of degree $N$. 

More generally, Stahl \cite{Stahl2003} showed  that the best rational approximant $r^*_{N}(x)$ of degree $N$
for $|x|^\alpha$ on $[-1,1]$ satisfies
\begin{align}\label{eq:newmann2}
  \lim_{N \to \infty} e^{\pi \sqrt{\alpha N}}  \max_{x\in [-1,1]}\big||x|^\alpha-r^*_N(x)\big|
  = 4^{1+\alpha/2} \left|\sin\left(\frac{\alpha\pi}{2}\right)\right|
\end{align}
or equivalently for $x^\alpha$ on $[0,1]$ 
\begin{align}\label{eq:newmann1}
  \lim_{N \to \infty} e^{2\pi \sqrt{\alpha N}} \max_{x\in [0,1]}\big|x^\alpha-r^*_N(x)\big| = 4^{1+\alpha}|\sin(\alpha\pi)|
\end{align}
for each $\alpha>0$. A  rational function $r=p/q$, where $p$ and $q$ are polynomials, is said to be of degree $N$ if the degrees of both $p$ and $q$ do not exceed $N$, while $r$ is of type $(m,n)$ if the degree of $p$  $\le m$ and $q$ $\le n$, respectively.

In the realm of scientific computing for planar corner domains, extensive investigations into singularity boundaries defined by analytic curves intersecting at corners have been conducted by Lewy \cite{1950Lewy}, Lehman \cite{Lehman1954DevelopmentsIT,Lehman1957DevelopmentOT} and Wasow \cite{Wasow}.
The corner domain $\Omega$ may consist of either straight or curvy sides, whose
interior angles are $\varphi_1\pi, \cdots,\varphi_m\pi$ (all $\varphi_k\in(0,2)$). In the case of curvy sides, these angles are determined by the tangent rays of the sides of  $L_{k,j}$ ($j=1,2$) at the common vertex $w_k$ (see {\sc Figure} \ref{tangent_covering_domain}).

\begin{figure}[htbp]
\vspace{-.3cm}
\centerline{\includegraphics[width=11cm]{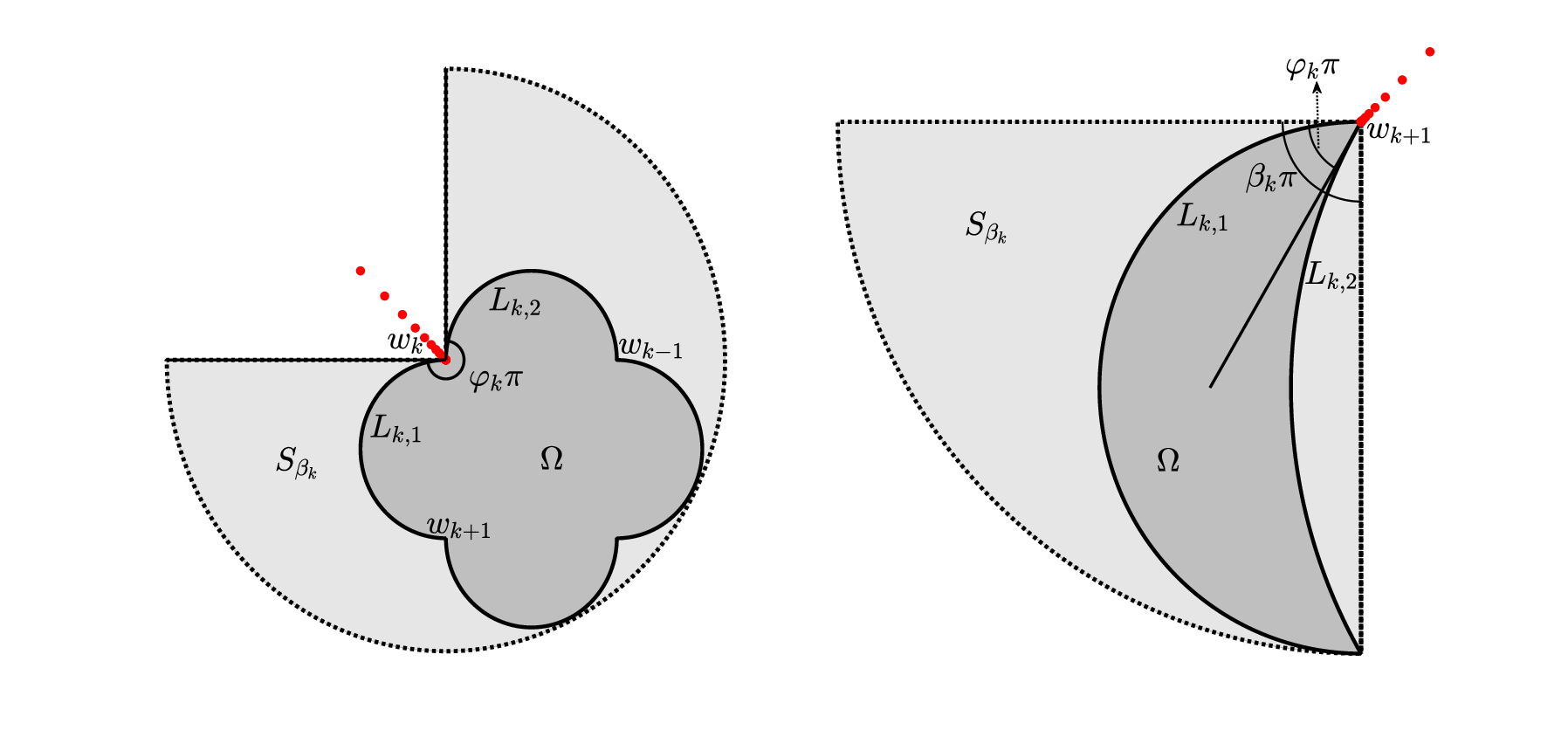}}\vspace{-.5cm}
\caption{Curvy domains with an interior angle $\varphi_k\pi$, determined by the tangent rays extending from the common vertex. Additionally, all these domains can be covered by a sufficiently large sector domain centered at the vertex, with a radius angle $\beta_k\pi$ coinciding with or larger than the interior angle $\varphi_k\pi$. The red points illustrate the distribution of the clustering poles around vertex $w_k$.}
\label{tangent_covering_domain}
\end{figure}

As is known to all,
the solution $u(x,y)$ of the Laplace equation in domain $\Omega$ is the real part of a holomorphic function $f(z)$.
According to \cite[Theorem 1]{Lehman1957DevelopmentOT} and \cite[Theorems 3, 4 and 5]{Wasow},
for piecewise analytic boundary
data exhibiting a jump in the first derivative at the corner point $w_k$, the holomorphic function $f(z)$ corresponding to $u(x,y)$
can be asymptotically represented in any finite sector around $w_k$ by a power series in terms of two variables $z-w_k$ and $(z-w_k)^{\alpha_k}$ if $\varphi_k$ is irrational, and in terms of three variables $z-w_k$, $(z-w_k)^{\alpha_k}$ and $(z-w_k)^{\mu_k}\log{(z-w_k)}$ if $\varphi_k$ is rational as $z\rightarrow w_k$, where $\alpha_k=\frac{1}{\varphi_k}$ and
$\varphi_k=\frac{\mu_k}{q_k}$, $(\mu_k,q_k)=1$ if $\varphi_k$ is rational.

The coexistence of multiple intrinsic singularities in
$f$ poses significant analytical challenges in directly quantifying the convergence rate of rational approximations over
$\Omega$.
To circumvent this obstruction, we employ the Gopal-Trefethen canonical decomposition  \cite[the proof of Theorem 2.3]{Gopal2019}, which resolves
$f(z)$ into a superposition of $2m$ Cauchy-type contour integrals
\begin{align}\label{decompose_singularity}
f(z)=\frac{1}{2\pi i}\sum_{k=1}^m\int_{\Lambda_k}\frac{f(\zeta)}{\zeta-z}\mathrm{d}\zeta
+\frac{1}{2\pi i}\sum_{k=1}^m\int_{\Gamma_k}\frac{f(\zeta)}{\zeta-z}\mathrm{d}\zeta
=:\sum_{k=1}^m f_k(z)+\sum_{k=1}^m g_k(z),
\end{align}
where  $\Lambda_k$ consists of the two sides of an exterior bisector at $w_k$, while $\Gamma_k$ links the end of the slit contour at vertex $w_k$ to the beginning of the slit contour at vertex $w_{k+1}$  (where $w_{m+1}=w_1$ by definition) (see {\sc Figure} \ref{decompose_path} for example),
each $g_k$ is holomorphic in $\mathbb{C}\setminus \Gamma_k$ containing $\Omega$, and $f_k$ is holomorphic in $\mathbb{C}\setminus \Lambda_k$.

\begin{figure}[htbp]
\vspace{-.5cm}
\centerline{\includegraphics[width=8cm]{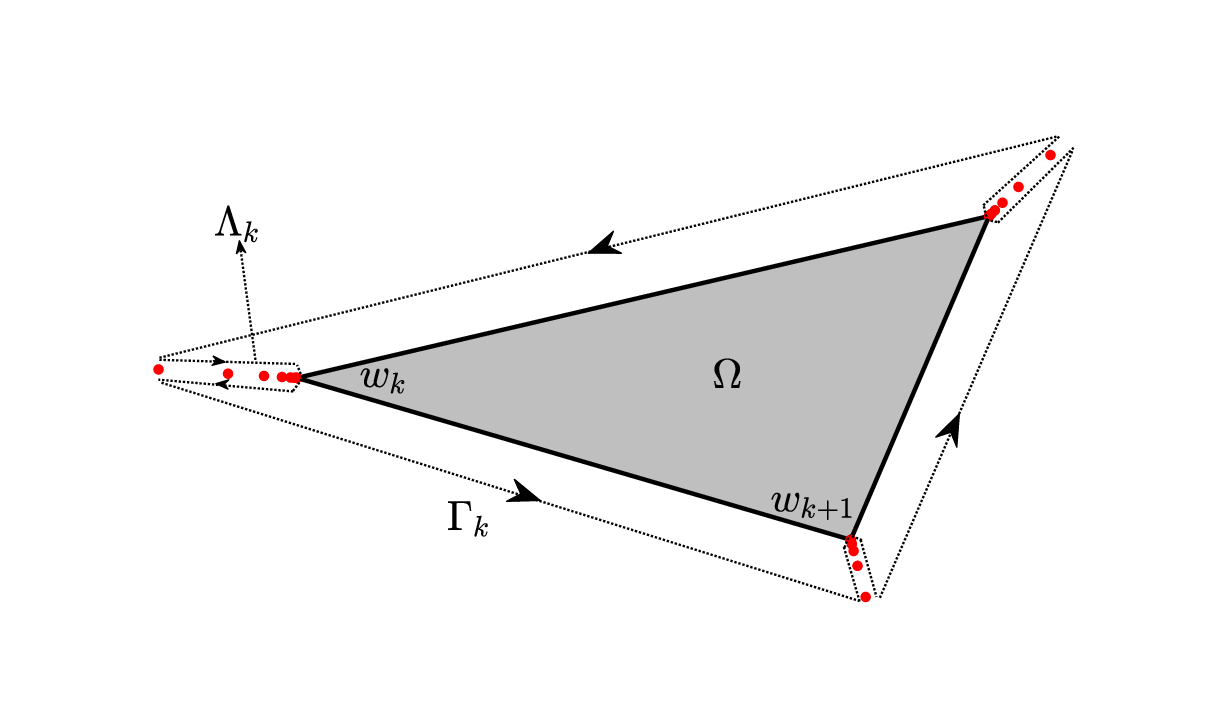}}\vspace{-.5cm}
\caption{This figure is cited from \cite[{\sc Fig. 3}]{Gopal2019}: A holomorphic function $f(z)$ defined in the corner domain $\Omega$ is decomposed as the sum of $2m$ Cauchy-type integrals: $\sum_{k=1}^m f_k(z)+\sum_{k=1}^m g_k(z)$, with $f_k(z)=\frac{1}{2\pi i}\int_{\Lambda_k}\frac{f(\zeta)}{\zeta-z}d\zeta$ along the two sides of an exterior bisector slit to each corner, and $g_k(z)=\frac{1}{2\pi i}\int_{\Gamma_k}\frac{f(\zeta)}{\zeta-z}d\zeta$ along each line segment connecting the beginnings and ends of those slit contours.}
\label{decompose_path}
\end{figure}

Subsequently, by Runge's approximation theorem \cite[pp. 76-77]{Gaier1987} and \cite[pp. 8-9]{Walsh1965}, the term $\sum_{k=1}^mg_k(z)$ can be uniformly approximated  with an exponential convergence rate by a polynomial with a lower degree on $\Omega$.
Therefore, in order to find  an efficient and highly accurate rational approximation $r_n(z)$ for  holomorphic function $f(z)$ on $\Omega$, it can be transformed into constructing a rational approximation $r_{N,k}$ for each $f_k$ on $\Omega$. The summation of these rational approximations coupled with the above low-degree polynomial constructs a new rational approximation for $f(z)$  on $\Omega$.

It is worth noting that $f$ and $f_k$ have the same singularity  in any finite sector domain around the vertex $w_k$ as $z\rightarrow w_k$ \cite{Lehman1957DevelopmentOT,Wasow}, and $f_k$ is singular only at $w_k$ on $\Omega$. In light of the intricate geometry of
$\Omega$, then in the sequel, we will consider  rational approximation for $f_k$ on an optimal sectorial region covering $\Omega$ around the vertex singularity $w_k$ similar to \cite{Gopal20192,Gopal2019,TNWNM2021}: Suppose that the boundary $\partial\Omega$ is a Jordan curve, and  $\Omega$ can be covered by some sectors centred at $w_k$ with sufficiently large radii and angles, among which the smallest one is denoted by $\mathcal{S}_{\beta_k}$, with radius angle $\beta_k\pi\ge\varphi_k\pi$.
If  the two tangent rays at  corner point $w_k$ are outside of $\Omega$, $\beta_k=\varphi_k$, and otherwise, $\beta_k>\varphi_k$ (see {\sc Figure} \ref{tangent_covering_domain} for example).

Consequently, the numerical methods for these classic boundary value problems naturally transform into how to construct effective approximation formats for rapidly approximating phototype functions
\(f(z)=z^\alpha\) and $f(z)=z^\alpha\log{z}$ in the standard sector-shaped domain $S_\beta$  (see {\sc Figure} \ref{Vsector})  with $\alpha>0$.
\begin{figure}[hpbt]
\centerline{\includegraphics[width=8cm]{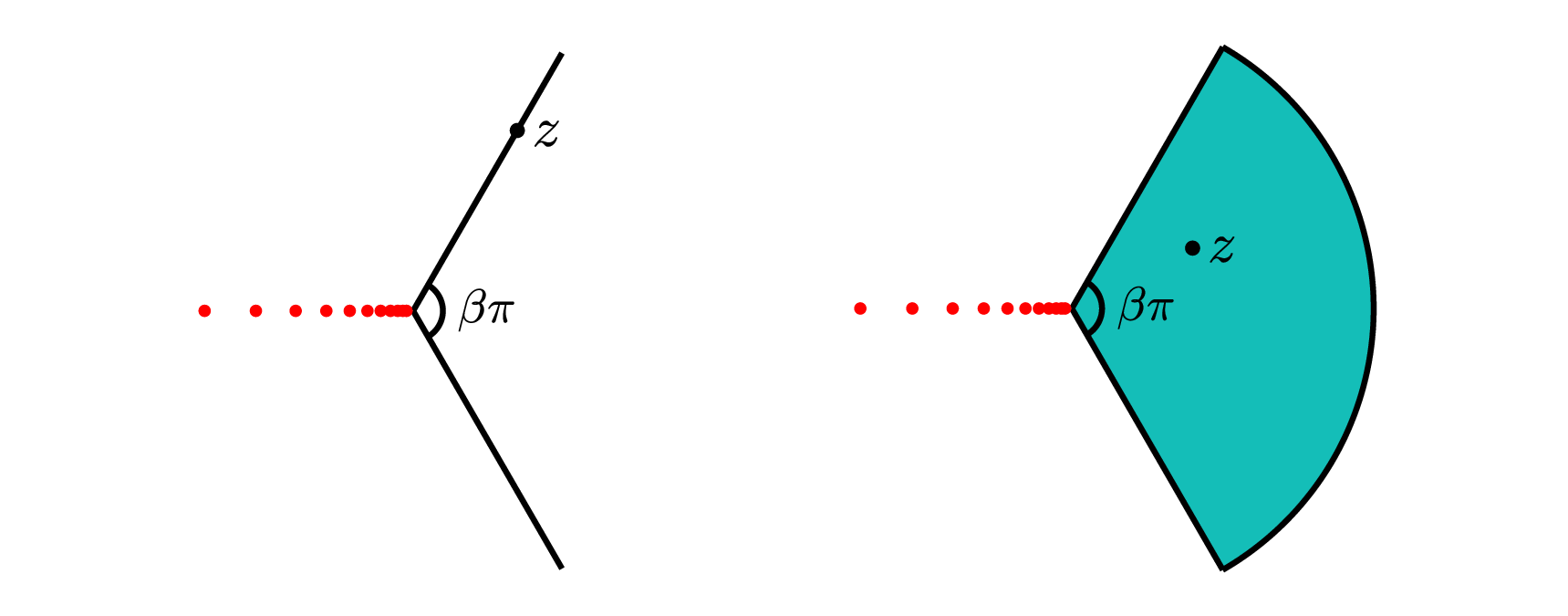}}
\caption{V-shaped domain (left):
$V_\beta=\big\{z: \, z=xe^{\pm \frac{\beta\pi}{2}i}$
with $x\in[0,1]\big\}$ and
sector domain (right):
$S_{\beta}=\big\{z: \, z=xe^{\pm \frac{\theta\pi}{2}i}$ with $x\in [0,1]$ and $\theta\in [0,\beta]\big\}$
for fixed $\beta\in [0,2)$.
The red points illustrate the distributions of the clustering poles \eqref{eq:uniform0}.}
\label{Vsector}
\end{figure}

A powerful and robust lightning plus polynomial scheme (LP) by using a rational function
\begin{equation}\label{eq:rat}
r_N(z)=\frac{p(z)}{q(z)}=\sum_{j=0}^{N_1}\frac{a_j}{z-p_j}+\sum_{j=0}^{N_2} b_jz^j:=r_{N_1}(z)+P_{N_2}(z),\, N=N_1+N_2+1
\end{equation}
to approximate  functions  with corner singularities at $z=0$ on sector domain $S_\beta$ was first introduced in Gopal and Trefethen \cite{Gopal20192,Gopal2019}
by introducing the uniform exponentially clustered poles
\begin{equation}\label{eq:uniform0}
p_j =-C\exp\big(-\sigma j/\sqrt{N_1}\big),\quad 0\leq j\leq N_1.
\end{equation}
Consequently, for a function $f(z)$ defined on $\Omega$ with isolated branch points at vertices $w_k$ ($k=1,\cdots,m$), an LP approximation of the form
\begin{align}\label{LP_cornerdomain}
r_n(z)=\sum_{k=1}^m\sum_{j=0}^{N_{1,k}}\frac{a_{k,j}}{z-p_{k,j}}+\sum_{j=0}^{N_2} b_{j}z^j
\end{align}
can provide an excellent approximation. Here, the lightning poles $\{p_{k,j}\}$ (See {\sc Figures}  \ref{tangent_covering_domain} and \ref{general_domain} for example) are uniformly exponentially clustered with parameter $\sigma_k$
  along the exterior bisector of each corner $w_k$
 in the sector domain $S_{\beta_k}$. Confer to \cite{Gopal20192,Gopal2019} for more details.

 \begin{figure}[ht!]
\centerline{\includegraphics[width=15cm]{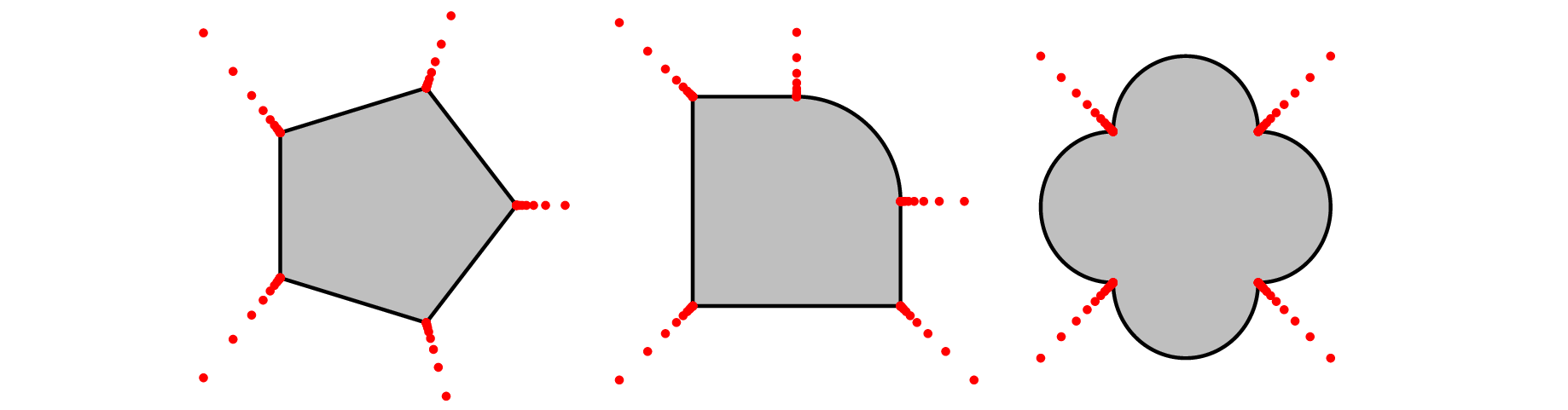}}
\caption{Various corner domains: pentagon (first), curvy pentagon (second) and quincunx-shaped (third) domains.
The red points illustrate the distributions of the clustering poles.}
\label{general_domain}
\end{figure}

To analyse the root-exponential convergence of LP \eqref{eq:rat},  Gopal and Trefethen \cite{Gopal2019} considered a rational interpolation  with poles $p_j=-C\exp\big(-\sigma j/\sqrt{N}\big),\, j=0,1,\ldots,N-1$ and interpolation nodes
\begin{align}\label{Gopalintnodes}
z_0=0,\quad z_j=-p_j,\quad j=1,2,\ldots,N-1
\end{align}and showed the root-exponential convergence \cite[Theorem 2.2]{Gopal2019} based on Walsh's Hermite integral formula \cite[Theorem 2 of Chapter 8]{Walsh1965}.

\begin{theorem}\label{interGopal} \cite[Theorem 2.2]{Gopal2019}
 Let $f$ be a bounded analytic function in the slit disk $S_\beta$ that satisfies $f(z) =\mathcal{O}(|z|^\delta)$ as $z\rightarrow 0$ for some $\delta>0$
and let $\beta\in (0,1)$ be fixed. Then for
some $\hat{\eta}\in (0, 1)$ depending on $\beta$ but not $f$, there exist type $(N-1, N)$ rational functions
$\{\check{r}_N\}_{N=1}^{\infty}$, such that
\begin{align}\label{Gopal2019}
 \|f- \check{r}_N\|_{C(\Omega)}=\mathcal{O}(e^{-C_0\sqrt{N}})
\end{align}
as $n\rightarrow \infty$ for some $C_0>0$, where $\Omega=\hat{\eta}S_\beta$.
\end{theorem}

In Theorem \ref{interGopal}, the constant $C_0$
remains unspecified. Moreover, from the numerical results illustrated in {\sc Figure} \ref{FIG2_uniformcluster}, we see that the rational interpolant $\check{r}_N$ 
exhibits a significantly slower convergence rate in the approximation of $z^\alpha$ compared to LP \eqref{eq:rat} in \cite{Gopal20192,Gopal2019,Herremans2023} with  the best parameter $\sigma=\sigma_{\mathrm{opt}}=\frac{\pi\sqrt{2-\beta}}{\sqrt{\alpha}}$ and $N_2=\mathcal{O}(\sqrt{N})$ (see  Theorem \ref{mainthm}). In addition, the assumption $\beta\in \left(0,1\right)$ in Theorem \ref{interGopal} is essential and cannot be removed (see the right of {\sc Figure} \ref{FIG2_uniformcluster}).

\begin{figure}[htbp]
\centerline{\includegraphics[width=13cm]{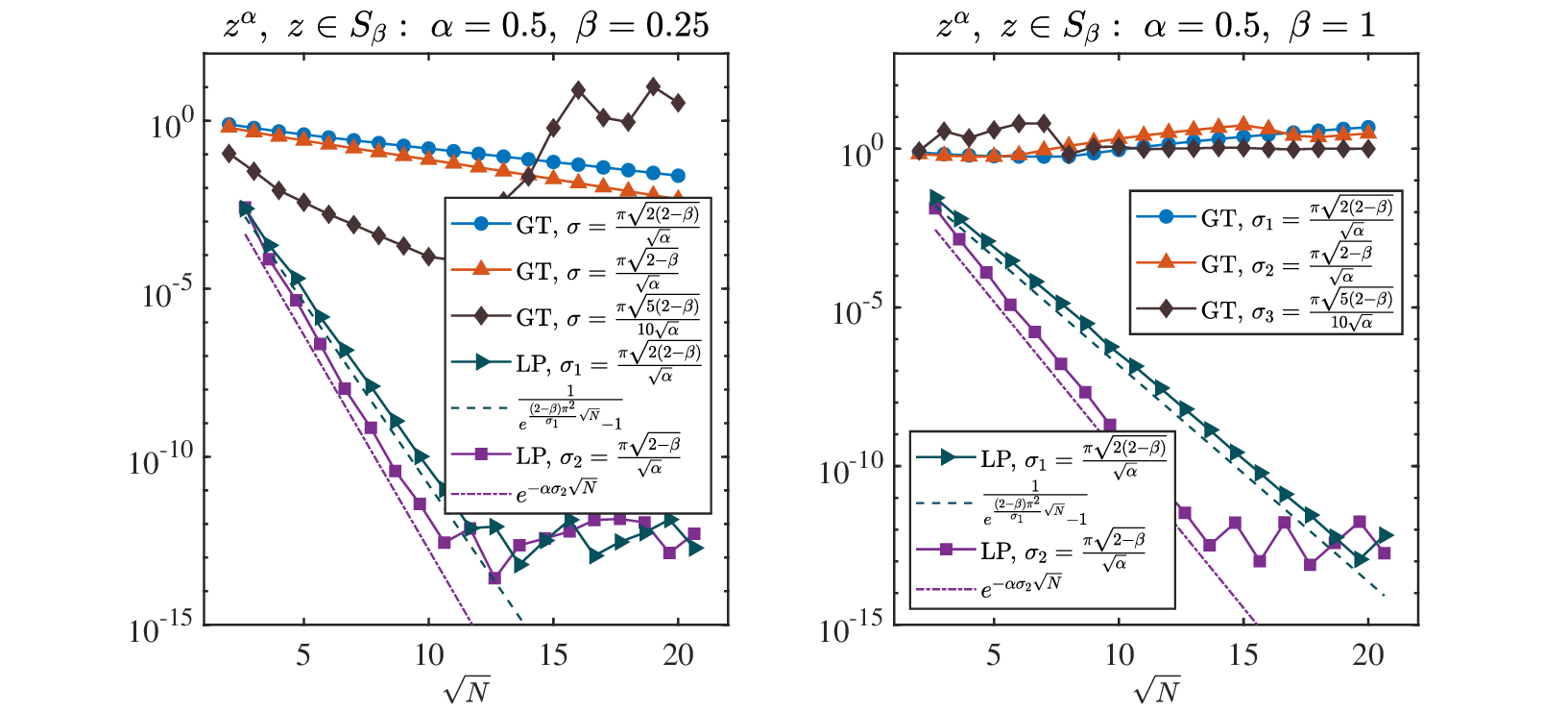}}
\caption{Decay rates of approximation errors $\|z^{\alpha}-\check{r}_N(z)\|_{C(S_{\beta})}$ of Gopal and Trefethen's interpolation (GTs) in \cite{Gopal2019} are compared with the LP \eqref{eq:rat} with $N_2={\rm ceil}(1.3\sqrt{N_1})$ on $S_{\beta}$ with various values of $\sigma$ as well as $\alpha$, $\beta$, where $N$ is the degree of rational approximation. The lightning parameter $\sigma_2=\sigma_{\mathrm{opt}}\left(=\frac{\pi\sqrt{2-\beta}}{\sqrt{\alpha}}\right)$ is the optimal choice among all of $\sigma>0$ to get the corresponding fastest convergence rate.}
\label{FIG2_uniformcluster}
\end{figure}

To explore and accelerate the convergence rate of the rational interpolation $\check{r}_N$ and overcome the restriction $\beta\in \left(0,1\right)$, Trefethen, Nakatsukasa and Weideman \cite{TNWNM2021} considered the interpolation nodes obtained from the following potential function
\begin{align}\label{PT2021}
u(z)=\frac{1}{N}\sum_{k=0}^{N}\log |z- z_k|-\frac{1}{N}\sum_{k=0}^{N-1}\log |z- p_k|,
\end{align}
which approximates the potential function
\begin{align*}
u(z)=-\int\log|z-t|\mathrm{d}\mu(t),
\end{align*}
and approximately minimizes the energy
 \begin{align*}
 I(\mu)=-\iint\log|z-t|\mathrm{d}\mu(z)\mathrm{d}\mu(t),
 \end{align*}
where $\mu$ is a signed measure and defines a continua of interpolation points and poles. See \cite{TNWNM2021} for details.
As a result, the corresponding rational interpolation approximation may also achieve a root-exponential convergence rate for specific uniform exponentially clustered poles
\begin{equation*}
p_j =-C\exp\big(-\pi j/\sqrt{N}\big),\quad 0\leq j\leq N-1,
\end{equation*}
and tapered exponential clustering of the poles
\begin{equation*}
q_j =-C\exp\left(-\sqrt{2}\pi\big(\sqrt{N}-\sqrt{j}\big)/\sqrt{\alpha}\right),\quad 1\le j\le N.
\end{equation*}
The accuracy of the lightning methods using these two exponentially clustered poles for  approximating the prototype function $x^\alpha$  with $0<\alpha<1$ on $[0,1]$  is $\mathcal{O}(e^{-\pi\sqrt{\alpha N}})$
and $\mathcal{O}(e^{-\pi\sqrt{2\alpha N}})$, respectively.

To achieve the minimax convergence rate $\mathcal{O}(e^{-2\pi\sqrt{\alpha N}})$ in \eqref{eq:newmann1}, Herremans, Huybrechs and Trefethen \cite{Herremans2023} introduced a new LP \eqref{eq:rat} with $N_2=\mathcal{O}(\sqrt{N_1})$
based upon
a new type of tapered exponential clustering
\begin{equation*}
q_j =-C\exp\left(-\sigma\big(\sqrt{N_1}-\sqrt{j}\big)\right),\quad 1\leq j\leq N_1
\end{equation*}
to approximate $x^\alpha$  and $x^\alpha\log x$ on $[0,1]$,  and $z^\alpha$ on a V-shaped domain
$$V_{\beta}=\big\{z: \, z=xe^{\pm \frac{\beta\pi}{2}i} \mbox{\, with\, $x\in [0,1]$}\big\}$$
for fixed $\beta\in [0,2)$.
Wherein, the optimal choice of
$\sigma=\frac{2\pi}{\sqrt{\alpha}}$ for $x^\alpha$ and $x^\alpha\log x$, while $\sigma=\frac{\pi\sqrt{2(2-\beta)}}{\sqrt{\alpha}}$ for $z^\alpha$ on $V_\beta$, are confirmed by ample numerical examples. In addition,  the root-exponential convergence rates $\mathcal{O}\big(e^{-2\pi\sqrt{\alpha N}})$ for $x^\alpha$ and $\mathcal{O}\big(e^{-\pi\sqrt{2(2-\beta)\alpha N}})$ for $z^\alpha$ are acquired, respectively, based on three conjectures  \cite{Herremans2023}.
However, achieving good approximation on the V-shaped region does not necessarily ensure accuracy within the entire region $\Omega$.

{\sc Figure} \ref{FIG3_uniformcluster} illustrates the performance of the new interpolant $\widehat{r}_N$ considered in \cite{TNWNM2021} and poles \eqref{eq:uniform0}.
It shows that the nodes $\{z_k\}_{k=0}^{N}\subseteq V_{\beta}$ chosen based on the potential calculated by BIEP \cite{ZX2024} improve the rational interpolation efficiently, while it is still much slower than LPs \eqref{eq:rat} with $\sigma_{\rm opt}$ and even fails on the V-shaped domain $V_{\beta}$ with a larger radius angle.

\begin{figure}[hpbt]
\centerline{\includegraphics[width=12cm]{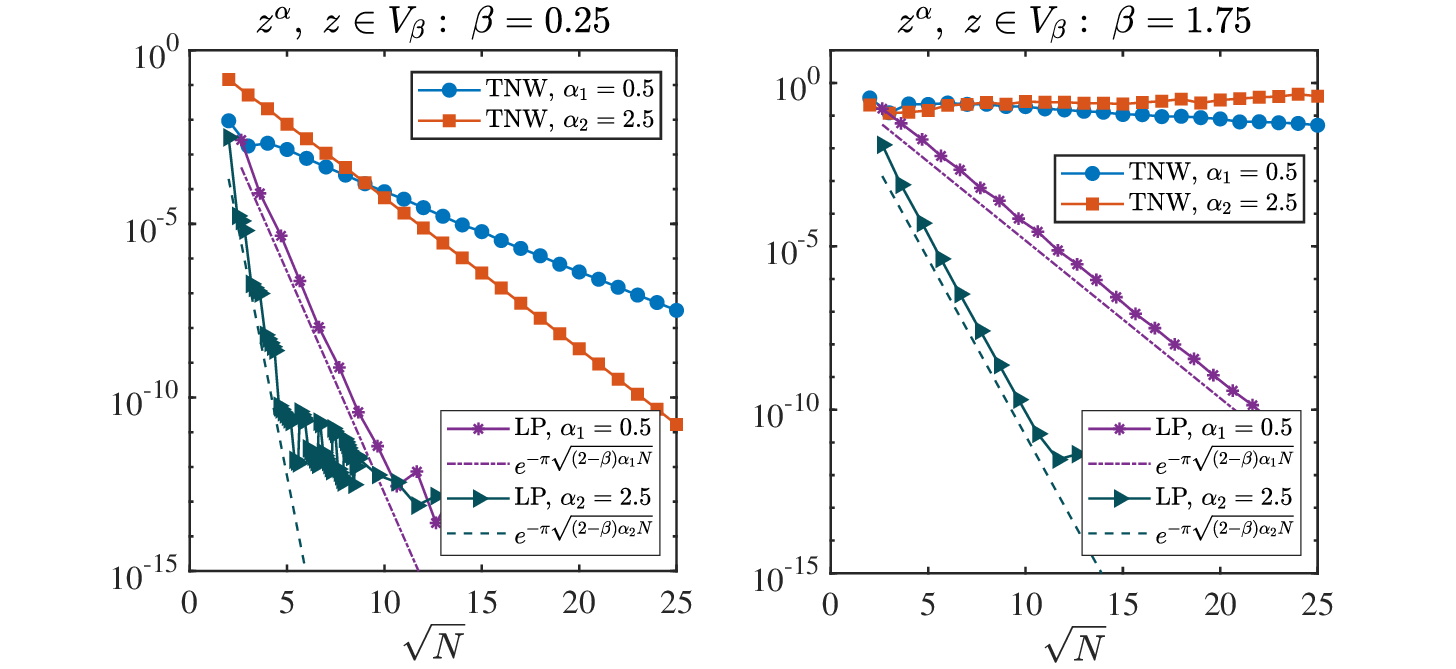}}
\caption{Decay rates of approximation errors $\|z^{\alpha}-\widehat{r}_N(z)\|_{C(V_{\beta})}$ of Trefethen, Nakatsukasa and Weideman's interpolation (TNWs) in \cite{TNWNM2021}  are compared with the LP \eqref{eq:rat} with $N_2={\rm ceil}(1.3\sqrt{N_1})$ for $z^{\alpha}$ on $S_{\beta}$ with various values of $\alpha$ and $\beta$, where $N$ is the degree of rational approximation, and we choose the optimal lightning parameter $\sigma=\sigma_{\mathrm{opt}}\left(=\frac{\sqrt{2-\beta}\pi}{\sqrt{\alpha}}\right)$ for the LPs.}
\label{FIG3_uniformcluster}
\end{figure}

From the above numerical tests, we see that LP \eqref{eq:rat} with poles \eqref{eq:uniform0} and $N_2=\mathcal{O}(\sqrt{N_1})=\mathcal{O}(\sqrt{N})$  exhibits root-exponential convergence with  an exact order in the approximation of $z^\alpha$ in $S_\beta$, and  significantly outperforms the rational interpolants in  \cite{Gopal20192,Gopal2019,TNWNM2021,ZX2024}, when using the optimal choice of $\sigma$.
However, existing theoretical frameworks fail to fully characterize the root-exponential convergence behavior of LP \eqref{eq:rat}, especially for parameter values $\beta\in [0,2)$ and when solving the Laplace equation on corner domains $\Omega$ with piecewise analytic boundary conditions.

The goal of this paper is to lay the rigorous groundwork for the lightning scheme \eqref{eq:rat} \cite{Gopal20192,Gopal2019}.  With the help of Cauchy's integral theorem and residue theorem, we firstly derive the integral representations of $z^\alpha$ and $z^\alpha\log z$.
By employing the integral representations, along with Runge's approximation theorem and Poisson summation formula \cite{Henrici}, we shall prove theoretically the root-exponential convergence rates of the LPs on  $S_\beta$ and acquire the optimal convergence rate, where the assumption $0\le \beta <1$ in Theorem \ref{interGopal} is removed.

\begin{theorem}\label{mainthm}
Let $\alpha$ and $\sigma$ be positive real numbers, $\sigma_{\rm opt}=\frac{\sqrt{2-\beta}\pi}{\sqrt{\alpha}}$ and
$\eta=\frac{\sigma_{\rm opt}}{\sigma}$. If $g(z)$ is analytic in a neighborhood of $S_\beta$, then there exist coefficients $\{\bar{a}^{(g)}_j\}_{j=0}^{N_1}$, $\{\widetilde{a}^{(g)}_j\}_{j=0}^{N_1}$ and polynomials $\bar{P}^{(g)}_{N_2}$, $\widetilde{P}^{(g)}_{N_2}$ with degree
$N_2 = \mathcal{O}(\sqrt{N_1})=\mathcal{O}(\sqrt{N})$, for which the LPs of the form \eqref{eq:rat} to $g(z)z^\alpha$ or $g(z)z^\alpha\log z$ furnished with the poles \eqref{eq:uniform0}
satisfy
\allowdisplaybreaks
\begin{align}
|\bar{r}^{(g)}_N(z)- g(z)z^\alpha|=&\frac{\mathcal{G}^{\alpha}\max\{1,C^{\alpha}\}}{\varkappa(\beta)}
\left[\frac{\mathcal{O}(1)}{e^{\frac{(2-\beta)\pi^2}{\sigma}\sqrt{N}}-1}
+\mathcal{O}(1)e^{-\alpha\sigma\sqrt{N}}\right]\notag\\
=&\frac{\mathcal{G}^{\alpha}\max\{1,C^{\alpha}\}}{\varkappa(\beta)}
\left\{\begin{array}{ll}
\frac{\mathcal{O}(1)}{e^{\sigma\alpha\sqrt{N}}},&\sigma< \sigma_{\rm opt},\\
\frac{\mathcal{O}(1)}{e^{\pi\sqrt{(2-\beta)N\alpha}}},&\sigma= \sigma_{\rm opt},\\
\frac{\mathcal{O}(1)}{e^{\pi\eta\sqrt{(2-\beta)N\alpha}}-1},&\sigma> \sigma_{\rm opt},
\end{array}\right.\label{eq: rate1}\\
|\widetilde{r}^{(g)}_N(z)- g(z)z^\alpha\log{z}|=&\frac{\mathcal{G}^{\alpha}\max\{1,C^{\alpha}\}}{(\alpha+1)^{-1}\alpha\varkappa(\beta)}
\left[\frac{\mathcal{O}(1)}{e^{\frac{(2-\beta)\pi^2}{\sigma}\sqrt{N}}-1}
+\mathcal{O}(1)\sigma\sqrt{N}e^{-\alpha\sigma\sqrt{N}}\right]\notag\\
=&\frac{\mathcal{G}^{\alpha}\max\{1,C^{\alpha}\}}{(\alpha+1)^{-1}\alpha\varkappa(\beta)}
\left\{\begin{array}{ll}
\frac{\mathcal{O}(1)\sigma\sqrt{N}}{e^{\sigma\alpha\sqrt{N}}},&\sigma< \sigma_{\rm opt},\\
\frac{\mathcal{O}(1)\sigma\sqrt{N}}{e^{\pi\sqrt{(2-\beta)N\alpha}}},&\sigma=\sigma_{\rm opt},\\
\frac{\mathcal{O}(1)}{e^{\pi\eta\sqrt{(2-\beta)N\alpha}}-1},&\sigma> \sigma_{\rm opt},
\end{array}\right.\label{eq: rate2}
\end{align}
as $N \to\infty$, uniformly for $z\in S_{\beta}$, where $\varkappa(\beta)=1$ for $0\le\beta<1$ and $\varkappa(\beta)=\sin\frac{\beta\pi}{2}$ for $1\le\beta<2$, $\mathcal{G}=\frac{\sqrt{2}+2}{\sqrt{2}-1}=8.24264068711928\cdots$, and all the constants in the above $\mathcal{O}$ terms are independent of $\alpha$, $\sigma$, $N$ and $z$. In addition, if $\alpha$ is a positive integer, the rate for $g(z)z^\alpha$ is  $\mathcal{O}(e^{-N})$ while for  $g(z)z^\alpha\log z$ the rate enjoys \eqref{eq: rate1}.

In particular, for the case $\beta=0$, that is, $f(x)=g(x)x^\alpha$ or $f(x)=g(x)x^\alpha\log{x}$ for $x\in [0,1]$,
the constant $\mathcal{G}=\frac{\sqrt{2}+2}{\sqrt{2}-1}$ in \eqref{eq: rate1} and \eqref{eq: rate2} can be improved to $\frac{\sqrt{2}}{\sqrt{2}-1}=3.41421356237309\cdots$ uniformly for $x\in[0,1]$ as $N \rightarrow \infty$.
\end{theorem}


\begin{figure}[htbp]
\centerline{\includegraphics[width=13cm]{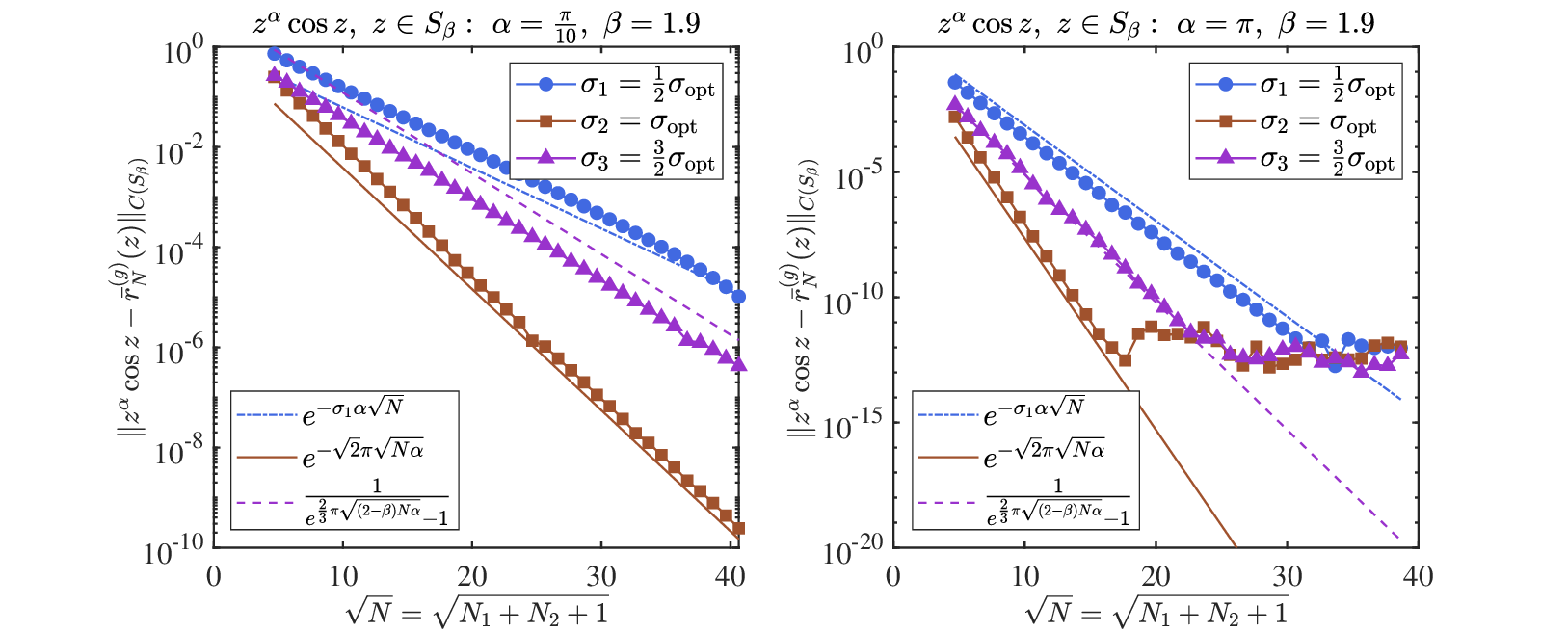}}
\centerline{\includegraphics[width=13cm]{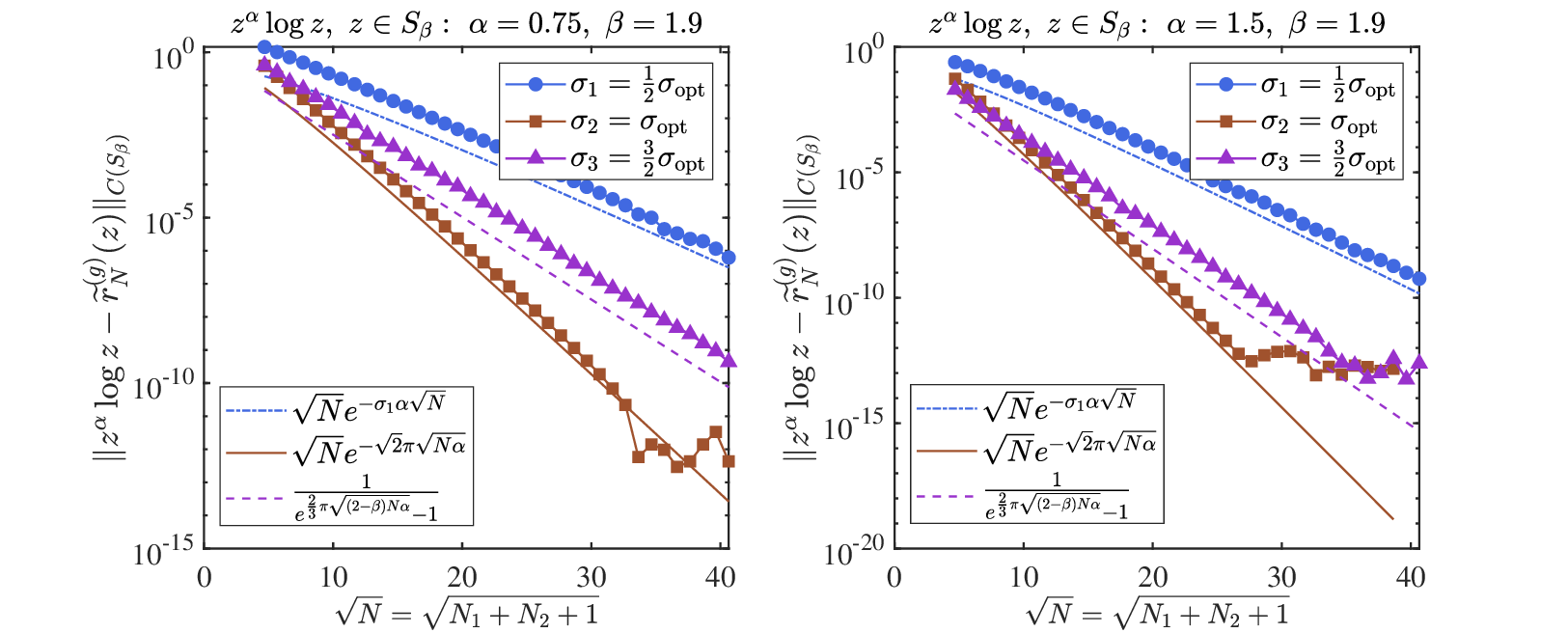}}
\caption{Decay rates of errors $\|z^{\alpha}-r_N(z)\|_{C(S_{\beta})}$ of LPs
$\bar{r}^{(g)}_{N}(z)$ for $g(z)z^\alpha$ and $\widetilde{r}^{(g)}_{N}(z)$ for $g(z)z^\alpha\log{z}$
with $g(z)=\cos{z}$ or $g(z)=1$ on $S_{\beta}$ and various values of $\alpha$ and
$\sigma_1=0.5\sigma_{\mathrm{opt}}$, $\sigma_2=\sigma_{\mathrm{opt}}$ and
$\sigma_3=1.5\sigma_{\mathrm{opt}}$, where $N=N_1+N_2+1$, $N_2={\rm ceil}(1.3\sqrt{N_1})$ and $\sigma_{\mathrm{opt}}
=\frac{\pi\sqrt{2-\beta}}{\sqrt{\alpha}}$.}\label{FIG_4LPsConvergenceRate}
\end{figure}

From Theorem \ref{mainthm}, we see that  $\bar{r}^{(g)}_N(z)$ and $\widetilde{r}^{(g)}_N(z)$  for $ g(z)z^\alpha$ and
$g(z)z^\alpha\log z$ with a fixed $\alpha>0$ achieve the fastest rates
$\mathcal{O}\big(e^{-\pi\sqrt{(2-\beta)N\alpha}}\big)$
and
$\mathcal{O}\big(\sqrt{N}e^{-\pi\sqrt{(2-\beta)N\alpha}}\big)$  with
$\sigma_{\rm opt}=\frac{\sqrt{2-\beta}\pi}{\sqrt{\alpha}}$ among all
$\sigma>0$, respectively
(see {\sc Figure} \ref{FIG_4LPsConvergenceRate}). Furthermore, for each $\sigma>0$, the rate in Theorem \ref{mainthm} is attainable. Thus, on the sector domain $S_{\beta}$ the optimal clustering parameter $\sigma$ is mainly determined by the magnitude of given $\alpha$.

In addition,
for a fixed lightning parameter $\sigma_0>0$, we can find $\alpha_0$ from $\sigma_0=\frac{\sqrt{2-\beta}\pi}{\sqrt{\alpha_0}}$, and see  from the first identity in \eqref{eq: rate1} and \eqref{eq: rate2} respectively that  the LPs enjoy a common convergence order  regardless of the increase of $\alpha(\ge\alpha_0)$
if $\alpha$ is not a positive integer:
\begin{align}
|\bar{r}^{(g)}_N(z)- g(z)z^\alpha|
=&\frac{\mathcal{G}^{\alpha}\max\{1,C^{\alpha}\}}{\varkappa(\beta)}
\left\{\begin{array}{ll}
\frac{\mathcal{O}(1)}{e^{\sigma_0\alpha\sqrt{N}}},&\alpha\le \alpha_0,\\
\frac{\mathcal{O}(1)}{e^{\pi\sqrt{(2-\beta)N\alpha_0}}-1},&\alpha> \alpha_0,
\end{array}\right.\label{eq: rate12}\\
|\widetilde{r}^{(g)}_N(z)- g(z)z^\alpha\log{z}|
=&\frac{\mathcal{G}^{\alpha}\max\{1,C^{\alpha}\}}{(\alpha+1)^{-1}\alpha\varkappa(\beta)}
\left\{\begin{array}{ll}
\frac{\mathcal{O}(1)\sqrt{N}}{e^{\sigma_0\alpha\sqrt{N}}},&\alpha\le \alpha_0,\\
\frac{\mathcal{O}(1)}{e^{\pi\sqrt{(2-\beta)N\alpha_0}}-1},&\alpha> \alpha_0,
\end{array}\right.\label{eq: rate22}
\end{align}
as $N \rightarrow \infty$, uniformly for $z\in S_{\beta}$, respectively (see {\sc Figure} \ref{FIG_4alpha0} for example), and  all the constants in the above $\mathcal{O}$ terms are independent of $\alpha$, $\sigma$, $N$ and $z$. Therefore, selecting the optimal
 $\sigma$ in \eqref{eq:uniform0}  is crucial for achieving the best convergence rate in the exploration of LPs on corner domains.

\begin{figure}[htbp]
\centerline{\includegraphics[width=12.6cm]{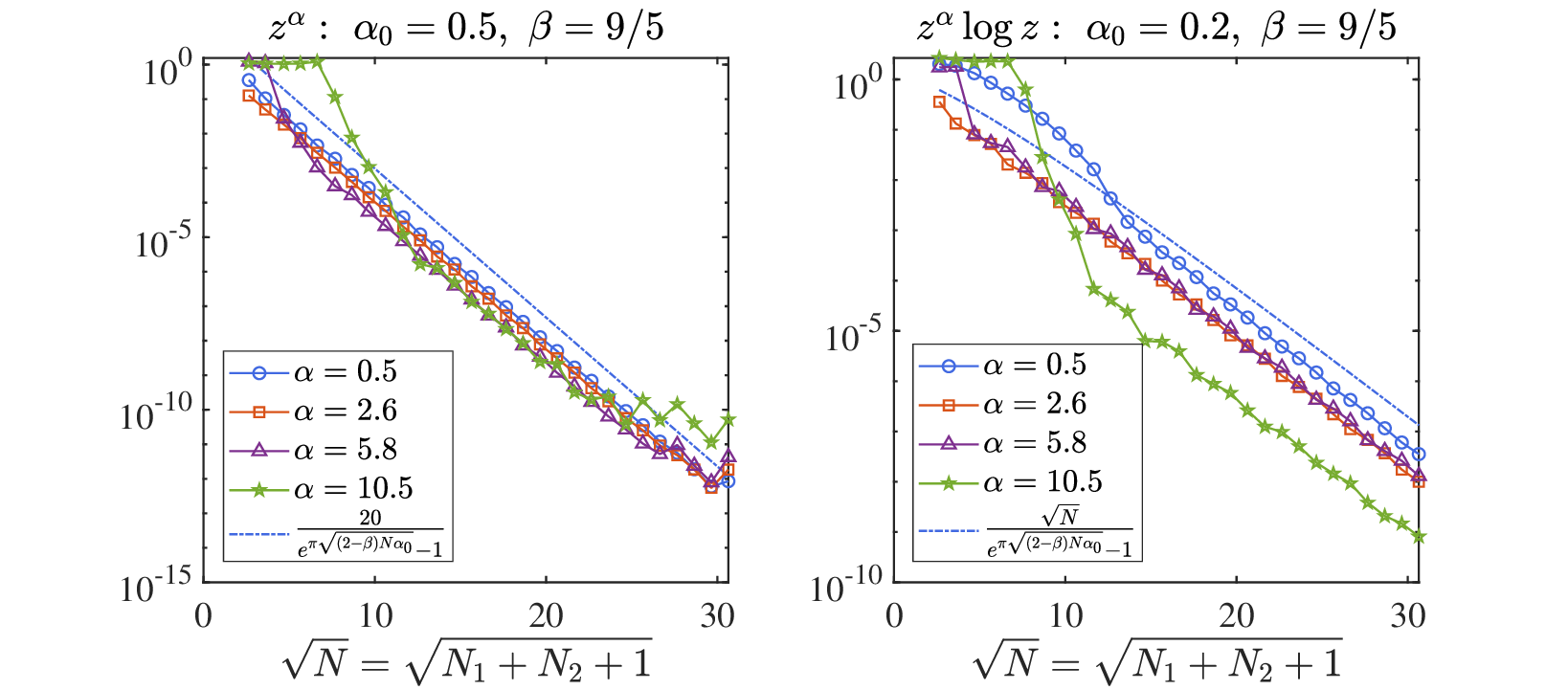}}
\caption{Decay rates of errors of LPs $\bar{r}^{(g)}_{N}(z)$ for $g(z)z^\alpha$ and $\widetilde{r}^{(g)}_{N}(z)$ for $g(z)z^\alpha\log{z}$ with $g(z)=1$ and various values of $\alpha$, equipped with the same clustering parameter $\sigma_0=\frac{\pi\sqrt{2-\beta}}{\sqrt{\alpha_0}}$,
where $N=N_1+N_2+1$, $N_2={\rm ceil}(1.3\sqrt{N_1})$.}\label{FIG_4alpha0}
\end{figure}

To fully characterize all the attainable  rates in Theorem \ref{mainthm}, we will extend  Paley-Wiener theorem \cite[Theorem 2.1, Chapter 4]{Stein} in a horizontal strip, and derive asymptotic results related to Fourier transforms  and the best exponential convergence rate $\mathcal{O}(h^{1-m_0}e^{-\frac{2a\pi}{h}})$ for the trapezoidal rule approximation with a step size of $h$ over
the whole real line (see Theorem \ref{PW2} and Corollary \ref{PW}), which, coupled with the synergistic application of fundamental complex analysis tools including Runge's approximation theorem, Cauchy's integral theorem, Poisson summation formula and the residue theorem, leads to Theorem \ref{mainthm}.

The rest of this paper is organized as follows. Section \ref{sec:22} is initially devoted to the integral representations of $z^{\alpha}$ and $z^{\alpha}\log{z}$ in the complex domain $\mathbb{C}$ slit along the negative semi-axis, with an exponential parameter $\alpha>0$. In Section \ref{sec:3}
the LP schemes are constructed and  a detailed analysis of the truncated  and approximate errors is presented.
Section \ref{sec:21} investigates the achievable upper bounds for the partial inverse of the Paley-Wiener theorem and analyzes the asymptotic decay rates of both continuous and discrete Fourier transforms. Additionally, it establishes the exact exponential convergence rate for the trapezoidal rule approximation over the entire real line.
In Section \ref{sec:4} we present a thorough analysis of the convergence rates of numerical quadratures of the integrals for $z^{\alpha}$ and $z^{\alpha}\log{z}$ by applying
 the Paley-Wiener theorem and
Poisson summation formula, which is crucial in establishing the root-exponential decay rates of LPs.
Then  Theorem \ref{mainthm} is proved in Section \ref{sec:5}.
In Section \ref{LPapp_cornerdomain} we discuss the application of obtained results to problems with corner singularities, which demonstrate the root-exponential convergence for LPs  and the best choice of parameter $\sigma$.

The numerical experiments on LPs \eqref{LP_cornerdomain} in solving Laplace equations in this paper are conducted using the {\sc Matlab} function \texttt{laplace} developed by Gopal and Trefethen in \cite{Gopal2019} and Trefethen  \cite{Treweb} by applying the best choice of the clustered parameter presented in Theorem \ref{mainthm}  for solving corner problems replacing the original $\sigma$ in \texttt{laplace}. Following the {\sc Matlab} function \texttt{laplace}, we set $N_2={\rm ceil}\left(1.3\sum_{k=1}^m\sqrt{N_{1,k}}\right)$.

\section{Integral representations of $z^\alpha$ and $z^\alpha \log z$}\label{sec:22}
A powerful approach for constructing rational functions to approximate  singular functions is the trapezoidal approximation, which originates from Stenger's study of sinc functions and related approximations \cite{GTNM2019,Stenger1981,Stenger1986,Stenger1993,TNWNM2021}.

To establish the convergence rates for the LPs in Theorem \ref{mainthm}, the starting point is the integral representation of $z^\alpha$ and  $z^\alpha\log z$.
According to \cite[p. 319, (3.222)]{GR2014}, $x^\alpha$ on $[0,1]$ can be represented by
 \begin{align*}
x^\alpha&=\frac{\sin(\alpha\pi)}{\alpha\pi}\int_0^{+\infty} \frac{x}{y^{\frac{1}{\alpha}}+x}\mathrm{d}y, \quad 0<\alpha<1.
\end{align*}
We extend this integral representation to the complex plane for all $\alpha$
and provide an analogous extension for $z^{\alpha}\log z$.

\begin{theorem}\label{complex_int_repre}
Let $\alpha>0$ and $\ell\ge\lfloor\alpha\rfloor$ where $\lfloor\alpha\rfloor$ denotes the largest integer not larger than $\alpha$.
Suppose that $s_1,\ldots,s_\ell$ are $\ell$ distinct numbers located outside $(-\infty,0]$.
Then it holds for all $z\in\mathbb{C}\setminus(-\infty,0)$ that
\begin{align}
z^{\alpha}=&\frac{\sin{(\alpha\pi)}}{(-1)^{\ell}\pi}
\int_0^{+\infty}\frac{zy^{\alpha-1}}{y+z}\left(\prod\limits_{k=1}^{\ell}\frac{z-s_k}{y+s_k}\right)\mathrm{d}y
+z\mathcal{L}[z^{\alpha-1};s_1,\ldots,s_\ell],\label{eq:cint_gener}\\
z^{\alpha}\log{z} =& \frac{\sin(\alpha\pi)}{(-1)^\ell\pi}\int_{0}^{+\infty}
\frac{zy^{\alpha-1}\log{y}}{y+z}\left(\prod_{k=1}^{\ell}\frac{z-s_k}{y+s_k}\right)\mathrm{d}y\notag\\
&+\frac{\cos(\alpha\pi)}{(-1)^{\ell}}\int_0^{+\infty}\frac{zy^{\alpha-1}}{y+z}
\left(\prod\limits_{k=1}^{\ell}\frac{z-s_k}{y+s_k}\right)\mathrm{d}y
+z\mathcal{L}[z^{\alpha-1}\log{z};s_1,\ldots,s_\ell],\label{eq:cint_log_gener}
\end{align}
where $\mathcal{L}[X(z);s_1,\ldots,s_\ell]$ denotes the Lagrange interpolating polynomial at $s_1,\ldots,s_\ell$ for $X(z)=z^{\alpha-1}$ and
$z^{\alpha-1}\log{z}$, respectively.
Especially, it holds for $0<\alpha<1$ that
\begin{align}
z^{\alpha}=&\frac{\sin{(\alpha\pi)}}{\pi}
\int_0^{+\infty}\frac{zy^{\alpha-1}}{y+z}\mathrm{d}y,\label{eq:cint}\\
z^{\alpha}\log{z}=&\frac{\sin{(\alpha\pi)}}{\pi}
\int_{0}^{+\infty}\frac{zy^{\alpha-1}\log{y}}{y+z}\mathrm{d}y
+\cos{(\alpha\pi)}\int_{0}^{+\infty}\frac{zy^{\alpha-1}\mathrm{d}y}{y+z}.\label{eq:cintlog}
\end{align}
\end{theorem}

 \begin{proof}
Consider the integral $\int_0^{+\infty}K(y,z)\mathrm{d}y$ with $
K(y,z)=\frac{y^{\alpha-1}}{(y+z)\prod_{k=1}^{\ell}{(y+s_k)}}$
for $z\in\mathbb{C}\setminus(-\infty,0)$ and $z\not\in\{ 0,s_1,\ldots,s_\ell\}$ and $0<\alpha\notin\mathbb{N}_+$.
With the aid of Cauchy's residue theorem, we have an integral along a closed Jordan contour $\mathfrak{S}:\,\epsilon\rightarrow R\rightarrow\gamma_R\rightarrow R\rightarrow\epsilon\rightarrow\gamma^{-}_\epsilon$ (see {\sc Figure} \ref{path_for_representation}) in the complex plane split by the positive real line, which reads as
\begin{align}\label{general integral repres111}
\int_{\mathfrak{S}}K(y,z)\mathrm{d}y
=&\bigg\{\int_{\epsilon}^{R}+\int_{\gamma_R}
+e^{2i\alpha\pi}\int_{R}^{\epsilon}+\int_{\gamma_\epsilon^{-}}\bigg\}
K(y,z)\mathrm{d}y\notag\\
=&2i\pi{\rm Res}\left[K(y,z),-z\right]+2i\pi\sum\limits_{l=1}^{\ell}{\rm Res}\left[K(y,z),-s_l\right].
\end{align}
We used in \eqref{general integral repres111} the fact
$\log{y}|_{y\in[R\rightarrow\epsilon]}
=\log{y}|_{y\in[\epsilon\rightarrow R]}+2i\pi$, which implies that
$$y^{\alpha-1}|_{y\in[R\rightarrow\epsilon]}
=e^{(\alpha-1)\log{y}}|_{y\in[R\rightarrow\epsilon]}
=e^{2i\alpha\pi}y^{\alpha-1}|_{y\in[\epsilon\rightarrow R]}.$$
Here the radii $R$ and $\epsilon$ of $\gamma_R$ and $\gamma_\epsilon$ are chosen to be sufficiently large and small, respectively, such that $0<\epsilon<1<R$ and $-z,-s_1,\ldots,-s_\ell$ locate inside the domain included by $\mathfrak{S}$.

\begin{figure}[htbp]\vspace{-.5cm}
\centerline{\includegraphics[width=5cm]{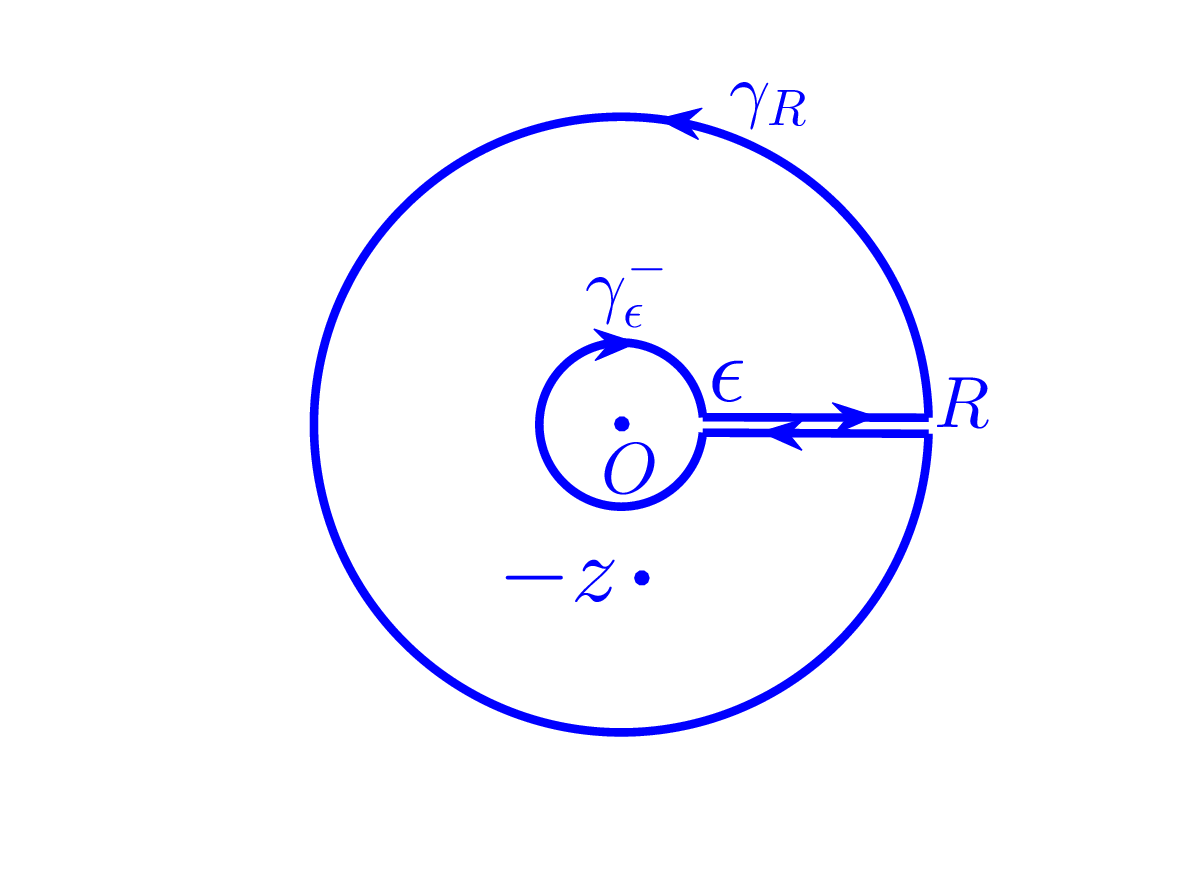}}\vspace{-.5cm}
\caption{The integral contour $\mathfrak{S}$ of \eqref{general integral repres111}.}
\label{path_for_representation}
\end{figure}

Letting  $R$ tend to $+\infty$ and $\epsilon$ to $0$ in \eqref{general integral repres111}, we have
\begin{align*}
\frac{1-e^{2i\alpha\pi}}{2i\pi}\int_0^{+\infty}\frac{y^{\alpha-1}\mathrm{d}y}{(y+z)\prod\limits_{k=1}^{\ell}{(y+s_k)}}
=&{\rm Res}\left[K(y,z),-z\right]+\sum\limits_{l=1}^{\ell}{\rm Res}\left[K(y,z),-s_l\right]
\end{align*}
since
\begin{align*}
\bigg|\int_{\gamma_R}K(y,z)\mathrm{d}y\bigg|
\le&\int_{\gamma_R}\frac{|e^{(\alpha-1)\log{y}}|{\rm d}s}{\left(|y|-|z|\right)\prod_{k=1}^{\ell}{(|y|-|s_k|)}}\\
=&\frac{R^{\alpha-1}}{\left(R-|z|\right)\prod_{k=1}^{\ell}{(R-|s_k|)}}2\pi R
\end{align*}
approaches to $0$ as $R\rightarrow+\infty$, and
\begin{align*}
\bigg|\int_{\gamma_\varepsilon}K(y,z)\mathrm{d}y\bigg|
\le&\frac{\epsilon^{\alpha-1}}{\left(|z|-\epsilon\right)\prod_{k=1}^{\ell}{(|s_k|-\epsilon)}}2\pi\epsilon
\end{align*}
tends to $0$ as $\epsilon\rightarrow0$.

By substituting the residues\footnote{We used here the fact that $(-z)^{\alpha-1}=e^{(\alpha-1)\left[\log{z}+\log(-1)\right]}
=e^{(\alpha-1)(\log{z}+i\pi)}=-e^{i\alpha\pi}z^{\alpha-1}$, since for $\zeta\ne0$ in the complex plane slit along the positive real semi-axis the principal argument angle $\arg{\zeta}\in[0,2\pi)$, and then $\arg(-1)=\pi$.
The analogous argument is also valid for the residues at $s_l,\ l=1,\ldots,\ell$.}
\begin{align*}
{\rm Res}\left[K(y,z),-z\right]
=\frac{(-1)^{\ell+1}e^{i\alpha\pi}z^{\alpha-1}}{\prod_{k=1}^{\ell}(z-s_k)}
\end{align*}
and
\begin{align*}
{\rm Res}\left[K(y,z),-s_l\right]
=-\frac{(-1)^{\ell+1}e^{i\alpha\pi}s_l^{\alpha-1}}{(z-s_l)\prod_{k=1,k\ne l}^{\ell}(s_l-s_k)}
\end{align*}
into \eqref{general integral repres111}, we get
\begin{align*}
\frac{\sin{(\alpha\pi)}}{\pi}\int_0^{+\infty}K(y,z)\mathrm{d}y
=\frac{(-1)^{\ell}z^{\alpha-1}}{\prod_{k=1}^{\ell}(z-s_k)}
-\sum_{l=1}^{\ell}\frac{(-1)^{\ell}s_l^{\alpha-1}}{(z-s_l)\prod_{k=1,k\ne l}^{\ell}(s_l-s_k)}
\end{align*}
which is equivalent to
\begin{align*}
z^{\alpha-1}=\frac{\sin{(\alpha\pi)}}{(-1)^\ell\pi}
\int_0^{+\infty}\frac{y^{\alpha-1}\prod_{k=1}^{\ell}(z-s_k)}{(y+z)\prod_{k=1}^{\ell}(y+s_k)}\mathrm{d}y
+\sum_{l=1}^{\ell}s_l^{\alpha-1}\prod\limits_{k=1,k\ne l}^{\ell}\frac{z-s_k}{s_l-s_k},
\end{align*}
then we arrive at the conclusion \eqref{eq:cint_gener} for $z\in\mathbb{C}\setminus(-\infty,0)$ and $z\ne 0, s_1,\ldots,s_\ell$
and $0<\alpha\notin\mathbb{N}_+$. Additionally, \eqref{eq:cint_gener} holds obviously for $\alpha\in\mathbb{N}_+$.

Analogously, considering the integrand of
\begin{align*}
  K_{\log}(y,z)=\frac{y^{\alpha-1}\log{y}}{(y+z)\prod_{k=1}^{\ell}(y+s_k)},\ \alpha>0,\ z\in\mathbb{C}\setminus(-\infty,0)
\end{align*}
along the closed Jordan contour $\mathfrak{S}$, 
we see that for $z\ne 0,s_1,\ldots,s_{\ell}$,
\allowdisplaybreaks
\begin{align*}
&\frac{\sin(\alpha\pi)}{\pi}\int_{0}^{+\infty}K_{\log}(y,z)\mathrm{d}y
+e^{i\alpha\pi}\int_{0}^{+\infty}K(y,z)\mathrm{d}y\\
=&\frac{(-1)^{\ell}z^{\alpha-1}(i\pi+\log{z})}{\prod_{k=1}^{\ell}(z-s_k)}
-\sum_{l=1}^{\ell}\frac{(-1)^{\ell}s_l^{\alpha-1}(i\pi+\log{s_l})}{(z-s_l)\prod_{k=1,k\ne l}^{\ell}(s_l-s_k)}\\
=&\frac{(-1)^{\ell}z^{\alpha-1}\log{z}}{\prod_{k=1}^{\ell}(z-s_k)}
-\sum_{l=1}^{\ell}\frac{(-1)^{\ell}s_l^{\alpha-1}\log{s_l}}{(z-s_l)\prod_{k=1,k\ne l}^{\ell}(s_l-s_k)}\\
&+\frac{(-1)^{\ell}i\pi z^{\alpha-1}}{\prod_{k=1}^{\ell}(z-s_k)}
-\sum_{l=1}^{\ell}\frac{(-1)^{\ell}i\pi s_l^{\alpha-1}}{(z-s_l)\prod_{k=1,k\ne l}^{\ell}(s_l-s_k)},
\end{align*}
that is,
\begin{multline*}
\frac{\sin(\alpha\pi)}{(-1)^\ell\pi}\int_{0}^{+\infty}
\frac{y^{\alpha-1}\log{y}}{y+z}\left(\prod_{k=1}^{\ell}\frac{z-s_k}{y+s_k}\right)\mathrm{d}y
+\frac{e^{i\alpha\pi}}{(-1)^{\ell}}\int_0^{+\infty}\frac{y^{\alpha-1}}{y+z}\left(\prod\limits_{k=1}^{\ell}\frac{z-s_k}{y+s_k}\right)\mathrm{d}y\\
=z^{\alpha-1}\log{z}-\mathcal{L}[z^{\alpha-1}\log{z};s_1,\ldots,s_\ell]
+i\pi\left(z^{\alpha-1}-\mathcal{L}[z^{\alpha-1};s_1,\ldots,s_\ell]\right).
\end{multline*}
Thus we establish by \eqref{eq:cint_gener} that
\begin{align*}
  z^{\alpha}\log{z} =& \frac{\sin(\alpha\pi)}{(-1)^\ell\pi}\int_{0}^{+\infty}
\frac{zy^{\alpha-1}\log{y}}{y+z}\left(\prod_{k=1}^{\ell}\frac{z-s_k}{y+s_k}\right)\mathrm{d}y\\
&+\frac{\cos(\alpha\pi)}{(-1)^{\ell}}\int_0^{+\infty}\frac{zy^{\alpha-1}}{y+z}
\left(\prod\limits_{k=1}^{\ell}\frac{z-s_k}{y+s_k}\right)\mathrm{d}y
+z\mathcal{L}[z^{\alpha-1}\log{z};s_1,\ldots,s_\ell]
\end{align*}
due to that
\begin{align*}
i\pi\left(z^{\alpha}-z\mathcal{L}[z^{\alpha-1};s_1,\ldots,s_\ell]\right)
=\frac{i\sin(\alpha\pi)}{(-1)^{\ell}}
\int_0^{+\infty}\frac{y^{\alpha-1}}{y+z}
\left(\prod\limits_{k=1}^{\ell}\frac{z-s_k}{y+s_k}\right)\mathrm{d}y.
\end{align*}
Thus we arrive at \eqref{eq:cint_log_gener} for $z\in\mathbb{C}\setminus(-\infty,0]$ with $z\ne 0, s_1,\ldots,s_\ell$.

It is clear that \eqref{eq:cint_gener} and \eqref{eq:cint_log_gener} hold for $z=0$. In addition, for $z\in \{s_1,\ldots,s_\ell\}$,  \eqref{eq:cint_gener} and \eqref{eq:cint_log_gener} are also satisfied due to $s_k^\alpha=s_k\mathcal{L}[s_k^{\alpha-1};s_1,\ldots,s_\ell]$ and $s_k^\alpha\log s_k=s_k\mathcal{L}[s_k^{\alpha-1}\log s_k;s_1,\ldots,s_\ell]$.

By setting $\ell=0$, we directly obtain  \eqref{eq:cint} and \eqref{eq:cintlog}.
\end{proof}

\section{Principles of LPs \eqref{eq:rat} for $z^{\alpha}$ and  $z^{\alpha}\log z$}\label{sec:3}
Using the integral representations \eqref{eq:cint_gener} and \eqref{eq:cint_log_gener}, along with a rigorous analysis of truncated errors, this section develops LPs
for $z^{\alpha}$ and $z^{\alpha}\log{z}$ ($\alpha>0$). To achieve the sharp estimates on the convergence rates in Theorem \ref{mainthm}, we set $s_k$ as the roots of the shifted Chebyshev polynomial of the first kind $\mathrm{T}_{\ell}(2s-2\delta-1)$, i.e., $s_k=\delta+\frac{1}{2}\left(1+\cos\frac{(2\ell -2k+1)\pi}{2\ell}\right)$ $\in [\delta,\delta+1]$ for some $\delta> 0$ following \cite{LWL,Trefethen2011,Trefethen2013,wangSINUM,X2016}.

\subsection{Exponential transformation}\label{sec3-1}
By applying the exponential transformation $y=Ce^{\frac{1}{\alpha}t}$, from \eqref{eq:cint_gener} and \eqref{eq:cint_log_gener} it follows for $z\in S_{\beta}$  that
\begin{align}
z^\alpha=&\frac{\sin(\alpha\pi)}{(-1)^{\ell}\alpha\pi}
\int_{-\infty}^{+\infty}
\frac{zC^{\alpha}e^t}{Ce^{\frac{1}{\alpha}t}+z}
\left(\prod\limits_{k=1}^{\ell}\frac{z-s_k}{Ce^{\frac{1}{\alpha}t}+s_k}\right)\mathrm{d}t
+z\mathcal{L}[z^{\alpha-1};s_1,\ldots,s_\ell],\label{eq:zalpha}\\
  z^{\alpha}\log{z} =& \frac{\sin(\alpha\pi)}{(-1)^\ell\alpha^2\pi}\int_{-\infty}^{+\infty}
\frac{zC^{\alpha}te^{t}}{Ce^{\frac{1}{\alpha}t}+z}
\left(\prod_{k=1}^{\ell}\frac{z-s_k}{Ce^{\frac{1}{\alpha}t}+s_k}\right)\mathrm{d}t\notag\\
&+\left[\frac{\sin(\alpha\pi)\log{C}}{(-1)^\ell\alpha\pi}+\frac{\cos(\alpha\pi)}{(-1)^{\ell}\alpha}\right]
\int_{-\infty}^{+\infty}\frac{zC^{\alpha}e^{t}}{Ce^{\frac{1}{\alpha}t}+z}
\left(\prod_{k=1}^{\ell}\frac{z-s_k}{Ce^{\frac{1}{\alpha}t}+s_k}\right)\mathrm{d}t\label{eq:cint_log_gener_substituting}\\
&+z\mathcal{L}[z^{\alpha-1}\log{z};s_1,\ldots,s_\ell].\notag
\end{align}

To ensure the uniform approximation of LPs \eqref{eq:rat} on $S_\beta$, we set $\ell=\lfloor\alpha\rfloor$ 
and $\kappa=\frac{\alpha}{\ell+1-\alpha}$  for $z^{\alpha}$ such that
$$
\bigg|\frac{\sin(\alpha\pi)}{(-1)^{\ell}\alpha\pi}\kappa\bigg|=\frac{|\sin((\ell+1-\alpha)\pi)|}{(\ell+1-\alpha)\pi}
$$
is bounded as $\alpha$ tends to a nonnegative integer, then Theorem \ref{mainthm} on $g(z)z^\alpha$
is uniformly satisfied for all $\alpha>0$, respectively.

 To obtain the uniformity of LP \eqref{eq:rat} for $z^{\alpha}\log{z}$ on $\alpha$, we set $\ell=\lceil\alpha\rceil$ (i.e., the smallest  integer greater than or equal to $\alpha$) and $\kappa=\frac{\alpha}{\ell+1-\alpha}$ such that $\kappa$ is bounded as $\alpha$ tends to a positive integer.

Consequently, using  \eqref{eq:zalpha} and \eqref{eq:cint_log_gener_substituting}, the LPs for  $z^{\alpha}$ and $z^{\alpha}\log{z}$ are obtained by discretizing truncations over a finite interval of improper integrals
\begin{align*}
\int_{-\infty}^{+\infty}
\frac{zC^{\alpha}e^{t}}{Ce^{\frac{1}{\alpha}t}+z}
\left(\prod_{k=1}^{\ell}\frac{z-s_k}{Ce^{\frac{1}{\alpha}t}+s_k}\right)\mathrm{d}t\quad\mbox{and}\quad
\int_{-\infty}^{+\infty}
\frac{zC^{\alpha}te^{t}}{Ce^{\frac{1}{\alpha}t}+z}
\left(\prod_{k=1}^{\ell}\frac{z-s_k}{Ce^{\frac{1}{\alpha}t}+s_k}\right)\mathrm{d}t.
\end{align*}
Subsequently, crucial effort  will be devoted to analyzing the truncation and quadrature errors of these integrals.

\subsection{Truncation errors}\label{sec:truncation_error}\label{sec3-2}
It is worthy of noting that for $z=|z|e^{\pm i\frac{\phi}{2}}$ and $t\in \mathbb{R}$ it holds
\begin{align*}
\big{|}Ce^{\frac{1}{\alpha}t}+z\big{|}=\left\{\begin{array}{ll}
\sqrt{C^2e^{\frac{2}{\alpha}t}+2C|z|e^{\frac{1}{\alpha}t}\cos\frac{\phi}{2} +|z|^2}\ge |z|,&0\le \phi\le \pi,\\
\sqrt{\left(Ce^{\frac{1}{\alpha}t}+|z|\cos\frac{\phi}{2}\right)^2+|z|^2\sin^2\frac{\phi}{2}}\ge |z|\sin\frac{\phi}{2},&\pi<\phi<2\pi,
\end{array}\right.
\end{align*}
then, for $z=xe^{\pm i\frac{\theta}{2}\pi}\in S_{\beta}$ and $0\le \theta\le \beta<2$ it follows that
\begin{align}\label{eq:est1}
\big{|}Ce^{\frac{1}{\alpha}t}+z\big{|}
\ge x\varkappa(\theta),\quad \varkappa(\theta)=\left\{
\begin{array}{ll}
1,&0\le \theta\le 1\\
\sin{\frac{\theta\pi}{2}}
\ge\sin\frac{\beta\pi}{2},&1< \theta <2
\end{array}\right.
\end{align}
with $\varkappa(\theta)\ge\varkappa(\beta)$, which, together with 
\allowdisplaybreaks
\begin{align*}
\bigg|\prod_{k=1}^{\ell}\frac{z-s_k}{Ce^{\frac{1}{\alpha}t}+s_k}\bigg|\le  \frac{\Big\|\prod\limits_{k=1}^{\ell}(z-s_k)\Big\|_{C(S_\beta)}}{\prod_{k=1}^{\ell}s_k}
=\frac{\mathbb{T}_{\ell,\beta}}
{\prod_{k=1}^{\ell}\left[\delta+\frac{1}{2}\left(1+\cos\frac{(2k-1)\pi}{2\ell}\right)\right]}
\le\frac{\mathbb{T}_{\ell,\beta}}{\delta^{\ell}},
\end{align*}
implies
\begin{align}\label{ine:minus_u}
\left|\frac{zC^{\alpha}|t|^l e^t}{Ce^{\frac{1}{\alpha}t}+z}
\left(\prod\limits_{k=1}^{\ell}\frac{z-s_k}{Ce^{\frac{1}{\alpha}t}+s_k}\right)\right|
\le& \frac{x|t|^lC^{\alpha}e^t} {x\varkappa(\beta)}\cdot\frac{\mathbb{T}_{\ell,\beta}}{\delta^{\ell}}
\le\frac{\mathbb{T}_{\ell,\beta}|t|^l C^{\alpha}e^t}{\delta^{\ell}\varkappa(\beta)}
\end{align}
for $l=0,1$, where $\mathbb{T}_{\ell,\beta}=\Big\|\prod\limits_{k=1}^{\ell}|z-s_k|\Big\|_{C(S_\beta)}=\max\limits_{z\in S_{\beta}}\prod\limits_{k=1}^{\ell}|z-s_k|$. Consequently we have
\begin{align}\label{eq:inequ_neg}
\left|\int_{-\infty}^{-T}
\frac{zC^{\alpha}t^l e^t}{Ce^{\frac{1}{\alpha}t}+z}
\left(\prod\limits_{k=1}^{\ell}\frac{z-s_k}{Ce^{\frac{1}{\alpha}t}+s_k}\right)\mathrm{d}t\right|
\le\frac{\mathbb{T}_{\ell,\beta}}{\delta^{\ell}}\frac{C^{\alpha}(1+T)^le^{-T}}{\varkappa(\beta)}.
\end{align}

While by
\begin{align}\label{eq:est2}
\big{|}Ce^{\frac{1}{\alpha}t}+z\big{|}
=&\left\{\begin{array}{ll}
\sqrt{C^2e^{\frac{2}{\alpha}t}+2Cxe^{\frac{1}{\alpha}t}\cos\frac{\theta\pi}{2} +x^2},&0\le \theta\le 1\\
\sqrt{\left(Ce^{\frac{1}{\alpha}t}\cos\frac{\theta\pi}{2}+x\right)^2+C^2e^{\frac{2}{\alpha}t}\sin^2\frac{\theta\pi}{2}},&1< \theta\le \beta< 2
\end{array}\right.\\
\ge&Ce^{\frac{1}{\alpha}t}\varkappa(\beta),\quad z\in S_\beta,\notag
\end{align}
we get for $t\ge 0$ that
\begin{align}\label{ieq:positive_u}
\left|\frac{zt^lC^{\alpha}e^t}{Ce^{\frac{1}{\alpha}t}+z}
\left(\prod\limits_{k=1}^{\ell}\frac{z-s_k}{Ce^{\frac{1}{\alpha}t}+s_k}\right)\right|
\le \frac{xt^lC^{\alpha}e^t}{Ce^{\frac{1}{\alpha}t}\varkappa(\beta)}\cdot\frac{\mathbb{T}_{\ell,\beta}}{C^{\ell}e^{\frac{\ell}{\alpha}t}}
\le\frac{\mathbb{T}_{\ell,\beta}t^le^{-\frac{1}{\kappa}t}}{C^{\ell+1-\alpha}\varkappa(\beta)}
\end{align}
and then
\begin{align}\label{eq:inequ_pos}
\left|\int_{\kappa T}^{+\infty}
\frac{zC^{\alpha}t^l e^t}{Ce^{\frac{1}{\alpha}t}+z}
\left(\prod\limits_{k=1}^{\ell}\frac{z-s_k}{Ce^{\frac{1}{\alpha}t}+s_k}\right)\mathrm{d}t\right|
\le\frac{\mathbb{T}_{\ell,\beta}\kappa(\kappa+\kappa T)^le^{-T}}{C^{\ell+1-\alpha}\varkappa(\beta)}.
\end{align}
Thus, together with \eqref{eq:inequ_neg} and \eqref{eq:inequ_pos}, it derives
\begin{align}\label{eq:truncerrors}
&\int_{-\infty}^{+\infty}
\frac{zC^{\alpha}t^l e^{t}}{Ce^{\frac{1}{\alpha}t}+z}
\left(\prod_{k=1}^{\ell}\frac{z-s_k}{Ce^{\frac{1}{\alpha}t}+s_k}\right)\mathrm{d}t\notag\\
=&\left\{\int_{-\infty}^{-T}+\int_{-T}^{\kappa T}+\int_{\kappa T}^{+\infty}\right\}
\frac{zC^{\alpha}t^l e^t}{Ce^{\frac{1}{\alpha}t}+z}
\left(\prod\limits_{k=1}^{\ell}\frac{z-s_k}{Ce^{\frac{1}{\alpha}t}+s_k}\right)\mathrm{d}t\\
=&\int_{-T}^{\kappa T}
\frac{zC^{\alpha}t^l e^{t}}{Ce^{\frac{1}{\alpha}t}+z}
\left(\prod_{k=1}^{\ell}\frac{z-s_k}{Ce^{\frac{1}{\alpha}t}+s_k}\right)\mathrm{d}t+\widehat{E}^{(l)}_T(z)\notag
\end{align}
where
\begin{align}\label{eq:boundforwidehatE}
\left|\widehat{E}^{(l)}_T(z)\right|\le
\frac{(1+T)^l\mathbb{T}_{\ell,\beta}C^{\alpha}e^{-T}}{\varkappa(\beta)}
\left(\frac{1}{\delta^{\ell}}+\frac{\kappa^{l+1}}{C^{\ell+1}}\right),\quad l=0,1.
\end{align}

\subsection{Construction of the rational functions for $z^\alpha$ and $z^\alpha\log z$}\label{sec3-3}
Discretization using the rectangular rule
 in $N_t+1$ quadrature points  with step length $h=\frac{\sigma\alpha}{\sqrt{N_1}}$:
\begin{align*}
T=N_1h=\sigma\alpha\sqrt{N_1},\quad \mathcal{N}_th =(\kappa+1)T\, ({\rm i.e.},\, \mathcal{N}_t=(\kappa+1)N_1),\quad N_t={\rm ceil}(\mathcal{N}_t),
 \end{align*}
gives rise to the following rational approximations by \eqref{eq:truncerrors} 
\allowdisplaybreaks
\begin{align}\label{eq:intC1}
&\int_{-\infty}^{+\infty}
\frac{zC^{\alpha}t^l e^{t}}{Ce^{\frac{1}{\alpha}t}+z}
\left(\prod_{k=1}^{\ell}\frac{z-s_k}{Ce^{\frac{1}{\alpha}t}+s_k}\right)\mathrm{d}t\notag\\
=&\int_{-T}^{\kappa T}\frac{zC^{\alpha}t^le^t}{Ce^{\frac{1}{\alpha}t}+z}
\left(\prod\limits_{k=1}^{\ell}\frac{z-s_k}{Ce^{\frac{1}{\alpha}t}+s_k}\right)\mathrm{d}t +\widehat{E}^{(l)}_T(z)\\
=&\int_0^{(\kappa+1)T}
\frac{z C^{\alpha}(u-T)^le^{u-T}}{Ce^{\frac{1}{\alpha}(u-T)}+z}
 \left(\prod\limits_{k=1}^{\ell}\frac{z-s_k}{Ce^{\frac{1}{\alpha}(u-T)}+s_k}\right)\mathrm{d}u +\widehat{E}^{(l)}_T(z)
\notag\\
=&\int_0^{N_th}
 \frac{z C^{\alpha}(u-T)^le^{u-T}}{Ce^{\frac{1}{\alpha}(u-T)}+z}
 \left(\prod\limits_{k=1}^{\ell}\frac{z-s_k}{Ce^{\frac{1}{\alpha}(u-T)}+s_k}\right)\mathrm{d}u
 +E^{(l)}_T(z)\notag\\
=:& r^{(l)}_{N_t}(z) +E_Q^{(l)}(z)+E^{(l)}_T(z),\notag
\end{align}
where
 $r^{(l)}_{N_t}(z)$ is the rational approximation defined by
\begin{align}\label{eq:ECrat1CC}
 r^{(l)}_{N_t}(z)=&h\sum_{j=0}^{N_t}
\frac{zC^{\alpha}(jh-T)^le^{jh-T}}{Ce^{\frac{1}{\alpha}(jh-T)}+z}
\left(\prod\limits_{k=1}^{\ell}\frac{z-s_k}{Ce^{\frac{1}{\alpha}(jh-T)}+s_k}\right)
\end{align}
derived from the rectangular rule, and $E_Q^{(l)}(z)$ represents the quadrature error
\begin{align*}
E_Q^{(l)}(z)=\int_0^{N_th}
\frac{z C^{\alpha}(u-T)^le^{u-T}}{Ce^{\frac{1}{\alpha}(u-T)}+z}
 \left(\prod\limits_{k=1}^{\ell}\frac{z-s_k}{Ce^{\frac{1}{\alpha}(u-T)}+s_k}\right)\mathrm{d}u
- r^{(l)}_{N_t}(z),
\end{align*}
and
\begin{align}\label{boundforEl(z)}
|E^{(l)}_T(z)|=&\left|\widehat{E}^{(l)}_T(z)-\int_{(\kappa+1)T}^{N_th}
\frac{zC^{\alpha}(u-T)^le^{u-T}}{e^{\frac{1}{\alpha}(u-T)}+z}
\left(\prod\limits_{k=1}^{\ell}\frac{z-s_k}{Ce^{\frac{1}{\alpha}(u-T)}+s_k}\right)du\right|\notag\\
\le&\left|\widehat{E}^{(l)}_T(z)\right|+\int_{(\kappa+1)T}^{+\infty}\left|
\frac{zC^{\alpha}(u-T)^le^{u-T}}{e^{\frac{1}{\alpha}(u-T)}+z}
\left(\prod\limits_{k=1}^{\ell}\frac{z-s_k}{Ce^{\frac{1}{\alpha}(u-T)}+s_k}\right)\right|\mathrm{d}u\\
=&|\widehat{E}^{(l)}_T(z)|+\int_{\kappa T}^{+\infty}
\left|\frac{zC^{\alpha}t^le^{t}}{e^{\frac{1}{\alpha}t}+z}
\left(\prod\limits_{k=1}^{\ell}\frac{z-s_k}{Ce^{\frac{1}{\alpha}t}+s_k}\right)\right|\mathrm{d}t\notag\\
\le&\frac{2(1+T)^{l}\mathbb{T}_{\ell,\beta}C^{\alpha}}{\varkappa(\beta)e^{T}}
\left(\frac{1}{\delta^{\ell}}+\frac{\kappa^{l+1}}{C^{\ell+1}}\right)\notag
\end{align}
by \eqref{eq:inequ_pos} and \eqref{eq:boundforwidehatE}.

In particular, the rational function $r^{(l)}_{N_t}(z)$ \eqref{eq:ECrat1CC} can be rewritten by using the exponential clustered poles $p_j=-Ce^{\frac{1}{\alpha}(jh-T)}
=-Ce^{-\frac{\sigma(N_1-j)}{\sqrt{N_1}}}$ ($0\le j\le N_t$) as
\begin{align}\label{eq:ECrat1C}
r^{(l)}_{N_t}(z)=&\sum_{j=0}^{N_t}\frac{z|p_j|^{\alpha}h(jh-T)^l}{z-p_j}
\left(\prod\limits_{k=1}^{\ell}\frac{z-s_k}{s_k-p_j}\right)\notag\\
=&\sum_{j=0}^{N_1}\left(
\frac{p_j|p_j|^{\alpha}h(jh-T)^l}{z-p_j}+|p_j|^{\alpha}h(jh-T)^l\right)
\left(\prod\limits_{k=1}^{\ell}\frac{z-s_k}{s_k-p_j}
\right)\\
&+\sum_{j=N_1+1}^{N_t}
\frac{z|p_j|^{\alpha}h(jh-T)^l}{z-p_j}
\left(\prod\limits_{k=1}^{\ell}\frac{z-s_k}{s_k-p_j}
\right)\notag\\
=&\sum_{j=0}^{N_1}\frac{a^{(l)}_j}{z-p_j}
+\sum_{j=N_1+1}^{N_t}
\frac{z|p_j|^{\alpha}h(jh-T)^l}{z-p_j}
\left(\prod\limits_{k=1}^{\ell}\frac{z-s_k}{s_k-p_j}
\right)+P^{(l)}_\ell(z)\notag\\
=&:r^{(l)}_{N_1}(z)+r^{(l)}_2(z)+P^{(l)}_\ell(z)\notag
 \end{align}
where $a^{(l)}_j=(-1)^{\ell}hp_j|p_j|^\alpha(jh-T)^l$ ($0\le j\le N_1$) evaluated
by Cauchy's residue theorem, and $P^{(l)}_\ell(z)$ is a polynomial of degree $\ell$.

 \subsection{Runge's approximation theorem}\label{sec3-4}
 Subsequently, we will demonstrate that $r^{(l)}_2(z)$ in \eqref{eq:ECrat1C} can be efficiently approximated  with an exponential convergence rate by a polynomial $P^{(l)}_{N_2}(z)$ of degree $N_2=\mathcal{O}(\sqrt{N_1})$ from the proof of Runge's approximation theorem \cite[pp. 76-77]{Gaier1987} and \cite[pp. 8-9]{Walsh1965}.

\begin{theorem}\label{Runge}\cite[1885, Runge]{Gaier1987}
Suppose $K\subset \mathbb{C}$ is compacted, $K^C=\mathbb{C}\setminus K$ is connected, and $f$ is analytic on $K$. Then there exist polynomials $\left\{P_n\right\}_{n=1}^{\infty}$ such that
\begin{align*}
\lim_{n\rightarrow \infty}\max_{z\in K}|f(z)-P_n(z)|=0.
\end{align*}
\end{theorem}
\vspace{-.5cm}
\begin{figure}[htbp]
\centerline{\includegraphics[width=7cm]{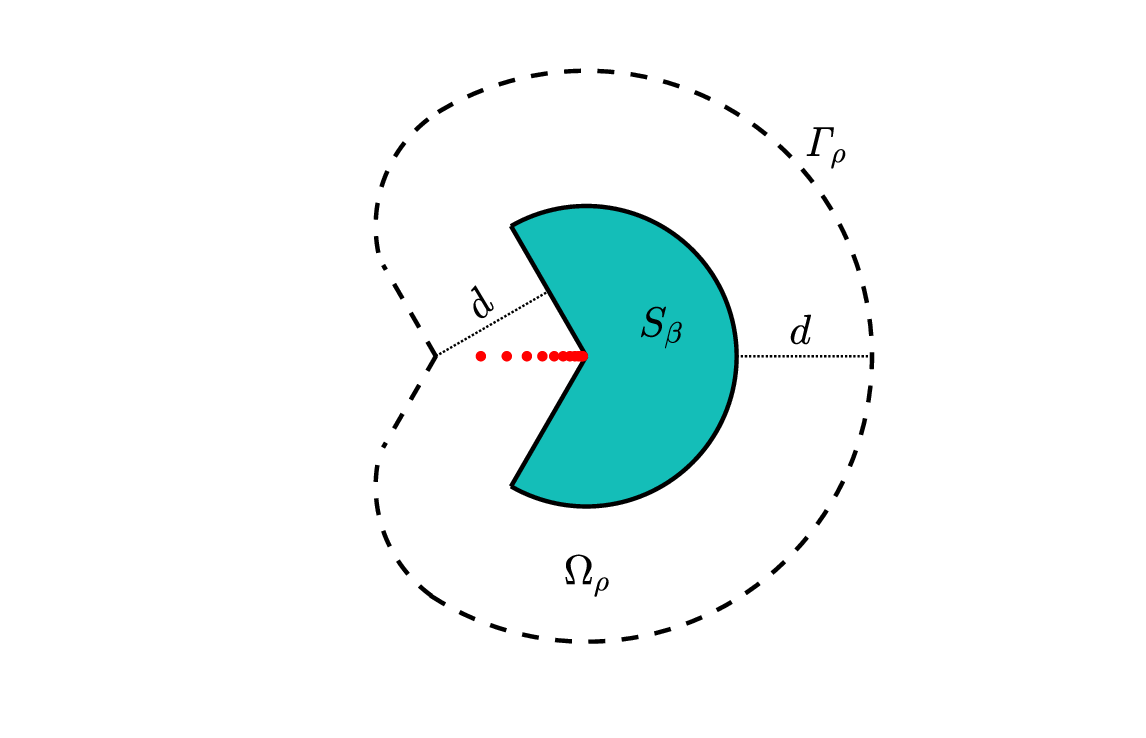}}
\vspace{-0.55cm}
\caption{A holomorphic domain $\Omega_{\rho}$ of $r^{(l)}_2(z)$ covering $S_{\beta}$.}
\label{Omega}
\end{figure}

It is worthy of noting from \eqref{eq:ECrat1C} that $p_j<-C$ for $N_1<j\le N_t$ and $r^{(l)}_2(z)$ is analytic in a fixed neighborhood $\Omega_\rho$  of $S_\beta$
$$
\Omega_\rho=\left\{z:\, {\rm dist}(z,S_\beta)\le d\right\},\quad 0<d=\left\{\begin{array}{ll}
\min\left\{\frac{1}{2},\frac{C}{2}\right\},&0\le \beta\le 1\\
\min\left\{\frac{1}{2},\frac{C}{2}\right\}\sin\frac{\beta}{2}\pi,&1<\beta<2\end{array}\right.
$$
 independent of $N_t$, $N_1$ and $p_j$ for $N_1<j\le N_t$ (see {\sc Figure} \ref{Omega}).

A straightforward verification shows that $p_j$ is outside $\Omega_\rho$ for $j\ge N_1$ and $|z|\le |z-z_0|+|z_0|\le d+1\le \frac{3}{2}$,
where $z_0\in S_\beta$ satisfies $|z-z_0|={\rm dist}(z,S_\beta)$ for $z\in \Omega_\rho$. Then it implies
$|z-s_k|\le s_k+\frac{3}{2}$ and
$$\max\limits_{z\in  \Omega_\rho}\prod\limits_{k=1}^{\ell}|z-s_k|\le\prod\limits_{k=1}^{\ell}\left(s_k+\frac{3}{2}\right)
=\prod\limits_{k=1}^{\ell}\left[\delta+2+\frac{1}{2}\cos{\frac{(2k-1)\pi}{2\ell}}\right]
\le(\delta+2)^{\ell}
$$
for $z\in \Omega_\rho$, and
$$
\left|Ce^{\frac{1}{\alpha}(jh-T)}+z\right|\ge \left|Ce^{\frac{1}{\alpha}(jh-T)}+z_0\right|-|z-z_0|\ge Ce^{\frac{1}{\alpha}(jh-T)} \varkappa(\beta)-d
\ge \frac{1}{2}Ce^{\frac{1}{\alpha}(jh-T)} \varkappa(\beta)
$$
by \eqref{eq:est2}.
Subsequently it is revealed that
for $N_1< j\le N_t-l$ that
\begin{align*}
\left|\frac{zC^{\alpha}(jh-T)^le^{jh-T}}{Ce^{\frac{1}{\alpha}(jh-T)}+z}
\left(\prod\limits_{k=1}^{\ell}\frac{z-s_k}{Ce^{\frac{1}{\alpha}(jh-T)}+s_k}\right)\right|
\le \frac{3(\kappa T)^le^{-\frac{1}{\kappa}(jh-T)}(\delta+2)^{\ell}}{C^{\ell+1-\alpha}\varkappa(\beta)}.
 \end{align*}
Thus, it yields  that
\allowdisplaybreaks
\begin{align}\label{eq:boundofr2}
\left|r^{(l)}_2(z)\right|=&\left|h\sum_{j=N_1+1}^{N_t}\frac{z C^\alpha(jh-T)^l e^{jh-T}}{Ce^{\frac{1}{\alpha}(jh-T)}+z}
\left(\prod\limits_{k=1}^{\ell}\frac{z-s_k}{Ce^{\frac{1}{\alpha}(jh-T)}+s_k}\right)\right|\notag\\
\le&\frac{3(\kappa T)^l(\delta+2)^{\ell}}{C^{\ell+1-\alpha}\varkappa(\beta)}\sum_{j=N_1+1}^{N_t-l}he^{-\frac{1}{\kappa}(jh-T)}+
hl\frac{3(N_th-T)^le^{-\frac{1}{\kappa}(N_th-T)}}{(\delta+2)^{-\ell}C^{\ell+1-\alpha}\varkappa(\beta)}\\
\le & \frac{3(\kappa T)^l(\delta+2)^{\ell}}{C^{\ell+1-\alpha}\varkappa(\beta)}\int_{N_1h}^{+\infty}e^{-\frac{1}{\kappa}(u-T)}\mathrm{d}u+
\frac{3\kappa^l hl(\delta+2)^{\ell}}{C^{\ell+1-\alpha}\varkappa(\beta)}\notag\\
\le&\frac{3(\delta+2)^{\ell}\kappa^{l+1} (T^l+hl)}{C^{\ell+1-\alpha}\varkappa(\beta)}\le \frac{3(\delta+2)^{\ell}\kappa^{l+1}(1+l)T^l}{C^{\ell+1-\alpha}\varkappa( \beta)}\notag
\end{align}
due to the monotonicity of $e^{-\frac{1}{\kappa}t}$ and $te^{-t}\le 1$ for $t\ge 0$ and $T=N_1h$.

Following Levin and Saff \cite[(8) and (9)]{LevinSaff}, there exists a $\rho>1$ and a polynomial $q^{(l)}_n$ such that
\begin{align}\label{Runge11}
&r^{(l)}_2(z)-q^{(l)}_n(z)=\frac{1}{2\pi i}\int_{\Gamma_\rho}\frac{\ell_{n+1}(z)}{\ell_{n+1}(t)}\frac{r^{(l)}_2(t)}{t-z}\mathrm{d}t,\quad \ell_{n+1}(z)=\prod_{k=0}^n(z-z_k),\notag\\
&\limsup_{n\rightarrow \infty}\left\|r^{(l)}_2-q^{(l)}_n\right\|_{C(S_\beta)}^{1/n}\le \limsup_{n\rightarrow \infty}\left(\max_{z\in \Omega_\rho}\left|r^{(l)}_2\right|\right)^{1/n}\frac{1}{\rho},
\end{align}
i.e., there is a positive integer $N_0$ such that $\left\|r^{(l)}_2-q^{(l)}_n\right\|_{C(S_\beta)}^{1/n}\le \left(\max_{z\in \Omega_\rho}\left|r^{(l)}_2\right|\right)^{1/n}\frac{1}{\rho}$ for $n\ge N_0$ by the definition of limsup,
where $\Gamma_\rho$ is the boundary of $\Omega_\rho$ and $\{z_k\}_{k=0}^n$ are the $n+1$ Fekete points on $\Gamma_\rho$.

Analogous to Herremans, Huybrechs and Trefethen \cite[p. 5]{Herremans2023}, there is a polynomial $q^{(l)}_{N_2}(z)$ of degree
$N_2\ge\frac{\sigma\alpha}{\log{\rho}}\sqrt{N_1}=\mathcal{O}(\sqrt{N_1})\ge N_0$
such that $\rho^{-N_2}\le e^{-\sigma\alpha\sqrt{N_1}}=e^{-T}$. Then, using \eqref{Runge11} we have
\begin{align}\label{polynomial app}
\bigg|E^{(l)}_{PA}(z)\bigg|=\bigg|r^{(l)}_2(z)-q^{(l)}_{N_2}(z)\bigg|
\le \frac{3(\delta+2)^{\ell}\kappa^{l+1} (1+l)T^l}{C^{\ell+1-\alpha}\varkappa({\beta})}e^{-T}.
\end{align}
Consequently, by denoting
$P^{(l)}_{N_2}(z)=q^{(l)}_{N_2}(z)+P^{(l)}_\ell(z)$ with $N_2\ge\ell$, we obtain
\begin{align}\label{app_for_rNt}
r^{(l)}_{N_t}(z)=r^{(l)}_{N_1}+P^{(l)}_{N_2}(z)+E^{(l)}_{PA}(z).
\end{align}

\subsection{LPs with $N_2=\mathcal{O}(\sqrt{N_1})$ for $z^{\alpha}$ and  $z^{\alpha}\log z$}\label{sec3-5} Thus, by combining
\eqref{eq:zalpha}, \eqref{eq:intC1}, \eqref{polynomial app} and \eqref{app_for_rNt} with $\ell=\lfloor\alpha\rfloor$, the LP for $z^{\alpha}$ with $z\in S_{\beta}$ is established as
\begin{align}\label{ratappforzalpha}
  z^{\alpha}
  =&\frac{\sin(\alpha\pi)}{(-1)^{\ell}\alpha\pi}
  \left[r^{(0)}_{N_t}(z)+E_Q^{(0)}(z)+E^{(0)}_T(z)\right]
  +z\mathcal{L}[z^{\alpha-1};s_1,\ldots,s_\ell]\notag\\
  =&\frac{\sin(\alpha\pi)}{(-1)^{\ell}\alpha\pi}
  \left[r^{(0)}_{N_1}(z)+P^{(0)}_{N_2}(z)+E^{(0)}_{PA}(z)+E_Q^{(0)}(z)+E^{(0)}_T(z)\right]\\
  &+z\mathcal{L}[z^{\alpha-1};s_1,\ldots,s_\ell] \notag\\
  =&\bar{r}_{N_1}(z)+\bar{P}_{N_2}(z)+\bar{E}(z)=:\bar{r}_{N}(z)+\bar{E}(z)\notag
\end{align}
with
\begin{align}
\bar{r}_{N_1}(z)=&\sum_{j=0}^{N_1}\frac{\bar{a}_j}{z-p_j},\ \bar{a}_j=\frac{hp_j|p_j|^\alpha\sin(\alpha\pi)}{\alpha\pi},\ 0\le j\le N_1,\label{eq:error1}\\
\bar{P}_{N_2}(z) =&\frac{\sin(\alpha\pi)}{(-1)^{\ell}\alpha\pi}P^{(0)}_{N_2}(z)+z\mathcal{L}[z^{\alpha-1};s_1,\ldots,s_\ell],\label{eq:error11}\\
\bar{E}(z)=&\frac{\sin(\alpha\pi)}{(-1)^{\ell}\alpha\pi}\left[
E^{(0)}_T(z)+E_Q^{(0)}(z)+E^{(0)}_{PA}(z)\right],\label{eq:error111}
\end{align}
where $\bar{P}_{N_2}(z)$ is a polynomial of degree $N_2=\mathcal{O}(\sqrt{N_1})$ and $\bar{E}(z)$ satisfies by \eqref{boundforEl(z)} and \eqref{polynomial app} that
\begin{align}
\left|\bar{E}(z)\right|
\le&\frac{|\sin(\alpha\pi)|}{\alpha\pi}\bigg[
\frac{2\mathbb{T}_{\ell,\beta}C^{\alpha}}{\varkappa(\beta)e^T}\left(\frac{1}{\delta^{\ell}}+\frac{\kappa }{C^{\ell+1}}\right)
+\frac{3(\delta+2)^{\ell}\kappa e^{-T}}{C^{\ell+1-\alpha}\varkappa(\beta)}
+\left|E_Q^{(0)}(z)\right|\bigg]\notag\\
=&\frac{2|\sin(\alpha\pi)|\mathbb{T}_{\ell,\beta}}{\alpha\pi\delta^{\ell}\varkappa(\beta)e^TC^{-\alpha}}
+\frac{|\sin(\alpha\pi)|C^{\alpha-\ell-1}}{(\ell+1-\alpha)\pi e^{T}}\left[\frac{2\mathbb{T}_{\ell,\beta}+3(\delta+2)^{\ell}}{\varkappa(\beta)}\right]
+\frac{|\sin(\alpha\pi)|}{\alpha\pi}\left|E_Q^{(0)}(z)\right|\label{boundforEbar}\\
=&\max\left\{\frac{C^{\alpha}}{\delta^{\alpha}},1\right\}
\frac{\mathbb{T}_{\ell,\beta}\mathcal{O}(e^{-T})}{\varkappa(\beta)}
+\frac{(\delta+2)^{\ell}\mathcal{O}(e^{-T})}{\varkappa(\beta)}
+\frac{|\sin(\alpha\pi)|}{\alpha\pi}\left|E_Q^{(0)}(z)\right|\notag
\end{align}
by noticing that $\ell=\lfloor\alpha\rfloor$.

\bigskip
Moreover, from \eqref{eq:cint_log_gener_substituting}, \eqref{eq:intC1}, \eqref{polynomial app} and \eqref{app_for_rNt} with $\ell=\lceil\alpha\rceil$ we also establish the LP for $z^{\alpha}\log{z}$ with $z\in S_{\beta}$
\allowdisplaybreaks
\begin{align}\label{ratappforzalphalog}
  z^{\alpha}\log{z}
  =&\frac{\sin(\alpha\pi)}{(-1)^\ell\alpha^2\pi}
  \left[r^{(1)}_{N_t}(z)+E_Q^{(1)}(z)+E^{(1)}_T(z)\right]
  +\bigg[\frac{\sin(\alpha\pi)\log{C}}{(-1)^\ell\alpha\pi}
  +\frac{\cos(\alpha\pi)}{(-1)^{\ell}\alpha}\bigg]\notag\\
  &\cdot\Big[r^{(0)}_{N_t}(z)+E_Q^{(0)}(z)+E^{(0)}_T(z)\Big]+z\mathcal{L}[z^{\alpha-1}\log{z};s_1,\ldots,s_\ell]\\
  =&\frac{\sin(\alpha\pi)}{(-1)^\ell\alpha^2\pi}
  \left[r^{(1)}_{N_1}(z)+P^{(1)}_{N_2}(z)+E^{(1)}_{PA}(z)+E_Q^{(1)}(z)+E^{(1)}_T(z)\right]\notag\\
  &+\left[\frac{\sin(\alpha\pi)\log{C}}{(-1)^\ell\alpha\pi}+\frac{\cos(\alpha\pi)}{(-1)^{\ell}\alpha}\right]
  \Big[r^{(0)}_{N_1}(z)+P^{(0)}_{N_2}(z)+E^{(0)}_{PA}(z)+E_Q^{(0)}(z)+E^{(0)}_T(z)\Big]\notag\\
  &+z\mathcal{L}[z^{\alpha-1}\log{z};s_1,\ldots,s_\ell]\notag\\
  =&\widetilde{r}_{N_1}(z)+\widetilde{P}_{N_2}(z)+\widetilde{E}(z)=:\widetilde{r}_{N}(z)+\widetilde{E}(z)\notag
\end{align}
with
\begin{align}\label{eq:error1_2}
&\widetilde{r}_{N_1}(z)=\sum_{j=0}^{N_1}\frac{\widetilde{a}_j}{z-p_j},\\ &\widetilde{a}_j=\frac{hp_j|p_j|^\alpha(jh-T)\sin(\alpha\pi)}{\alpha^2\pi}+
\left[\frac{\sin(\alpha\pi)\log{C}}{\alpha\pi}
+\frac{\cos(\alpha\pi)}{\alpha}\right]hp_j|p_j|^\alpha\notag
\end{align}
for $0\le j\le N_1$ according to Cauchy's residue theorem, and
\begin{align}
\widetilde{P}_{N_2}(z) =&\frac{\sin(\alpha\pi)}{(-1)^{\ell}\alpha^2\pi}P^{(1)}_{N_2}(z)
+\left[\frac{\sin(\alpha\pi)\log{C}}{(-1)^\ell\alpha\pi}
+\frac{\cos(\alpha\pi)}{(-1)^{\ell}\alpha}\right]P^{(0)}_{N_2}(z)\notag\\
&+z\mathcal{L}[z^{\alpha-1}\log{z};s_1,\ldots,s_\ell],\\
\widetilde{E}(z)=&\frac{\sin(\alpha\pi)}{(-1)^{\ell}\alpha^2\pi}\left[
E^{(1)}_T(z)+E_Q^{(1)}(z)+E^{(1)}_{PA}(z)\right]\notag\\
&+\left[\frac{\sin(\alpha\pi)\log{C}}{(-1)^\ell\alpha\pi}
+\frac{\cos(\alpha\pi)}{(-1)^{\ell}\alpha}\right]
\left[E^{(0)}_T(z)+E_Q^{(0)}(z)+E^{(0)}_{PA}(z)\right],
\end{align}
where $\widetilde{P}_{N_2}(z)$ is a polynomial of degree $N_2$ and $\widetilde{E}(z)$ is derived from \eqref{boundforEl(z)}, \eqref{polynomial app}, and $\kappa=\frac{\alpha}{\lceil\alpha\rceil+1-\alpha}$, $\widetilde{E}(z)$ satisfies
\begin{align}\label{boundforwidetildeE}
\left|\widetilde{E}(z)\right|
\le&\frac{|\sin(\alpha\pi)|}{\alpha^2\pi}\bigg[
\frac{2(1+T)\mathbb{T}_{\ell,\beta}}{C^{-\alpha}\varkappa(\beta)e^{T}}
\left(\frac{1}{\delta^{\ell}}+\frac{\kappa^2}{C^{\ell+1}}\right)
+\frac{3\kappa^{2}(\delta+2)^{\ell}2T}{C^{\ell+1-\alpha}\varkappa(\beta)}e^{-T}
+\left|E_Q^{(1)}(z)\right|\bigg]\notag\\
&+\left|\frac{\sin(\alpha\pi)\log{C}}{(-1)^\ell\alpha\pi}
+\frac{\cos(\alpha\pi)}{(-1)^{\ell}\alpha}\right|
\bigg[
\frac{2\mathbb{T}_{\ell,\beta}C^{\alpha}}{\varkappa(\beta)e^T}\left(\frac{1}{\delta^{\ell}}+\frac{\kappa }{C^{\ell+1}}\right)\\
&+\frac{3(\delta+2)^{\ell}\kappa e^{-T}}{C^{\ell+1-\alpha}\varkappa(\beta)}
+\left|E_Q^{(0)}(z)\right|\bigg]\notag\\
=&\max\left\{\frac{C^{\alpha}}{\delta^{\alpha}},1\right\}
\frac{(\alpha+1)\mathbb{T}_{\ell,\beta}}{\alpha\varkappa(\beta)}\mathcal{O}(Te^{-T})
+\frac{(\delta+2)^{\ell}}{\varkappa(\beta)}\mathcal{O}(Te^{-T})\notag\\
&+\frac{|\sin(\alpha\pi)|}{\alpha^2\pi}\left|E_Q^{(1)}(z)\right|
+\left|\frac{\sin(\alpha\pi)\log{C}}{(-1)^\ell\alpha\pi}
+\frac{\cos(\alpha\pi)}{(-1)^{\ell}\alpha}\right|\left|E_Q^{(0)}(z)\right|.\notag
\end{align}

 It is remarkable that
according to the above discussion based upon Runge's approximation theorem, choosing $N_2=\mathcal{O}(\sqrt{N_1})$ is requisite to balance the truncated errors and the approximation errors on $r^{(l)}_2(z)$. For more details see Remark \ref{rungeN} and {\sc Figure} \ref{LPsLargerN2}.

\subsection{Extension of LPs to  $g(z)z^{\alpha}$ and $g(z)z^{\alpha}\log{z}$}\label{sec3-6}
Suppose $g(z)$ is an analytic function in a neighborhood of $S_\beta$, then
similarly from the proof of Runge's approximation theorem (see Subsection \ref{sec3-4}),
$g(z)$ can be approximated by a polynomial $P^{(g)}_{N_2}(z)$ with exponential convergence rate, that is, $\|g-P^{(g)}_{N_2}\|_{C(S_{\beta})}=\mathcal{O}(e^{-T})$.

Combining with \eqref{ratappforzalpha} and \eqref{ratappforzalphalog}, we have for some coefficients $\big\{\bar{a}^{(g)}_j\big\}_{j=0}^{N_1}$ and a polynomial $\bar{Q}^{(g)}_{N_2}(z)$ of degree $N_2$ that
\begin{align}\label{LPg(z)zalpha}
g(z)z^\alpha=&\left[P^{(g)}_{N_2}(z)+\mathcal{O}(e^{-T})\right]\left[\bar{r}_{N}(z)+\bar{E}(z)\right]\notag\\
=&P^{(g)}_{N_2}(z)\bar{r}_{N}(z)+P^{(g)}_{N_2}(z)\bar{E}(z)+\left[\bar{r}_{N}(z)+\bar{E}(z)\right]\mathcal{O}(e^{-T})\\
=&\sum_{j=0}^{N_1}\frac{\bar{a}^{(g)}_j}{z-p_j}+\bar{Q}_{N_2}^{(g)}(z)+P^{(g)}_{N_2}(z)\bar{P}_{N_2}(z)+\bar{E}^{(g)}(z)
=:\bar{r}^{(g)}_{N}(z)+\bar{E}^{(g)}(z)\notag
\end{align}
with
\begin{align*}
&\bar{r}^{(g)}_{N_1}(z):=\sum_{j=0}^{N_1}\frac{\bar{a}^{(g)}_j}{z-p_j},\
\bar{P}^{(g)}_{N_2}(z):=\bar{Q}_{N_2}^{(g)}(z)+P^{(g)}_{N_2}(z)\bar{P}_{N_2}(z),\\
&\bar{E}^{(g)}(z)=:\left[g(z)+\mathcal{O}(e^{-T})\right]\bar{E}(z)+
\left[\bar{r}_{N_1}(z)+\bar{P}_{N_2}(z)\right]\mathcal{O}(e^{-T})
\end{align*}
and for some coefficients $\big\{\widetilde{a}^{(g)}_j\big\}_{j=0}^{N_1}$ and a polynomial $\widetilde{Q}^{(g)}_{N_2}(z)$ of degree $N_2$ that
\begin{align}\label{LPg(z)zalphalogz}
g(z)z^\alpha\log{z}
=&\left[P^{(g)}_{N_2}(z)+\mathcal{O}(e^{-T})\right]\left[\widetilde{r}_{N}(z)
+\widetilde{E}(z)\right]\notag\\
=&P^{(g)}_{N_2}(z)\widetilde{r}_{N_1}(z)+P^{(g)}_{N_2}(z)\widetilde{P}_{N_2}(z)+\widetilde{E}^{(g)}(z)\\
=&\sum_{j=0}^{N_1}\frac{\widetilde{a}^{(g)}_j}{z-p_j}
+\widetilde{Q}_{N_2}^{(g)}(z)+P^{(g)}_{N_2}(z)\widetilde{P}_{N_2}(z)+\widetilde{E}^{(g)}(z)
=:\widetilde{r}^{(g)}_{N}(z)+\widetilde{E}^{(g)}(z)\notag
\end{align}
with
\begin{align*}
&\widetilde{r}^{(g)}_{N_1}(z):=\sum_{j=0}^{N_1}\frac{\widetilde{a}^{(g)}_j}{z-p_j},\
\widetilde{P}^{(g)}_{N_2}(z):=\widetilde{Q}_{N_2}^{(g)}(z)+P^{(g)}_{N_2}(z)\widetilde{P}_{N_2}(z),\\
&\widetilde{E}^{(g)}(z):=\left[g(z)+\mathcal{O}(e^{-T})\right]\widetilde{E}(z)+
\left[\widetilde{r}_{N_1}(z)+\widetilde{P}_{N_2}(z)\right]\mathcal{O}(e^{-T}).
\end{align*}

Therefore, if Theorem \ref{mainthm} holds for prototype functions $z^\alpha$ and $z^\alpha\log z$, then Theorem \ref{mainthm} also holds for  $g(z)z^{\alpha}$ and $g(z)z^{\alpha}\log{z}$, respectively. In the following, we are mainly concerned with  Theorem \ref{mainthm} for $z^\alpha$ and $z^\alpha\log z$.

In particular, from LPs \eqref{ratappforzalpha}, \eqref{boundforEbar} and  \eqref{ratappforzalphalog}, \eqref{boundforwidetildeE} for  $z^\alpha$ and $z^\alpha\log z$ for $z\in S_{\beta}$,  respectively, we only need to focus on the quadrature errors on $\bar{r}_{N_t}(z)$ and $\widetilde{r}_{N_t}(z)$ by introducing Paley-Wiener type theorems, from which we may establish Theorem \ref{mainthm}.

\section{Paley-Wiener type theorems and best convergence rate for
trapezoidal rule}\label{sec:21}
The Paley-Wiener theorem, a fundamental pillar of complex and harmonic analysis, establishes a deep duality between the decay properties of a function's Fourier transform and its analytic continuation in the complex plane. This profound result finds a natural counterpart in the Poisson summation formula, which serves as a critical bridge linking discrete Fourier series expansions with their continuous Fourier transform counterparts through periodic summation.

In this section, we will extend the Paley-Wiener theorem to a horizontal strip in the complex plane and derive the optimal exponential convergence rate for the trapezoidal rule approximation over the entire real line.
Without ambiguity, we denote $f$ for any function defined on $\mathbb{R}$.

Assume the validity of the Fourier
inversion formula
\begin{align*}
f(x)=\int_{-\infty}^{+\infty}\mathfrak{F}[f](\xi)e^{2\pi i x\xi}\mathrm{d}\xi\quad {\rm if}\quad
\mathfrak{F}[f](\xi)=\int_{-\infty}^{+\infty}f(x)e^{-2\pi i x\xi}\mathrm{d}x.
\end{align*}

\allowdisplaybreaks
\begin{theorem}\label{PalWieThm}\cite[Paley-Wiener Theorem, Chapter 4]{Stein}
  Suppose $f$ is continuous and of moderate decrease on
$\mathbb{R}$. Then, $f$ has an extension to the complex plane that is entire with
$|f(z)|\le Ae^{2\pi M|z|}$  for some $A >0$, if and only if $\mathfrak{F}[f](\xi)$ is supported in the given
interval $[-M, M]$.
\end{theorem}

The Paley-Wiener theorem specifically states that if a function bounded by $Ae^{2\pi M|z|}$ is entire and then its Fourier transform is supported on a specific interval. If $f$ is not entire on  $\mathbb{C}$ but is holomorphic in the horizontal strip
\begin{align}\label{strip}
\Xi_a=\{z\in \mathbb{C}:\, |\Im(z)| < a\}\quad (a>0)
\end{align}
and
\begin{align}\label{unicon}
|f(x + iy)|\le \frac{A}{1+x^\tau}
\end{align}
for all $x\in \mathbb{R}$ and $|y| < a$ (letting $\mathcal{F}_a$ denote the set of all functions that satisfy \eqref{strip} and \eqref{unicon}), then it follows

\begin{theorem}\label{PalWie} \cite[Theorem 2.1, Chapter 4]{Stein}
If $f$ belongs to the class $\mathcal{F}_a$, then its Fourier transform
$\left|\mathfrak{F}[f](\xi)\right| \le  Be^{-2\pi b|\xi|}$ for some constant $B$ and any $0 \le b < a$.
\end{theorem}

The constant $b$ in the inequality
 $|\mathfrak{F}[f](\xi)| \le  Be^{-2\pi b|\xi|}$ cannot be replaced  by $a$. For instance, let $f(x)=\frac{2x}{(1+x^2)^2}\in \mathcal{F}_1$, then
 $$
 |\mathfrak{F}[f](\xi)|=\left|-\int_{-\infty}^{+\infty}e^{-2\pi i x\xi}\mathrm{d}\Big(\frac{1}{1+x^2}\Big)\right|=2\pi|\xi|\left|\int_{-\infty}^{+\infty}\frac{e^{-2\pi i x\xi}}{1+x^2}\mathrm{d}x\right|= 2\pi^2|\xi| e^{-2\pi|\xi|}
 $$
 (see Example \ref{ex1}).

\bigskip
It is of particular interest to determine under what conditions the upper bound $Be^{-2\pi a|\xi|}$ for $\xi\in \mathbb{R}$ can be attained?  Two most cited conditions are
\begin{align}\label{integral0}
\int_{-a}^a
|f(x + i\eta)|\mathrm{d}\eta =\mathcal{O} (|x|^{-\alpha}),\quad x\rightarrow \pm \infty,
\end{align}
for some $\alpha>0$ and
\begin{align}\label{integral}
\int_{-\infty}^{+\infty}
|f(x \pm i\eta)|\mathrm{d}x <+\infty
\end{align}
uniformly for $\eta\in (-a,a)$ (cf.  Denich and Novati \cite{DNJSC2024}, Lund and Bowers \cite[Definition 2.12]{LunBow1992}), 
then it implies $\left|\mathfrak{F}[f](\xi)\right|\le Be^{-2\pi a|\xi|}$ for $\xi\in \mathbb{R}$ by applying the uniform condition \eqref{integral} (the uniform condition \eqref{integral} is weaker than \eqref{unicon}).

However, the conditions \eqref{unicon} and \eqref{integral} are too strong.  For any function holomorphic in the horizontal strip $\Xi_a$ with a pole on the boundary, it fails to satisfy \eqref{integral}. Since from \eqref{integral}, without loss of generality, assume $x_0-ia$ is a pole of $f$ on the boundary $\Xi_a$, it may hold that
\begin{align*}
\lim_{\eta\rightarrow a^{-}}
\int_{-\infty}^{+\infty}
|f(x - i\eta)|\mathrm{d}x =\int_{-\infty}^{+\infty}
|f(x - ia^{-})|\mathrm{d}x=+\infty.
\end{align*}
For example, let $f(x)=\frac{1}{1+x^2}$ with poles $z_{1,2}=\mp i$. If $f$ satisfies \eqref{integral} for $a=1$, then
\begin{align*}
\lim_{\eta\rightarrow a^{-}}
\int_{-\infty}^{+\infty}
\bigg|\frac{1}{(x - i\eta)^2+1}\bigg|\mathrm{d}x =\int_{-\infty}^{+\infty}
\bigg|\frac{1}{(x - i)^2+1}\bigg|\mathrm{d}x=\int_{-\infty}^{+\infty}\frac{\mathrm{d}x}{|x|\sqrt{x^2+4}}=\infty.
\end{align*}

Next we will consider the attainability of the upper bound $Be^{-2\pi a|\xi|}$ in \eqref{PalWie} based on the multiplicities of the poles on the boundary of $\Xi_a$,
and the convergence rate for the trapezoidal rule approximation.
\begin{theorem}\label{PW2}
Suppose $|\mathfrak{F}[f](\xi)$ exists for $\xi \in \mathbb{R}$, and $f$ is holomorphic in the horizontal strip
$\Xi_0=\{z\in \mathbb{C}:\, |\Im(z)| < a_0\}$ for some $a_0>a>0$ except for finite poles $z_1,\ldots,z_m$, where $\Im(z)$ denotes the imaginary part of $z$. Let $a=\min_{1\le k\le m}|\Im(z_k)|$ be satisfied by $z_{k_1},\ldots,z_{k_n}$ ($k_j\le m$), that is, $|\Im(z_{k_1})|=\cdots=|\Im(z_{k_n})|=a$, and their orders are $j_{k_1},\ldots,j_{k_n}$ respectively and $m_0=\max\{j_{k_1},\ldots,j_{k_n}\}$.
If
\begin{align}\label{condition1_paley-wienner}
\lim_{x\rightarrow \infty}\int_{-a_0}^{a_0}
|f(x + i\eta)|\mathrm{d}\eta =0
\end{align}
and
\begin{align}\label{condition2_paley-wienner}
B^{\pm}_0:=\int_{-\infty}^{+\infty}|f(x \pm ia_0)|\mathrm{d}x <+\infty
\end{align}
hold, then
\begin{align}\label{asymFT}
|\mathfrak{F}[f](\xi)|\le  B(|\xi|^{m_0-1}+1)e^{-2\pi a|\xi|}
\end{align}
for some constant $B$.
In addition, for each $h>0$ it holds
\begin{align}\label{PossionSum2}
E(h)=&\bigg|\int_{ \mathbb{R}}f(x)\mathrm{d}x-h\sum_{n\not=0}f\left(nh\right)\bigg|
\le\sum_{n\not=0}\bigg|\mathfrak{F}[f]\left(\frac{n}{h}\right)\bigg|\notag\\
\le& \frac{2B}{e^{\frac{2a\pi}{h}}-1}+\widetilde{B}\max\left\{\frac{1}{h^{m_0-1}},h\right\}e^{-\frac{2a\pi}{h}}
\end{align}
for some constant $\widetilde{B}$.

Moreover, suppose $f$ is continuous and of moderate decrease on
$\mathbb{R}$ and $\mathfrak{F}[f](\xi)$ satisfies the decay condition
$$|\mathfrak{F}[f](\xi)| \le 2B(|\xi|^{m_0-1}+1)e^{-2\pi a|\xi|}
$$
for some constants $a,B >0$ and a positive integer $m_0$.
Then $f(x)$ is the restriction to $\mathbb{R}$ of a
function $f(z)$ holomorphic in the strip $\{z \in \mathbb{C}:\, |\Im(z)| <a \}$.
\end{theorem}

\begin{proof}
The inequality $|\mathfrak{F}[f](\xi)| \le  B(|\xi|^{m_0-1}+1)e^{-2\pi a|\xi|}$  obviously holds for $\xi=0$ with $B=|\mathfrak{F}[f](0)|$  due to that  $|\mathfrak{F}[f](\xi)$ exists for $\xi\in \mathbb{R}$. Assume first that $\xi>0$. A similar argument for $\xi<0$ can be applied.

Without loss of generality, we suppose $z_1,\ldots, z_{m_1}$ in the lower semi-strip domain,  $z_1$ is the pole with $\Im(z_1)=a$ and of order $m_0$, and $z_{m_1+1},\ldots, z_{m}$ in the upper semi-strip domain (see {\sc Figure} \ref{integral_path_uniform_paley_wienner}).
From \eqref{condition1_paley-wienner} we see that
$$\bigg|\int_{ X}^{X- ia_0}f(x)e^{-2\pi i x\xi}\mathrm{d}x\bigg|\le \int_{0}^{a_0}|f(X- i\eta)|\mathrm{d}\eta$$
tends to $0$ as $X\rightarrow  \pm \infty$, which,
together with Cauchy's integral theorem and residue theorem,  leads to
\begin{align}\label{integralcontour}
\mathfrak{F}[f](\xi)
=\int_{-\infty}^{+\infty}
f(x- ia_0)e^{-2\pi i x\xi}e^{-2\pi a_0\xi}\mathrm{d}x-2\pi i\sum_{k=1}^{m_1}\mathrm{Res}\left[f(z)e^{-2\pi iz\xi},z_k\right].
\end{align}

\begin{figure}[htbp]
\vspace{-.3cm}
\centerline{\includegraphics[width=13cm]{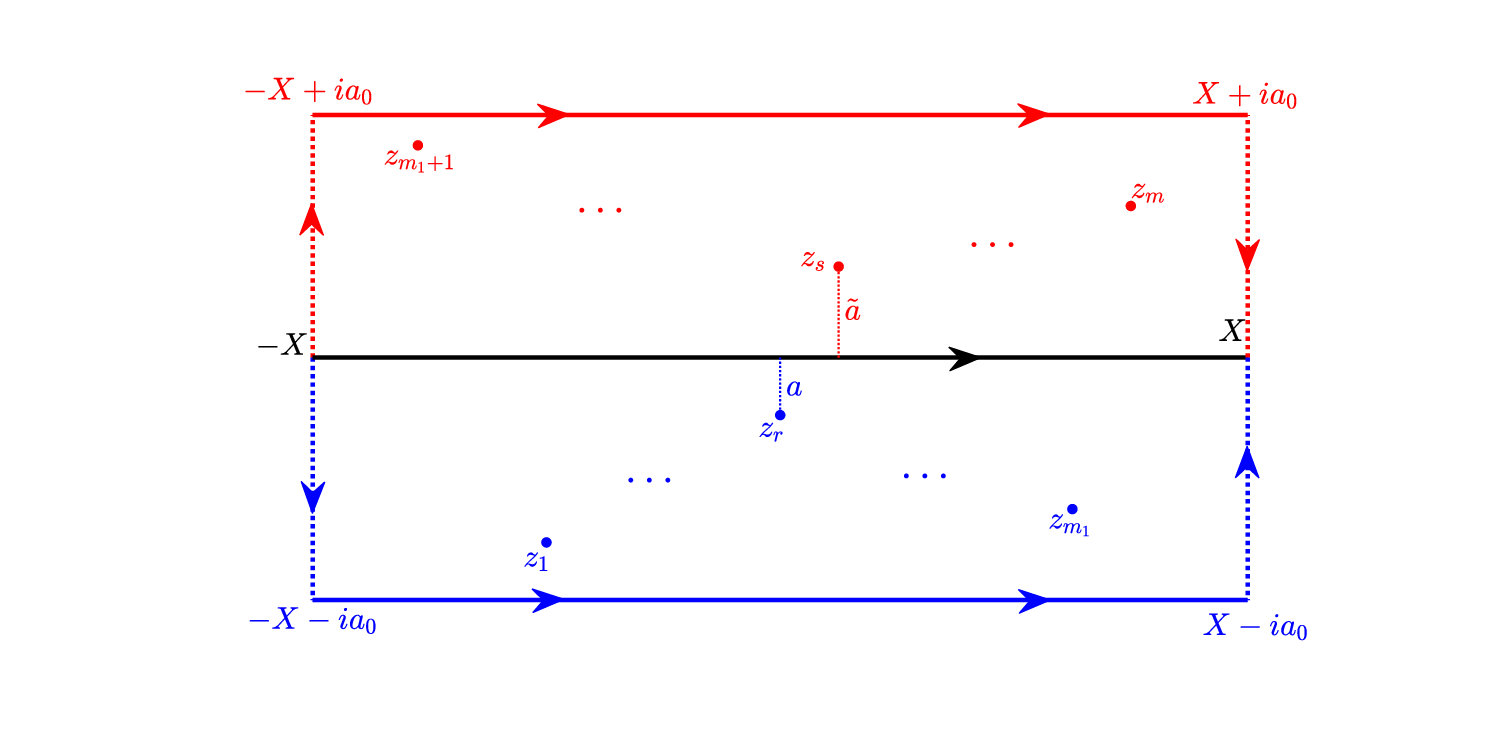}}\vspace{-.8cm}
\caption{The integrand $f(z)$ is holomorphic in the strip domain $\{z\in\mathbb{C}:\ |\Im(z)|\le a_0\}$ except for the poles $z_1,\ldots,z_{m}$.
}\label{integral_path_uniform_paley_wienner}
\end{figure}

Note that
\begin{align*}
\left|\int_{-\infty}^{+\infty}
f(x - ia_0)e^{-2\pi i x\xi}e^{-2\pi a_0\xi}\mathrm{d}x\right|
\le e^{-2\pi a_0\xi}\int_{-\infty}^{+\infty}
|f(x - ia_0)|\mathrm{d}x=B^{-}_0e^{-2\pi a_0\xi}
\end{align*}
and
\begin{align*}
&\mathrm{Res}\left[f(z)e^{-2\pi iz\xi},z_1\right]=\lim_{z\rightarrow z_1}\frac{1}{(m_0-1)!}
\frac{\mathrm{d}^{m_0-1}}{\mathrm{d}z^{m_0-1}}\left[(z-z_1)^{m_0}f(z)e^{-2\pi iz\xi}\right]\\
=&\frac{1}{(m_0-1)!}\sum_{j=0}^{m_0-1}\binom{m_0-1}{j}(-2\pi i \xi)^je^{-2\pi z_1 i \xi}\lim_{z\rightarrow z_1}\frac{\mathrm{d}^{m_0-j-1}}{\mathrm{d}z^{m_0-j-1}}\left[(z-z_1)^{m_0}f(z)\right]\notag
\end{align*}
with $\binom{m_0-1}{j}=\frac{(m_0-1)!}{j!(m_0-j-1)!}$.
It is easy to verify that $\lim_{z\rightarrow z_1}\frac{\mathrm{d}^{m_0-1-j}}{\mathrm{d}z^{m_0-1-j}}\left[(z-z_1)^{m_0}f(z)\right]$ is well-defined since $z_1$ is a pole of order $m_0$,
then by $|e^{-2\pi i z_1\xi}|=e^{-2\pi  a\xi}$ it derives
\begin{align*}
\big|\mathrm{Res}\left[f(z)e^{-2\pi iz\xi},z_1\right]\big|\le& \frac{e^{-2\pi  a\xi}}{(m_0-1)!}\sum_{j=0}^{m_0-1}\binom{m_0-1}{j}(2\pi\xi)^j\\
&\cdot \bigg|\lim_{z\rightarrow z_1}\frac{\mathrm{d}^{m_0-j-1}}{\mathrm{d}z^{m_0-j-1}}\left[(z-z_1)^{m_0}f(z)\right]\bigg|
\end{align*}
and consequently $\bigg|\mathrm{Res}\left[f(z)e^{-2\pi iz\xi},z_1\right]\bigg|=\mathcal{O}\left[(\xi^{m_0-1}+1)e^{-2\pi a\xi}\right]$.

Suppose $z_k$ ($1<k\le m_1$) is of order $j_k$ and $a\le|\Im(z_k)|<a_0$.
Analogously we have
\begin{align}\label{eq:res3}
\bigg|\mathrm{Res}\left[f(z)e^{-2\pi iz\xi},z_k\right]\bigg|=\mathcal{O}\left[(\xi^{j_k-1}+1)e^{-2\pi |\Im(z_{k})|\xi}\right],
\end{align}
which directly yields that
\begin{align}\label{eq:res2}
\bigg|\mathrm{Res}\left[f(z)e^{-2\pi iz\xi},z_k\right]\bigg|=\mathcal{O}\left[(\xi^{m_0-1}+1)e^{-2\pi a\xi}\right]
\end{align}
if $|\Im(z_k)|= a$. While for $|\Im(z_k)|> a$, it is established by the uniform boundedness of $\xi^{j_{k}-m_0}e^{-2\pi (|\Im(z_{k})|-a)\xi}$ for $\xi\ge 1$ that
\begin{align*}
\xi^{j_{k}-1}e^{-2\pi |\Im(z_{k})|\xi}=\mathcal{O}(\xi^{m_0-1}e^{-2\pi a\xi}),
\end{align*}
then from \eqref{eq:res3}, \eqref{eq:res2} still holds
 for all $\xi> 0$. These together  imply that the estimate $|\mathfrak{F}[f](\xi)|\le  \underline{B}(|\xi|^{m_0-1}+1)e^{-2\pi a|\xi|}$ in  \eqref{asymFT}  holds for $\xi>0$ and some constant $\underline{B}$.

Shifting the real line up by $a_0$ we can show $|\mathfrak{F}[f](\xi)| \le  \overline{B}(|\xi|^{m_0-1}+1)e^{-2\pi a|\xi|}$ \eqref{asymFT} for $\xi<0$ and some constant $\overline{B}$,
which allows us to finish the proof with $B=\max\{|\mathfrak{F}[f](0)|,\underline{B},\overline{B}\}$.

Inequality \eqref{PossionSum2}
  follows from Poisson summation formula
  \begin{align}\label{poisum}
 \int_{-\infty}^{+\infty}f(x)   \mathrm{d}x-h\sum_{n=-\infty}^{+\infty}f\left(nh\right)
=-\sum_{n\not=0}\mathfrak{F}[f]\left(\frac{n}{h}\right)\mbox{\quad(\cite[(1.3.18)]{Stenger})}
 \end{align}
and
\allowdisplaybreaks
  \begin{align}\label{eq:trap}
\sum_{n\not=0}\bigg|\mathfrak{F}[f]\left(\frac{n}{h}\right)\bigg|
 \le & 2 B\sum_{n=1}^{+\infty}
 \left[\left(\frac{n}{h}\right)^{m_0-1}+1\right]e^{-\frac{2 n\pi}{h} a}\notag\\
=&\frac{2 Be^{-\frac{2\pi}{h} a}}{h^{m_0-1}}\sum_{n=0}^{+\infty}
 n^{m_0-1}\left(e^{-\frac{2\pi}{h} a}\right)^n+\frac{2B}{e^{\frac{2a\pi}{h}}-1}.
\end{align}
The generating function for the sequence of powers \eqref{eq:trap} is given by the Eulerian polynomial:
 $$\sum_{n=1}^{+\infty}n^kz^n=\frac{zA_k(z)}{(1-z)^{k+1}}$$
 where $A_k(z)$ is the Eulerian polynomial. For summations with weight $n^{m_0-1}e^{-\frac{2n\pi}{h} a}$, we have
 \begin{align*}
\sum_{n=0}^{+\infty}
 n^{m_0-1}\left(e^{-\frac{2\pi}{h} a}\right)^n=\frac{1}{1-e^{-\frac{2\pi}{h} a}}=\mathcal{O}(1)\max\left\{1,\frac{h}{2\pi a}\right\},\quad m_0=1,
  \end{align*}
 \begin{align*}
 \sum_{n=0}^{+\infty}
 n^{m_0-1}\left(e^{-\frac{2\pi}{h} a}\right)^n=&\frac{e^{-\frac{2\pi}{h} a}}{\left(1-e^{-\frac{2\pi}{h} a}\right)^{m_0}}\sum_{k=0}^{m_0-2}\langle\begin{array}{c}m_0-2\\
 k\end{array}\rangle e^{-\frac{2k\pi}{h} a}
 =\mathcal{O}(1)\max\left\{1,\left(\frac{h}{2\pi a}\right)^{m_0}\right\}
 \end{align*}
for $m_0\ge 2$, where $\langle\begin{array}{c}m_0-2\\
 k\end{array}\rangle$ is the Eulerian number (see \cite{Hirzebruch}). These together lead to   \eqref{PossionSum2}.

Finally,
from the condition that  $f$ is continuous and of moderate decrease on
$\mathbb{R}$ and $\mathfrak{F}[f](\xi)$, we see that $\mathfrak{F}[f](\xi)$ is continuous for $\xi\in \mathbb{R}$. Then following \cite[Theorem 5.4, Chapter 2 and Theorem 3.1, Chapter 4]{Stein} we directly attain the desired result that function $f(z)$ is holomorphic in the strip for arbitrary $0<b<a$.
\end{proof}

The following result confirms that the upper bound on $\mathfrak{F}[f]$ in Theorem \ref{PalWie} is achievable via a direct application of
Theorem \ref{PW2}.
\begin{corollary}\label{PW}
Suppose $\mathfrak{F}[f](\xi)$ exists for $\xi \in \mathbb{R}$, and $f$ is holomorphic in the horizontal strip
$\Xi_0=\{z\in \mathbb{C}:\, |\Im(z)| < a_0\}$ for some $a_0>a>0$ except for finite poles $z_1,\ldots,z_m$, where $a=\min_{1\le k\le m}|\Im(z_k)|$.
If the poles $z_k$ with $|\Im(z_k)|=a$ are simple,
and both of \eqref{condition1_paley-wienner} and \eqref{condition2_paley-wienner} hold,
then $|\mathfrak{F}[f](\xi)| \le  Be^{-2\pi a|\xi|}$ for some constant $B$.
In addition, for each $h>0$ it holds that
\begin{align}\label{PossionSum}
E(h)=\left|\int_{ \mathbb{R}}f(x)\mathrm{d}x-h\sum_{n\not=0}f\left(nh\right)\right|
\le&\sum_{n\not=0}\bigg|\mathfrak{F}[f]\left(\frac{n}{h}\right)\bigg|\le \frac{2B}{e^{\frac{2a\pi}{h}}-1}.
\end{align}

In particular, if all the poles $z_1,\ldots,z_m$ are simple, the constant $B$ in $|\mathfrak{F}[f](\xi)| \le  Be^{-2\pi a|\xi|}$ and \eqref{PossionSum}  can be replaced by
\begin{align}\label{PoissonB}
B=\max \{B_0^-,B_0^+\}+2\pi \sum_{l=1}^{m}\Big|\mathrm{Res}\left[f(z),z_k\right]\Big|
\end{align}
with $B_0^{\pm}$ defined in \eqref{condition2_paley-wienner}.
\end{corollary}

\begin{proof}
The estimates $|\mathfrak{F}[f](\xi)| \le  Be^{-2\pi a|\xi|}$ and \eqref{PossionSum} are direct consequences of Theorem \ref{PW2} with $m_0=1$.

Furthermore, if $z_1,\ldots,z_m$ are simple poles then it implies
\begin{align}\label{eq:res4}
\left|\mathrm{Res}\left[f(z)e^{-2\pi iz|\xi|},z_k\right]\right|=&\left|\lim_{z\rightarrow z_k}(z-z_k)f(z)e^{-2\pi iz|\xi|}\right|=\left|e^{-2\pi iz_k|\xi|}
\mathrm{Res}\left[f(z),z_k\right]\right|\notag\\
\le &e^{-2\pi a|\xi|}\big|\mathrm{Res}\left[f(z),z_k\right]\big|,\quad k=1,2,\ldots,m.
\end{align}
As a result, combining \eqref{integralcontour} with the proof of Theorem \ref{PW2}, we obtain \eqref{PoissonB}.
\end{proof}

\begin{example}\label{ex1}
Consider the function $f(x)=\frac{1}{1+x^2}$, which has poles at $x=\pm i$. From Corollary \ref{PW}, we choose $a_0>a=1$. By Cauchy's integralformula \eqref{integralcontour}  it follows that
\begin{align*}
\left|\mathfrak{F}[f](\xi)\right|&=
\bigg|2 i\pi \mathrm{Res}\left[f(z)e^{-2\pi iz\xi},-i{\rm sgn(\xi)}\right] -e^{-2a_0\pi|\xi|}\int_{-\infty}^{+\infty}\frac{e^{-2\pi ix\xi}\mathrm{d}x}{1+(x-ia_0{\rm sgn(\xi)})^2}\bigg|.
\end{align*}
From the inequality $\big|1+(x-ia_0{\rm sgn(\xi)})^2\big|\ge x^2+a_0^2-1$, we obtain
\begin{align*}
e^{-2a_0\pi|\xi|}\int_{-\infty}^{+\infty}\bigg|\frac{e^{-2\pi ix\xi}}{1+(x-ia_0{\rm sgn(\xi)})^2}\bigg|\mathrm{d}x\le
e^{-2a_0\pi|\xi|}\int_{-\infty}^{+\infty}\frac{\mathrm{d}x}{x^2+a_0^2-1}=\frac{\pi e^{-2a_0\pi|\xi|}}{\sqrt{a_0^2-1}}.
\end{align*}
Combining these results, we have
\begin{align*}
\pi e^{-2\pi|\xi|}-\frac{\pi e^{-2a_0\pi|\xi|}}{\sqrt{a_0^2-1}}\le \left|\mathfrak{F}[f](\xi)\right|\le \pi e^{-2\pi|\xi|}+\frac{\pi e^{-2a_0\pi|\xi|}}{\sqrt{a_0^2-1}}.
\end{align*}
Since $\mathfrak{F}[f](0)=\pi$, it implies
\begin{align*}
1-\frac{e^{-2(a_0-1)\pi|\xi|}}{\sqrt{a_0^2-1}}\le \left|\mathfrak{F}[f](\xi)\right|\pi^{-1} e^{2\pi|\xi|}\le 1+\frac{e^{-2(a_0-1)\pi|\xi|}}{\sqrt{a_0^2-1}}.
\end{align*}
Taking the limit $a_0\rightarrow +\infty$, we conclude  $\left|\mathfrak{F}[f](\xi)\right|= \pi e^{-2\pi|\xi|}$.
\end{example}

\begin{example}\label{ex2}
Let $f(x)=\frac{1}{(x^2+16\pi^2)(e^x+1)}$. Its
 poles are $z=\pm 4\pi i$ and $z_k=i(2k-1)\pi$, $k=0,\pm 1,\ldots$. We may choose $a_0=2\pi>a=\pi$, then analogous to \eqref{integralcontour}, it follows
\begin{align*}
 \left|\mathfrak{F}[f](\xi)\right|=& \left|\int_{-\infty}^{\infty}\frac{e^{-2\pi ix\xi}{\rm d}x}{(x^2+16\pi^2)(e^x+1)}\right|\\
 \le& \int_{-\infty}^{+\infty}\frac{e^{-4\pi^2|\xi|}{\rm d}x}{\big|(x-2i\pi)^2+16\pi^2\big|}+\bigg|2 i\pi \mathrm{Res}\left[f(z)e^{-2\pi iz\xi},-i\pi{\rm sgn(\xi)}\right]\bigg|\\ \le& \int_{-\infty}^{+\infty}\frac{e^{-4\pi^2|\xi|}{\rm d}x}{x^2+12\pi^2}+\frac{2e^{-2\pi^2|\xi|}}{15\pi}
 <\frac{1}{2} e^{-2\pi a|\xi|}.
\end{align*}
While for  $f(x)=\frac{1}{(x+i)^2}$, $\mathfrak{F}[f](\xi)$ can be estimated for $\xi>0$ and $a_0=2>a=1$ as
\begin{align*}
4\pi^2\xi e^{-2\pi\xi}-\pi e^{-4\pi \xi}
\le\bigg|\mathfrak{F}[f](\xi)\bigg|
=&\bigg|e^{-4\pi \xi}\int_{-\infty}^{+\infty}
\frac{e^{-2\pi i x\xi}}{(x-i)^2}\mathrm{d}x+4\pi^2\xi e^{-2\pi\xi}\bigg|\\
\le& 4\pi^2(\xi+1) e^{-2\pi\xi},
\end{align*}
which also validates the estimate in Theorem \ref{PW2},
where we used $$\int_{-\infty}^{+\infty}\frac{\mathrm{d}x}{|(x-i)^2|}= \int_{-\infty}^{+\infty}\frac{\mathrm{d}x}{1+x^2}=\pi.$$
\end{example}

\begin{figure}[htp]
\centerline{\includegraphics[width=12cm]{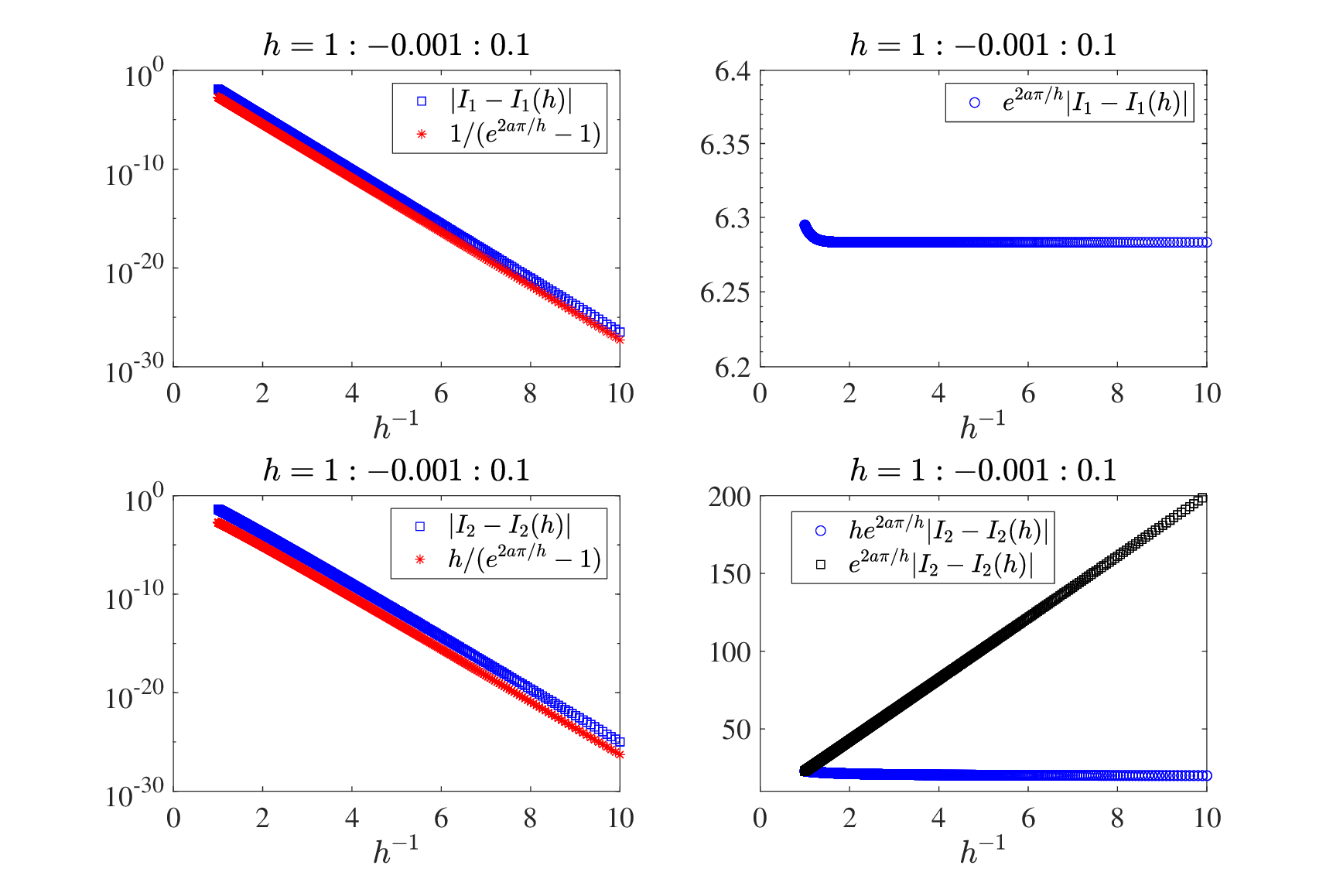}}
\caption{$I_k=\int_{-\infty}^{+\infty}\frac{\mathrm{d}x}{(1+x^2)^k}=\frac{\pi}{k}$ approximated by the trapezoidal rules for $a=1$ and $k=1,2$, respectively.}
\label{Trapezoidalerr}
\end{figure}

In the following,  we will illustrate the tight convergence rate in Theorem \ref{PW2} for the
trapezoidal rule approximation over the whole real line.

\begin{example}\label{ex3}
Let $I_1=\int_{-\infty}^{+\infty}\frac{1}{1+x^2}\mathrm{d}x=\pi$ and
$I_2=\int_{-\infty}^{+\infty}\frac{\mathrm{d}x}{(1+x^2)^2}=\frac{\pi}{2}$ be approximated by the trapezoidal rules \eqref{PossionSum2}, in which $x=\pm i$ are
$k$-th order poles of the integrand $\frac{1}{(1+x^2)^k}$, with $k$ taking values $1$ or $2$ and $a=|\Im(\pm i)|=1$. As demonstrated in \cite[Exercise 6, p. 128]{Stein},
this implies
\begin{align}\label{eq:traperr01}
I_1(h)=h\sum_{n=-\infty}^{+\infty}\frac{1}{1+n^2h^2}=\pi\coth\left(\frac{\pi}{h}\right).
\end{align}
Differentiating \eqref{eq:traperr01} with respect to $h$ (while accounting for the equation) gives
\begin{align}\label{eq:traperr02}
I_2(h)=h\sum_{n=-\infty}^{+\infty}\frac{1}{(1+n^2h^2)^2}=\frac{\pi^2}{2h}\left[\coth^2\left(\frac{\pi}{h}\right)-1\right]+\frac{\pi}{2}\coth\left(\frac{\pi}{h}\right).
\end{align}
Combining \eqref{eq:traperr01} and \eqref{eq:traperr02} leads to the conclusion that
$$
E_1(h)=I_1-I_1(h)=-\frac{2\pi}{e^{\frac{2\pi a}{h}}-1},\quad E_2(h)=I_2-I_2(h)=-\frac{h^{-1}}{e^{\frac{2\pi a}{h}}-1}\left[\frac{2\pi^2e^{\frac{2\pi a}{h}}}{e^{\frac{2\pi a}{h}}-1}+\pi h\right],
$$
which demonstrates the tight convergence rate bound \eqref{PossionSum2} in Theorem \ref{PW2} as shown in {\sc Figure} \ref{Trapezoidalerr} numerically.
\end{example}

\begin{remark}\label{remark4.1}
In particular, if $f$ is entire, $\mathfrak{F}[f](\xi)$ exists and $f$ satisfies \eqref{condition1_paley-wienner} and \eqref{condition2_paley-wienner}, from the proof of Theorem \ref{PW2}, we see that
$$
E(h)=\frac{\max \{B_0^-,B_0^+\}}{e^{\frac{2\pi a_0}{h}}-1}.
$$
Moreover, according to the Paley-Wiener Theorem (Theorem \ref{PalWieThm}), if the Fourier transform $\mathfrak{F}[f](\xi)$
is supported on the interval $[-M,M]$, then from Poisson summation formula
  \begin{align*}
 \int_{-\infty}^{+\infty}f(x)   \mathrm{d}x-h\sum_{n=-\infty}^{+\infty}f\left(nh\right)
=-\sum_{n\not=0}\mathfrak{F}[f]\left(\frac{n}{h}\right)
 \end{align*}
it follows that $E(h)=0$ whenever $\frac{1}{h}>M$.
\end{remark}
\begin{example}\label{example4}
Let $f(x)=\frac{\sin(2\pi x)}{2\pi x}$ be the sinc function. Then its integral on the real line is Dirichlet integral $I=\int_{-\infty}^{\infty}\frac{\sin(2\pi x)}{2\pi x}{\mathrm{d}}x$, equal to $1/2$, and its quadrature by rectangle rule
\begin{align*}
I(h)=&\sum_{n=-\infty}^{\infty}\frac{\sin(2n\pi h)}{2n\pi }=h+\frac{1}{\pi}\sum_{n=1}^{\infty}\frac{\sin(n\cdot2\pi h)}{n}
=h+\frac{1}{\pi}St(2\pi h),
\end{align*}
where $St(x)$ is the $2\pi$-periodic sawtooth function with $St(x)=\frac{\pi-x}{2},\ x\in(0,2\pi)$ and $St(2n\pi)=0$ for integer $n$. Thus,
$I(h)=\frac{1}{2}=I$, that is, $E(h)=0$ holds for $0<h<1$, while for $h=\lfloor h\rfloor+(h)\ge 1$ it yields
$$
I(h)=h+\frac{1}{\pi}St(2\pi h)=h+\frac{1}{\pi}St(2\pi (h))=\left\{\begin{array}{ll}\frac{1}{2}+\lfloor{h}\rfloor,& \mbox{$h$ is not a positive integer}\\
h,&\mbox{$h$ is a positive integer}\end{array}\right.
\not=\frac{1}{2},
$$
 which is also validated directly by Remark \ref{remark4.1}, since its Fourier transform is a gate function
\begin{align*}
\mathfrak{F}[f](\xi)=\left\{\begin{array}{ll}
\frac{1}{2},&|\xi|< 1,\\
\frac{1}{4},&|\xi|=1,\\
0,&{\rm otherwise}.
\end{array}\right.
\end{align*}
This can also be demonstrated by Poisson summation formula \eqref{poisum}
$$
I-I(1)=-2\mathfrak{F}[f](1)=-\frac{1}{2},\quad I-I(2.5)=-2\mathfrak{F}[f]\left(\frac{1}{2.5}\right)-2\mathfrak{F}[f]\left(\frac{2}{2.5}\right)=-2.
$$

\end{example}

\section{Quadrature errors for the rational approximations}\label{sec:4}
The crucial point in the analysis of the  quadrature errors on $r^{(l)}_{N_t}(z)$ \eqref{eq:boundofr2} in \eqref{boundforEbar} and \eqref{boundforwidetildeE}  is to utilize Poisson summation formula (cf. \cite[(10.6-21)]{Henrici}, \cite[Theorem 1.3.1]{Stenger} and
\cite[Sect. 5]{Trefethen2014SIREV}) based upon Corollary \ref{PW} to estimate the quadrature errors $E_Q^{(l)}$ ($l=0,1$)  of the the composite rectangular rules for the integrals in \eqref{eq:intC1}.

Define for $l=0,1$ that
\begin{align}
f^{(l)}(u,z)=&\frac{zC^{\alpha}(u-T)^le^{u-T}}
{Ce^{\frac{1}{\alpha}(u-T)}+z}
\left(\prod_{k=1}^{\ell}\frac{z-s_k}{Ce^{\frac{1}{\alpha}(u-T)}+s_k}\right),\label{eq:func}\\
I^{(l)}(z)=&\int_0^{N_th}f^{(l)}(u,z)\mathrm{d}u,\quad h=\frac{\sigma\alpha}{\sqrt{N_1}},\label{eq:quadraturec}\\
E_Q^{(l)}(z)=&I^{(l)}(z)-h\sum_{k=0}^{N_t}f^{(l)}(kh,z)=I^{(l)}(z)- r^{(l)}_{N_t}(z).\label{eq:quadraqll}
\end{align}

In the following, we shall show that the quadrature errors satisfy  uniformly for $z\in S_\beta$  that
\begin{align}\label{eq:quaderrs}
E_Q^{(l)}(z)=\left\{\begin{array}{ll}
\mathcal{O}(T^le^{-T}),&\sigma\le \sigma_{\rm opt},\\
\mathcal{O}(e^{-\pi\eta\sqrt{(2-\beta)N\alpha}}),&\sigma> \sigma_{\rm opt},
\end{array}\right. \quad T=\sigma\alpha\sqrt{N_1},\quad l=0,1.
\end{align}

Along the way on
rectangular rule for integrals over the real line \cite[Sect. 5]{Trefethen2014SIREV},
it is decisive to introduce Poisson summation formula (cf. \cite[(10.6-21)]{Henrici} and \cite[Theorem 1.3.1]{Stenger}).

\begin{theorem}\cite[Theorem 1.3.1]{Stenger}\label{StengerPossionFormula}
Let $w\in L^2(\mathbb{R})$ and let $w$ and its Fourier transform
$\mathfrak{F}[w](\xi)=\int_{-\infty}^{\infty}w(u)e^{-2\pi iu\xi}\mathrm{d}u$ for $\xi$ and $u$ in $\mathbb{R}$, satisfy the
conditions
$$
w(u)=\lim_{t\rightarrow 0^+}\frac{w(u-t)+w(u+t)}{2},\quad \mathfrak{F}[w](\xi)=\lim_{t\rightarrow 0^+}\frac{\mathfrak{F}[w](\xi-t)+\mathfrak{F}[w](\xi+t)}{2}.
$$
Then, for all $\hbar > 0$,
\begin{align}\label{Possionsummationformula}
\hbar\sum_{n=-\infty}^{+\infty}w(n\hbar)e^{2\pi in\hbar u}= \sum_{n=-\infty}^{+\infty}\mathfrak{F}[w]\left(\frac{n}{\hbar}+u\right).
\end{align}
\end{theorem}

From \eqref{Possionsummationformula} with $u=0$ and by $\mathfrak{F}[w]\left(0\right)=\int_{-\infty}^{+\infty}w(u)\mathrm{d}u$, it follows
\begin{align}\label{QuadratureErrorfor_w}
E^{w}_{Q}:=\int_{-\infty}^{+\infty}w(u)\mathrm{d}u
-\hbar \sum_{j=-\infty}^{+\infty}w(j\hbar)
=-\sum_{n\not=0}\mathfrak{F}[w]\left(\frac{n}{\hbar}\right).
\end{align}

For $z=0$, the quadrature error $E_Q^{(l)}(z)$ \eqref{eq:quadraqll} satisfies  $E_Q^{(l)}(z)=0$.
In order to attain the exponential convergence rates on $T$ of the quadrature errors \eqref{eq:quaderrs} for $0\not=z \in S_\beta$,
we now characterize  the asymptotic decay rate of the Poisson summation formula on
\begin{align}\label{eq:ffft}
\mathfrak{F}[f^{(l)}]\left(\frac{n}{h}\right)=\int_{-\infty}^{+\infty}f^{(l)}(u,z)e^{-\frac{2n\pi i u}{h}} \mathrm{d}u
\end{align}
with $h=\frac{\sigma\alpha}{\sqrt{N_1}}$, and  analyze
the quadrature error $E_Q^{(l)}$ through systematic application of Corollary \ref{PW}.

\begin{theorem}\label{Quadratrue_rat_uniform}
Let $f^{(l)}(u,z)$ be defined in \eqref{eq:func} with $u\in\mathbb{R}$ and $0\not=z\in S_\beta$.
Then the summation of the discrete Fourier transforms \eqref{eq:ffft} for all $n\neq0$ decays at an exponential rate
\begin{align}\label{eq:conclusionOfFouriersum_uniform}
\sum_{n\ne0}\left|\mathfrak{F}[f^{(l)}]\left(\frac{n}{h}\right)\right|
=&\frac{\mathcal{G}^{\alpha}\max\{1,C^{\alpha}\}\mathcal{O}(1)}
{\varkappa(\beta)\big[e^{\frac{(2-\beta)\alpha\pi^2}{h}}-1\big]}
\left(1+\kappa^{l+1}\right),
\end{align}
where the constant $\mathcal{O}(1)$ term in \eqref{eq:conclusionOfFouriersum_uniform} is independent of $n$, $h$, $z$, $\alpha$ and $\sigma$, and $\mathcal{G}=\frac{\sqrt{2}+2}{\sqrt{2}-1}=8.24264068711928\ldots$ by setting $\delta=\frac{\sqrt{2}-1}{2}$.

In particular, for the case $\beta=0$, that is, $z\in [0,1]$,
the constant $\mathcal{G}$ in \eqref{eq:conclusionOfFouriersum_uniform} can be improved to $\frac{\sqrt{2}}{\sqrt{2}-1}=3.41421356237309\ldots$ with $\delta=\frac{\sqrt{2}-1}{2}$.
\end{theorem}
\begin{proof}
From the definition \eqref{eq:func} and inequalities in \eqref{ine:minus_u} and \eqref{ieq:positive_u}, it is easy to verify that
$f^{(l)}(\cdot,z)\in L^2(\mathbb{R})\bigcap C(\mathbb{R})$ for all fixed $z\in S_{\beta}$, since
\begin{align}\label{sqint}
\int_{-\infty}^{+\infty}\left|f^{(l)}(u,z)\right|^2\mathrm{d}u
\le&\frac{\mathbb{T}_{\ell,\beta}^{2}C^{2\alpha}}{\delta^{2\ell}\varkappa^2(\beta)}
\int_{-\infty}^{0}|t|^{2l}e^{2t}\mathrm{d}t
+\frac{\mathbb{T}_{\ell,\beta}^{2}C^{2\alpha}}{C^{2(\ell+1)}\varkappa^2(\beta)}
\int_{0}^{+\infty}\frac{t^{2l}}{e^{\frac{2t}{\kappa}}}\mathrm{d}t\notag\\
\le&\frac{\mathbb{T}_{\ell,\beta}^{2}C^{2\alpha}}{\varkappa^2(\beta)}
\left(\frac{1}{2^{l+1}\delta^{2\ell}}+\frac{\kappa^{2l+1}}{2^{l+1}C^{2\ell+2}}\right)<+\infty,\ l=0,1.
\end{align}
Then $\mathfrak{F}[f^{(l)}](\xi)$ is well-defined for $\xi\in \mathbb{R}$ \cite[p. 9]{Stenger}.

Moreover,
we can check readily analogous to \eqref{sqint} that
\begin{align*}
\int_{-\infty}^{+\infty}\left|f^{(l)}(u,z)\right|\mathrm{d}u
\le\frac{\mathbb{T}_{\ell,\beta}C^{\alpha}}{\varkappa(\beta)}
\left(\frac{1}{\delta^{\ell}}+\frac{\kappa^{l+1}}{C^{\ell+1}}\right)<+\infty,\ l=0,1
\end{align*}
holds uniformly for all $z\in S_\beta$, thus the Fourier transform
$$\mathfrak{F}\left[f^{(l)}(u,z)\right](\xi)=\int_{-\infty}^{+\infty}f^{(l)}(u,z)e^{-2\pi i\xi u}\mathrm{d}u$$
is continuous on $\mathbb{R}$ \cite[(10.6-12)-(10.6-13)]{Henrici}.
Therefore, the functions $f^{(l)}(u,z)$ and $\mathfrak{F}\left[f^{(l)}(u,z)\right]$ satisfy the conditions of Theorem \ref{StengerPossionFormula} and \cite[(10.6-12)-(10.6-13)]{Henrici}.

\bigskip
In the following, we first show that the integrand $f^{(l)}(u,z)$ also satisfies the conditions of Corollary \ref{PW} with $a_0=2\alpha\pi$.
To characterize the dependence of the decay rate of $\mathfrak{F}[f^{(l)}]\left(\frac{n}{h}\right)$ on $\alpha,\ \sigma$ and $\beta$, all the constants $B_k$ in the proof of Corollary \ref{PW}  are estimated in detail as follows.

We observe that for $z=xe^{\pm\frac{\theta\pi}{2}i}\not=0$ with $0\le \theta\le\beta<2$,
$f_{\alpha}^{(l)}(u,z^{\pm})=f_{\alpha}^{(l)}(u,xe^{\pm\frac{\theta\pi}{2}i})$ has the simple poles
\begin{equation}\label{eq:all_poles_fux_C}
u_k(z^{\pm})=T+\alpha\log{\frac{x}{C}}
+i\alpha\pi\left(2k-1\pm\frac{\theta}{2}\right),\quad k=0,\pm1,\ldots,
\end{equation}
and
\begin{equation}\label{eq:all_poles_fux_C_sl}
u_k(s_v)=T+\alpha\log{\frac{s_v}{C}}
+i\alpha\pi\left(2k-1\right),\quad k=0,\pm1,\ldots,\quad v=1,\ldots,\ell.
\end{equation}

Subsequently, we mainly focus on the case $z^+=xe^{i\frac{\theta\pi}{2}}$, and another case $z^-=xe^{-i\frac{\theta\pi}{2}}$
can be proven in the same manner.
Among the poles $\left\{u_k(z^{+})\right\}$ the first two closest to the real axis are
$u_0(z^+)=T+\alpha\log{\frac{x}{C}}
-i\alpha\pi\left(1-\frac{\theta}{2}\right)$ and $u_1(z^+)=T+\alpha\log{\frac{x}{C}}
+i\alpha\pi\left(1+\frac{\theta}{2}\right)$, and $u_0(s_v)$ and $u_1(s_v)$ in $\{u_k(s_v)\}$ are the closest to the real line and locate symmetrically.

In accordance with Corollary \ref{PW} , we may choose $a_0=2\alpha\pi$ such that
 $f^{(l)}(u,z)$ is holomorphic in the strip domain
$\left\{u\in\mathbb{C}:\ |\Im(u)|\le  a_0\right\}$
except for  the simple poles $u_0(z^+)=:u_{00}$, $u_1(z^+)=:u_{10}$, $u_0(s_k)=T+\alpha\log{\frac{s_k}{C}}
-i\alpha\pi=:u_{0k}$ and $u_1(s_k)=T+\alpha\log{\frac{s_k}{C}}
+i\alpha\pi=:u_{1k}$ for $k=1,2,\ldots,\ell$.
Thus, $f(u,z)$ is holomorphic
$$\left\{u\in\mathbb{C}:\ |\Im(u)|\le  a_0\right\}$$
except for the simple poles $\left\{u_{0k}\right\}_{k=0}^{\ell}$ and $\left\{u_{1k}\right\}_{k=0}^{\ell}$, respectively, and $$a:=|\Im(u_{00})|=\alpha\pi\left(1-\frac{\theta}{2}\right).$$

\begin{figure}[htp]\vspace{-.3cm}
\centerline{\includegraphics[width=15cm]{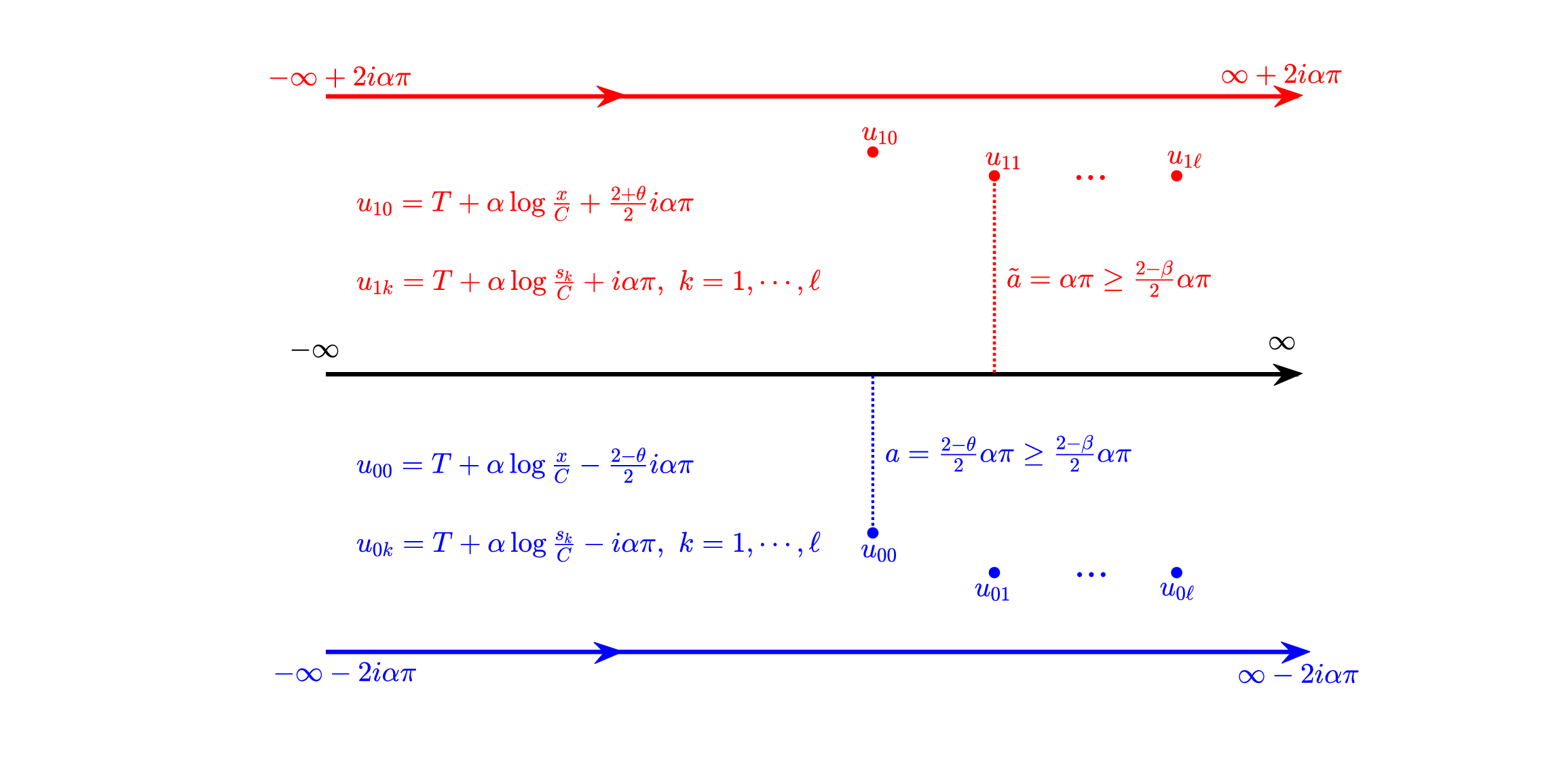}}\vspace{-.8cm}
\caption{The integrand $f^{(l)}(u,z)$ for $z=xe^{i\frac{\theta\pi}{2}}$, $0\le\theta\le\beta<2$ is holomorphic in the strip domain bounded by the horizontal  lines $\{z\in\mathbb{C}:\ \Im(z)=\mp a_0\}$ except for the simple poles $\left\{u_{0k}\right\}_{k=0}^{\ell}$ and $\left\{u_{1k}\right\}_{k=0}^{\ell}$ which are located in the lower and upper half-plane, respectively.}
\label{integral_contour_uniform}
\end{figure}

Particularly, for the line segments $u=-X\pm it$ and $X\pm it$ with $0\le t\le  a_0$,  and for $z\in S_{\beta}$ with $z\not=0$ and sufficiently large $X>0$, we have that
\allowdisplaybreaks
\begin{align}
\left|f^{(l)}(-X\pm it,z)\right|
=&\frac{|z||-X\pm it-T|^l\left|C^{\alpha}e^{-X\pm it-T}\right|}{\big|Ce^{\frac{1}{\alpha}(-X\pm it-T)}+z\big|}
\left|\left(\prod\limits_{k=1}^{\ell}\frac{z-s_k}{Ce^{\frac{1}{\alpha}(-X\pm it-T)}+s_k}\right)\right|\notag\\
\le& \frac{|z|(\sqrt{(X+T)^2+t^2})^lC^{\alpha}e^{-X-T}\mathbb{T}_{\ell,\beta}}{|z|-Ce^{\frac{1}{\alpha}(-X-T)}}
\left(\prod\limits_{k=1}^{\ell}\frac{1}{s_k-Ce^{\frac{1}{\alpha}(-X-T)}}\right)\label{eq:boundforfl-X}
\end{align}
and analogously by $s_\ell>|z|$ and $s_\ell\ge s_k$
\begin{align}
\allowdisplaybreaks
\left|f^{(l)}(X\pm it,z)\right|
\le&\frac{|z||X\pm it-T|^l\left|C^{\alpha}e^{X\pm it-T}\right|}{\big|Ce^{\frac{1}{\alpha}(X\pm it-T)}+z\big|}
\left|\prod\limits_{k=1}^{\ell}\frac{z-s_k}{Ce^{\frac{1}{\alpha}(X\pm it-T)}+s_k}\right|\notag\\
\le&\frac{|z|(\sqrt{(X-T)^2+t^2})^lC^{\alpha}e^{X-T}\mathbb{T}_{\ell,\beta}}{Ce^{\frac{1}{\alpha}(X-T)}-s_\ell}
\prod\limits_{k=1}^{\ell}\frac{1}{Ce^{\frac{1}{\alpha}(X-T)}-s_\ell}\notag\\
\le& \frac{|z|(\sqrt{(X-T)^2+t^2})^lC^{\alpha}\mathbb{T}_{\ell,\beta}}{\Big[Ce^{\frac{X-T}{(\ell+1)\kappa}}-s_\ell e^{-\frac{X-T}{\ell+1}}\Big]^{\ell+1}},\label{eq:boundforflx}
\end{align}
and both of \eqref{eq:boundforfl-X} and \eqref{eq:boundforflx} tend to zero as $X\rightarrow +\infty$ independent of $t$. Then \eqref{condition1_paley-wienner} in Theorem \ref{PW2}  is satisfied.

Furthermore, from \eqref{eq:est1} and \eqref{eq:est2} the integral of $f^{(l)}(u,z)$ over the lower and upper boundaries in \eqref{condition2_paley-wienner} can be bounded by
\allowdisplaybreaks
\begin{align}\label{boundforhorizontalline}
&B^{-\mathrm{sgn}(n)}_{0,\alpha,\sigma}
 :=\left|\int_{-\infty-i2\alpha\pi\mathrm{sgn}(n)}^{+\infty-i2\alpha\pi\mathrm{sgn}(n)}
  f^{(l)}(u,z)\mathrm{d}u\right|
  \le\int_{-\infty}^{+\infty}
  \left|f^{(l)}(t\mp i2\alpha\pi,z)\right|\mathrm{d}t\notag\\
  =&\int_{-\infty}^{+\infty}
  \Bigg|\frac{zC^{\alpha}\big(\sqrt{(t-T)^2+4\alpha^2\pi^2}\big)^l}
  {\big[Ce^{\frac{1}{\alpha}(t\mp2i\alpha\pi-T)}+z\big]e^{T-t\pm2i\alpha\pi}}
  \cdot
  \left(\prod\limits_{k=1}^{\ell}\frac{z-s_k}{Ce^{\frac{1}{\alpha}(t\mp2i\alpha\pi-T)}+s_k}\right)\Bigg|\mathrm{d}t\\
  =& \int_{-\infty}^{+\infty}
  \Bigg|\frac{xC^{\alpha}\big(\sqrt{(t-T)^2+4\alpha^2\pi^2}\big)^le^{t-T}}
  {Ce^{\frac{1}{\alpha}(t-T)}+z}\left(\prod\limits_{k=1}^{\ell}\frac{z-s_k}
  {Ce^{\frac{1}{\alpha}(t-T)}+s_k}\right)\Bigg|\mathrm{d}t\notag\\
  \le&  \int_{-\infty}^{T}\frac{\mathbb{T}_{\ell,\beta}C^{\alpha}}{\delta^{\ell}\varkappa(\beta)}\left(T-t
  +2\alpha\pi\right)^le^{t-T}\mathrm{d}t+\int_{T}^{+\infty}\frac{\mathbb{T}_{\ell,\beta}C^{\alpha-\ell-1}}{\varkappa(\beta)e^{\frac{1}{\kappa}(t-T)}}
  \left(t-T+2\alpha\pi\right)^l\mathrm{d}t\notag\\
  \le&\max\left\{\frac{C^{\alpha}}{\delta^{\ell}},1,\frac{1}{C^2}\right\}
  \frac{(2\alpha\pi+1)^l
  +\kappa(2\alpha\pi+\kappa)^l}{\varkappa(\beta)\mathbb{T}_{\ell,\beta}^{-1}}\notag\\
  =&\mathcal{O}(1)\max\left\{\frac{C^{\alpha}}{\delta^{\ell}},1\right\}
  \frac{(1+\kappa^{l+1})}{\varkappa(\beta)\mathbb{T}_{\ell,\beta}^{-1}}\notag
\end{align}
by  the definition of $\kappa=\frac{\alpha}{\ell+1-\alpha}$ with $\ell=\lceil\alpha\rceil$  for $l=1$.

From the above estimates we see that the integrand $f^{(l)}(u,z)$ satisfies the condition of Corollary \ref{PW}, and thus by \eqref{PossionSum} it follows that
\begin{align}\label{POISSION}
\sum_{n\not=0}\bigg|\mathfrak{F}[f^{(l)}]\left(\frac{n}{h}\right)\bigg|\le \frac{2B_{\alpha,\sigma}}{e^{\frac{2a\pi}{h}}-1},\quad
B_{\alpha,\sigma}=\max\left\{B^-_{\alpha,\sigma},\ B^+_{\alpha,\sigma}\right\}
\end{align}
with
\begin{align}
B^-_{\alpha,\sigma}=B^{-}_{0,\alpha,\sigma}+2\pi\sum_{k=0}^{\ell}\left|\mathrm{Res}[f^{(l)},u_{0k}]\right|,\quad
B^+_{\alpha,\sigma}=B^{+}_{0,\alpha,\sigma}+2\pi\sum_{k=0}^{\ell}\left|\mathrm{Res}[f^{(l)},u_{1k}]\right|,\label{constalphasigpositive}
\end{align}
in which $B^{-\mathrm{sgn}(n)}_{0,\alpha,\sigma}$ are bounded by \eqref{boundforhorizontalline}, and the residues can be estimated from \eqref{eq:all_poles_fux_C} and
\eqref{eq:all_poles_fux_C_sl} as follows by using $Ce^{\frac{1}{\alpha}(u_{00}-T)}=-z$
\allowdisplaybreaks
\begin{align}\label{asy00}
\left|\mathrm{Res}\left[f^{(l)}(u,z),u_{00}\right]\right|
=&\left|\lim_{u\rightarrow u_{00}}(u-u_{00})\frac{zC^{\alpha}(u-T)^le^{u-T}}
{Ce^{\frac{1}{\alpha}(u-T)}+z}
\left(\prod_{k=1}^{\ell}\frac{z-s_k}{Ce^{\frac{1}{\alpha}(u-T)}+s_k}\right)\right|\notag\\
=&C^{\alpha}\left|(u_{00}-T)^le^{u_{00}-T}\left(\prod_{k=1}^{\ell}\frac{z-s_k}{s_k-z}\right)\lim_{u\rightarrow u_{00}}\frac{u-u_{00}}
{e^{\frac{1}{\alpha}(u-u_{00})}-1}
\right|\\
=&\alpha x^{\alpha}\left|\alpha\log{\frac{x}{C}}-\frac{2-\theta}{2}i\alpha\pi\right|^l
=\mathcal{O}(1)\alpha [C_0(\alpha)]^l,\notag
\end{align}
where in the last identity in \eqref{asy00} we used
\begin{align*}
\max_{x\in (0,1]}\left|\alpha x^{\alpha}\log{\frac{x}{C}}\right|\le \left\{\begin{array}{ll}
\alpha|\log C|=\alpha\log C,&Ce^{-\frac{1}{\alpha}}\ge 1\\
\max\{\alpha|\log C|,e^{-1}C^\alpha\},&0<Ce^{-\frac{1}{\alpha}}< 1
\end{array}\right.=:C_0(\alpha).
\end{align*}

Similarly, by imposing  $Ce^{\frac{1}{\alpha}(u_{0k}-T)}=-s_k$, we obtain
\begin{align}\label{thelastellresidue}
\mathrm{Res}\left[f^{(l)}(u,z),u_{0k}\right]
=&-\alpha zs_k^{\alpha-1}\left(\alpha\log{\frac{s_k}{C}}-i\alpha\pi\right)^l
\prod_{\displaystyle \substack{v=1 \\ v \neq k}}^{\ell}\frac{z-s_v}{s_v-s_k}.
\end{align}
In particular, the summation of the residues in \eqref{thelastellresidue} satisfies
\begin{align}\label{resk}
\sum_{k=1}^{\ell}\mathrm{Res}\left[f^{(l)}(u,z),u_{0k}\right]
=&-\alpha z\sum_{k=1}^{\ell}s_k^{\alpha-2}\left(\alpha\log{\frac{s_k}{C}}-i\alpha\pi\right)^l
s_k\prod_{\displaystyle \substack{v=1 \\ v \neq k}}^{\ell}\frac{z-s_v}{s_v-s_k}.
\end{align}
By applying the following estimates $s_k^{\alpha-2}=s_k^{\ell-2}s_k^{\alpha-\ell}=\mathcal{O}(1) \max\{\delta^{\ell-2},(\delta+1)^{\ell-2}\}$ for $\delta\le s_k=\delta+\frac{1}{2}\left(1+\cos\frac{(2k-1)\pi}{2\ell}\right)\le \delta+1$ ($k=1,\ldots,\ell$), $\max\{\delta^{\ell-2},(\delta+1)^{\ell-2}\}=\mathcal{O}((\delta+1)^{\ell-2})$ for fixed $\delta>0$, and
$$\prod_{v=1}^{\ell}(1+s_v)\le \left[\frac{1}{\ell}\sum_{v=1}^{\ell}(1+s_v)\right]^\ell=\left(\delta+\frac{3}{2}\right)^{\ell},$$
$$\frac{s_k}{1+s_k}=\frac{\delta+\frac{1}{2}\left(1+\cos\frac{(2k-1)\pi}{2\ell}\right)}
{\delta+1+\frac{1}{2}\left(1+\cos\frac{(2k-1)\pi}{2\ell}\right)}\le\frac{\delta+1}{\delta+2} \le\frac{2\delta+2}{2\delta+3},
$$
and the identity
\begin{align}\label{representationofUz}
\prod_{\displaystyle \substack{k=1 \\ k \neq j}}^{\ell}(x_j-x_k)=\frac{1}{2^{\ell-1}}\frac{\mathrm{d}}{\mathrm{d}t}T_\ell(t)\bigg|_{t=x_j}=\frac{\ell}{2^{\ell-1}}\frac{\sin\frac{\ell(2j+1)\pi}{2\ell}}{\sin\frac{(2j+1)\pi}{2\ell}}
=
\frac{(-1)^j\ell}{2^{\ell-1}\sin\frac{(2j+1)\pi}{2\ell}}
\end{align}
for Chebyshev points $x_j=\cos\frac{(2j-1)\pi}{2m}$ ($j=1,\ldots,\ell$) from Mason and Handscomb \cite[Section 2.2]{Mason2003},
 the terms in \eqref{resk} can be bounded respectively by
\begin{align*}
\max_{k=1,2,\ldots,\ell}\left|\alpha z\left(\alpha\log{\frac{s_k}{C}}-i\alpha\pi\right)^ls_k^{\alpha-2}\right|
=\alpha^{l+1}(\delta+1)^{\ell-2}\mathcal{O}(1)
\end{align*}
and
\begin{align}\label{boundforproduct22}
\left\|s_k\prod_{\displaystyle \substack{v=1\\v\neq k}}^{\ell}\frac{z-s_v}{s_v-s_k}\right\|_{C(S_{\beta})}
\le& \frac{s_k}{1+s_k}\prod_{v=1}^\ell(1+s_v)\prod_{\displaystyle \substack{v=1\\v\neq k}}^{\ell}\frac{1}{s_v-s_k}\notag\\
\le& \frac{2\delta+2}{2\delta+3}\left(\delta+\frac{3}{2}\right)^{\ell}\frac{2^{2\ell-2}\big|\sin\frac{(2k+1)\pi}{2\ell}\big|}{\ell}\\
\le&\frac{1}{\ell}2^{\ell-1}\left(\delta+1\right)(2\delta+3)^{\ell-1}.\notag
\end{align}
Thus, we have
\begin{align}\label{asy0k}
\sum_{k=1}^{\ell}\left|\mathrm{Res}\left[f^{(l)}(u,z),u_{0k}\right]\right|
=\alpha^{l+1}(4\delta+6)^{\ell}(\delta+1)^{\ell}\mathcal{O}(1),
\end{align}
where the constants in $\mathcal{O}(1)$ terms \eqref{asy00} and \eqref{asy0k}  are independent of $n$, $\alpha$, $h$ and $z$, and $\ell$.

Substituting \eqref{boundforhorizontalline}, \eqref{asy00} and \eqref{asy0k} into \eqref{constalphasigpositive},
the constant $B^-_{\alpha,\sigma}$ can be evaluated by
\begin{align}\label{constantB-}
  B^-_{\alpha,\sigma}
  =&\mathcal{O}(1)\max\left\{\left[\frac{(2\delta+3)C}{2\delta}\right]^{\alpha},\frac{(2\delta+3)^{\alpha}}{2^{\alpha}}\right\}
  \frac{1
  +\kappa^{l+1}}{\varkappa(\beta)}\notag\\
  &+\mathcal{O}(1)\alpha [C_0(\alpha)]^l
  +\alpha^{l+1}(4\delta+6)^{\ell}(\delta+1)^{\ell}\mathcal{O}(1)\\
  =&\mathcal{O}(1)\max\left\{\left[\frac{(2\delta+3)C}{2\delta}\right]^{\alpha},\frac{(2\delta+3)^{\alpha}}{2^{\alpha}}\right\}
  \frac{1
  +\kappa^{l+1}}{\varkappa(\beta)}\notag\\
  &+\alpha^{l+1}(4\delta+6)^{\ell}(\delta+1)^{\ell}\mathcal{O}(1)\notag
\end{align}
 and using the estimate
\begin{align}\label{boundofTell}
\mathbb{T}_{\ell,\beta}=&\max_{z\in S_{\beta}}\prod_{k=1}^{\ell}|z-s_k|
\le\prod_{k=1}^{\ell}(1+s_k)
\le\left(\frac{2\delta+3}{2}\right)^{\ell},
\end{align}
where the term $\mathcal{O}(1)\alpha [C_0(\alpha)]^l$ in \eqref{constantB-} is absorbed in the first term in \eqref{constantB-} by the definition of $C_0(\alpha)$ and $\kappa$.

To balance the first and second terms in \eqref{constantB-}, we choose $\delta=\frac{\sqrt{2}-1}{2}$ such that $\left(\frac{2\delta+3}{2\delta}\right)^{\ell}=(4\delta+6)^{\ell}(\delta+1)^\ell:=\mathcal{G}^\ell$, therefore
\begin{align}\label{positive_DFT_decay_rat}
B^-_{\alpha,\sigma}
=\frac{\mathcal{G}^{\alpha}\max\{1,C^{\alpha}\}\mathcal{O}(1)}{\varkappa(\beta)}
\left(1+\kappa^{l+1}\right),
\end{align}
with $\mathcal{G}=\frac{\sqrt{2}+2}{\sqrt{2}-1}=8.24264068711928\ldots$.

Similarly, we can prove that \eqref{positive_DFT_decay_rat}
also holds for $B^+_{\alpha,\sigma}$ in \eqref{constalphasigpositive}. These together with \eqref{POISSION} and $a=\frac{(2-\theta)\alpha\pi}{2}\ge\frac{(2-\beta)\alpha\pi}{2}$ for arbitrary $z=xe^{\pm\frac{\theta\pi}{2}i}\not=0$ in $S_\beta$ establish that
\begin{align*}
\sum_{n\ne0}\mathfrak{F}[f^{(l)}]\Big{(}\frac{n}{h}\Big{)}
=&\frac{\mathcal{G}^{\alpha}\max\{1,C^{\alpha}\}\mathcal{O}(1)}
{\varkappa(\beta)\big[e^{\frac{(2-\beta)\alpha\pi^2}{h}}-1\big]}
\left(1+\kappa^{l+1}\right).
\end{align*}
It is clear that all the constants in $\mathcal{O}(1)$s in this proof are independent of $n$, $h$, $z$, $\alpha$ and $\sigma$.

Particularly, from \eqref{representationofUz}
the bound \eqref{boundforproduct22} for the case $\beta=0$ can be sharpened as
\begin{align*}
\bigg\|s_k\prod_{\displaystyle \substack{v=1\\v\neq k}}^{\ell}\frac{z-s_v}{s_v-s_k}\bigg\|_{C(S_{\beta})}
\le \prod_{v=1}^\ell s_v\prod_{\displaystyle \substack{v=1\\v\neq k}}^{\ell}\frac{1}{s_v-s_k}
\le 2^{\ell-2}\left(2\delta+1\right)^{\ell}.
\end{align*}
by the monotonicity of the second kind Chebyshev polynomial $U_{\ell-1}\left(2x-2\delta-1\right)$ outside of $\left[\delta,\delta+1\right]$ and its extreme point $\delta> 0$.

Analogously, the bound on $\mathbb{T}_{\ell,\beta}$ for the case $\beta=0$ can also be sharpened to
\begin{align}\label{boundofTell2}
\mathbb{T}_{\ell,\beta}
\le\prod_{k=1}^{\ell}s_k
\le\left(\frac{2\delta+1}{2}\right)^{\ell}.
\end{align}
Thus, the factor $(2\delta+3)^{\ell-1}$ in \eqref{asy0k} and \eqref{constantB-} may be shrunk to $(2\delta+1)^{\ell}$, which implies that the constant $\mathcal{G}=\frac{\sqrt{2}+2}{\sqrt{2}-1}$ in \eqref{positive_DFT_decay_rat} is improved to $\mathcal{G}=\frac{\sqrt{2}}{\sqrt{2}-1}=3.414213562373$ $09\ldots$ for the case $\beta=0$ with $\delta=\frac{\sqrt{2}-1}{2}$.

These together complete the proof of Theorem \ref{Quadratrue_rat_uniform}.
\end{proof}

Now by the Poisson summation formula \eqref{Quadratrue_rat_uniform}, the integral  \eqref{eq:intC1} and the rational approximation \eqref{eq:ECrat1CC}, the  quadrature error $E^{(l)}_{Q}(z)$ \eqref{eq:quadraqll} can be estimated uniformly by
\begin{align}\label{quadratureOfbarf}
E^{(l)}_{Q}(z)=&\int_{0}^{N_tT}f^{(l)}(u,z)\mathrm{d}u-r^{(l)}_{N_t}(z)\notag\\
=&\int_{-\infty}^{+\infty}f^{(l)}(u,z)\mathrm{d}u-E^{(l)}_T(z)-h\sum_{n=-\infty}^{+\infty}f^{(l)}(nh,z)\\
&+h\left(\sum_{n=-\infty}^{-1}+\sum_{n=N_t+1}^{+\infty}\right)f^{(l)}(nh,z)\notag\\
=&-\sum_{n\ne0}\mathfrak{F}[f^{(l)}]\big{(}\frac{n}{h}\big{)}-E^{(l)}_T(z)
+h\left(\sum_{n=-\infty}^{-1}+\sum_{n=N_t+1}^{+\infty}\right)f^{(l)}(nh,z)\notag\\
=&-\sum_{n\ne0}\mathfrak{F}[f^{(l)}]\big{(}\frac{n}{h}\big{)}-E^{(l)}_T(z)
+\mathcal{G}^{\alpha}\max\{1,C^{\alpha}\}(1+T)^l\left(1+\kappa^{l+1}\right)
\frac{\mathcal{O}(e^{-T})}{\varkappa(\beta)},\notag
\end{align}
with $\mathcal{G}=\frac{\sqrt{2}+2}{\sqrt{2}-1}$ for $\beta\in(0,2)$ and $\mathcal{G}=\frac{\sqrt{2}}{\sqrt{2}-1}$ for $\beta=0$,
where  we used the following estimate in the last identity in \eqref{quadratureOfbarf}
\begin{align*}
&\left|h\left(\sum_{n=-\infty}^{-1}+\sum_{n=N_t+1}^{+\infty}\right)f^{(l)}(nh,z)\right|\\
\le & h\left(\sum_{n=-\infty}^{-1}+\sum_{n=N_t+1}^{+\infty}\right)\frac{|z||nh-T|^lC^{\alpha}e^{nh-T}}
{\big|Ce^{\frac{1}{\alpha}(nh-T)}+z\big|}
\left|\prod\limits_{k=1}^{\ell}\frac{z-s_k}{Ce^{\frac{1}{\alpha}(nh-T)}+s_k}\right|\\
\le & h\sum_{n=-\infty}^{-1}\frac{\mathbb{T}_{\ell,\beta}|nh-T|^lC^{\alpha}e^{nh-T}}{\delta^{\ell}\varkappa(\beta)}
+h\sum_{n=N_t+1}^{+\infty}\frac{\mathbb{T}_{\ell,\beta}|nh-T|^le^{-\frac{1}{\kappa}(nh-T)}}
{\varkappa(\beta)C^{\ell+1-\alpha}}\\
\le& \int_{-\infty}^{-T}\frac{\mathbb{T}_{\ell,\beta}|t|^lC^{\alpha}e^{t}}{\delta^{\ell}\varkappa(\beta)}\mathrm{d}t
+\int_{\kappa T}^{+\infty}\frac{\mathbb{T}_{\ell,\beta}|t|^lC^{\alpha}e^{-\frac{t}{\kappa}}}
{C^{\ell+1-\alpha}\varkappa(\beta)}\mathrm{d}t\\
=&\frac{\mathbb{T}_{\ell,\beta}(T+1)^{l}C^{\alpha}}{\delta^{\ell}\varkappa(\beta)}\mathcal{O}(e^{-T})
+\frac{\mathbb{T}_{\ell,\beta}\kappa(\kappa T+\kappa)^l}{\varkappa(\beta)}\mathcal{O}(e^{-T})\\
=&\mathcal{G}^{\alpha}\max\{1,C^{\alpha}\}(1+T)^l\left(1+\kappa^{l+1}\right)
\frac{\mathcal{O}(e^{-T})}{\varkappa(\beta)}
\end{align*}
according to \eqref{eq:inequ_neg}, \eqref{eq:inequ_pos}, $N_th\ge (\kappa+1)T$ and by the monotonicities of $|t|^le^{t}$ for $t\le -T$ and $t^le^{-\frac{1}{\kappa}t}$ for $t\ge \kappa T$ and $T\ge 2$.

Hence, from \eqref{quadratureOfbarf} and \eqref{Quadratrue_rat_uniform} we have
\begin{corollary}\label{eq:quaderror}
\begin{align}\label{errquad_uniform}
E^{(l)}_Q(z)
=&\frac{\mathcal{G}^{\alpha}\max\{1,C^{\alpha}\}\mathcal{O}(1)}
{\varkappa(\beta)\left[e^{\frac{(2-\beta)\alpha\pi^2}{h}}-1\right]}
\left(1+\kappa^{l+1}\right)-E^{(l)}_T(z)\notag\\
&+\mathcal{G}^{\alpha}\max\{1,C^{\alpha}\}(1+T)^l\left(1+\kappa^{l+1}\right)
\frac{\mathcal{O}(e^{-T})}{\varkappa(\beta)},
\end{align}
and  all the constants in $\mathcal{O}$ terms are independent of $N_1$, $\alpha$, $\sigma$ and $z\in S_\beta$.
 \end{corollary}

\section{Proof of Theorem \ref{mainthm}}\label{sec:5}
Combining the uniform bounds in \eqref{errquad_uniform} for $l=0,1$
with the estimates in \eqref{boundforEbar} and \eqref{boundforwidetildeE} yields the proof of Theorem \ref{mainthm}.

{\bf Proof of  LP \eqref{eq:rat} with $N_2=\mathcal{O}(\sqrt{N_1})$ for $z^\alpha$}:
From \eqref{ratappforzalpha} and \eqref{boundforEbar}, together with \eqref{errquad_uniform}, it follows by setting $l=0$, $\ell=\lfloor\alpha\rfloor$, $\delta=\frac{\sqrt{2}-1}{2}$ and $\delta+2<\mathcal{G}$, and applying
\begin{align*}
\bigg|\frac{\sin(\alpha\pi)}{(-1)^{\ell}\alpha\pi}\kappa\bigg|=\frac{|\sin((\ell+1-\alpha)\pi)|}{(\ell+1-\alpha)\pi}\le 1, \quad \mathbb{T}_{\ell,\beta}
\le\left(\frac{2\delta+3}{2}\right)^{\ell}\le (\delta+2)^{\ell}<\mathcal{G}^\ell,\\
\frac{C^{\alpha}}{\delta^{\alpha}}\mathbb{T}_{\ell,\beta}\le \frac{C^{\alpha}}{\delta^{\alpha}}\left(\frac{2\delta+3}{2}\right)^{\alpha}\le\mathcal{G}^{\alpha}C^{\alpha}\le \mathcal{G}^{\alpha}\max\{C^\alpha,1\}, \quad C^{\ell+1-\alpha}=\mathcal{O}(1),
\end{align*}
 and using
\begin{align*}
\frac{|\sin(\alpha\pi)|}{\alpha\pi}\left|E_Q^{(0)}(z)\right|
=&\frac{|\sin(\alpha\pi)|}{\alpha\pi}\Bigg|
\frac{\mathcal{G}^{\alpha}\max\{1,C^{\alpha}\}\mathcal{O}(1)}
{\varkappa(\beta)\left[e^{\frac{(2-\beta)\alpha\pi^2}{h}}-1\right]}
(1+\kappa)-E^{(0)}_T(z)\\
&+\mathcal{G}^{\alpha}\max\{1,C^{\alpha}\}(1+\kappa)\frac{\mathcal{O}(e^{-T})}{\varkappa(\beta)}\Bigg|\\
=&\frac{\mathcal{G}^{\alpha}\max\{1,C^{\alpha}\}}{\varkappa(\beta)}
\left[\frac{\mathcal{O}(1)}{e^{\frac{(2-\beta)\pi^2\alpha}{h}}-1}+\mathcal{O}(1)e^{-T}\right],
\end{align*}
 that
\begin{align}\label{rat01}
|\bar{E}(z)|=&|z^\alpha-\bar{r}_N(z)|\notag\\
\le&\frac{\mathcal{G}^{\alpha}\max\{1,C^{\alpha}\}}{\varkappa(\beta)}\mathcal{O}(e^{-T})
+\frac{(\delta+2)^{\ell}\mathcal{O}(e^{-T})}{\varkappa(\beta)}
+\frac{|\sin(\alpha\pi)|}{\alpha\pi}\left|E_Q^{(0)}(z)\right|\\
=&\frac{\mathcal{G}^{\alpha}\max\{1,C^{\alpha}\}}{\varkappa(\beta)}
\left[\frac{\mathcal{O}(1)}{e^{\frac{(2-\beta)\pi^2\alpha}{h}}-1}+\mathcal{O}(1)e^{-T}\right]\notag\\
=&\frac{\mathcal{G}^{\alpha}\max\{1,C^{\alpha}\}}{\varkappa(\beta)}
\left[\frac{\mathcal{O}(1)}{e^{\frac{(2-\beta)\pi^2}{\sigma}\sqrt{N_1}}-1}
+\mathcal{O}(1)e^{-\alpha\sigma\sqrt{N_1}}\right]\notag\\
=&\frac{\mathcal{G}^{\alpha}\max\{1,C^{\alpha}\}}{\varkappa(\beta)}\left\{\begin{array}{ll}
\mathcal{O}(e^{-\sigma\alpha\sqrt{N_1}}),&\sigma\le \sigma_{\rm opt}\\
\frac{\mathcal{O}(1)}{e^{\pi\eta\sqrt{(2-\beta)N_1\alpha}}-1},&\sigma> \sigma_{\rm opt}
\end{array}\right.\notag\\
=&\frac{\mathcal{G}^{\alpha}\max\{1,C^{\alpha}\}}{\varkappa(\beta)}\left\{\begin{array}{ll}
\mathcal{O}(e^{-\sigma\alpha\sqrt{N}}),&\sigma\le \sigma_{\rm opt},\\
\frac{\mathcal{O}(1)}{e^{\pi\eta\sqrt{(2-\beta)N\alpha}}-1},&\sigma> \sigma_{\rm opt},
\end{array}\right.\notag
\end{align}
where we used in the last equation of \eqref{rat01} the fact
 \begin{align*}
 \sqrt{N_1}=\sqrt{N-N_2+1}=\sqrt{N}\left[1+\mathcal{O}\left(\frac{N_2}{N}\right)\right]=\sqrt{N}+\mathcal{O}(1)
 =:\sqrt{N}+c_N
 \end{align*}
with $c_N(<0)$ being uniformly bounded and independent of $N$.

\bigskip
{\bf Proof of  LP \eqref{eq:rat} with $N_2=\mathcal{O}(\sqrt{N_1})$ for $z^\alpha\log z$}: Analogously to the case of $z^\alpha$, from \eqref{ratappforzalphalog}, \eqref{boundforwidetildeE} together with \eqref{errquad_uniform}, we obtain by letting $\ell=\lceil\alpha\rceil$ that
\allowdisplaybreaks
\begin{align}\label{rat01log}
|\widetilde{E}(z)|=&|z^\alpha\log z-\widetilde{r}_N(z)|
\le\frac{(\alpha+1)\mathcal{G}^{\alpha}\max\{1,C^{\alpha}\}}{\alpha\varkappa(\beta)}\mathcal{O}(Te^{-T})
+\frac{(\delta+2)^{\ell}}{\varkappa(\beta)}\mathcal{O}(Te^{-T})\notag\\
&+\frac{|\sin(\alpha\pi)|}{\alpha^2\pi}\left|E_Q^{(1)}(z)\right|
+\left|\frac{\sin(\alpha\pi)\log{C}}{(-1)^\ell\alpha\pi}
+\frac{\cos(\alpha\pi)}{(-1)^{\ell}\alpha}\right|\left|E_Q^{(0)}(z)\right|\\
=&\frac{(\alpha+1)\mathcal{G}^{\alpha}\max\{1,C^{\alpha}\}}{\alpha\varkappa(\beta)}\left\{\begin{array}{ll}
\mathcal{O}(\sigma\sqrt{N_1}e^{-\sigma\alpha\sqrt{N_1}}),&\sigma\le \sigma_{\rm opt},\\
\frac{\mathcal{O}(1)}{e^{\pi\eta\sqrt{(2-\beta)N_1\alpha}}-1},&\sigma> \sigma_{\rm opt},
\end{array}\right.\notag\\
=&\frac{(\alpha+1)\mathcal{G}^{\alpha}\max\{1,C^{\alpha}\}}{\alpha\varkappa(\beta)}\left\{\begin{array}{ll}
\mathcal{O}(\sigma\sqrt{N}e^{-\sigma\alpha\sqrt{N}}),&\sigma\le \sigma_{\rm opt},\\
\frac{\mathcal{O}(1)}{e^{\pi\eta\sqrt{(2-\beta)N\alpha}}-1},&\sigma> \sigma_{\rm opt},\notag
\end{array}\right.
\end{align}
where we used the fact that
\begin{align*}
\frac{|\sin(\alpha\pi)|}{\alpha^2\pi}\left|E_Q^{(1)}(z)\right|
=&\frac{|\sin(\alpha\pi)|}{\alpha^2\pi}\Bigg|
\frac{\mathcal{G}^{\alpha}\max\{1,C^{\alpha}\}\mathcal{O}(1)}
{\varkappa(\beta)\big[e^{\frac{(2-\beta)\alpha\pi^2}{h}}-1\big]}
\left[\alpha^2+\kappa(2\alpha\pi+\kappa)\right]-E^{(1)}_T(z)\\
&+\mathcal{G}^{\alpha}\max\{1,C^{\alpha}\}(1+T)\left(1+\kappa^2\right)
\frac{\mathcal{O}(e^{-T})}{\varkappa(\beta)}\Bigg|\\
=&\frac{\mathcal{G}^{\alpha}\max\{1,C^{\alpha}\}\mathcal{O}(1)}
{\varkappa(\beta)\big(e^{\frac{(2-\beta)\pi^2\alpha}{h}}-1\big)}
+\frac{(\alpha+1)\mathcal{G}^{\alpha}\max\{1,C^{\alpha}\}}{\alpha\varkappa(\beta)}\mathcal{O}(Te^{-T})\\
=&\frac{(\alpha+1)\mathcal{G}^{\alpha}\max\{1,C^{\alpha}\}}{\alpha\varkappa(\beta)}
\left[\frac{\mathcal{O}(1)}{e^{\pi\eta\sqrt{(2-\beta)N\alpha}}-1}+\frac{\sigma\mathcal{O}(1) \sqrt{N_1}}{e^{\alpha\sigma\sqrt{N_1}}}\right]
\end{align*}
and similarly,
\allowdisplaybreaks
\begin{align*}
&\left|\frac{\sin(\alpha\pi)\log{C}}{(-1)^\ell\alpha\pi}
+\frac{\cos(\alpha\pi)}{(-1)^{\ell}\alpha}\right|\left|E_Q^{(0)}(z)\right|\\
=&\left|\frac{\sin(\alpha\pi)\log{C}}{(-1)^\ell\alpha\pi}
+\frac{\cos(\alpha\pi)}{(-1)^{\ell}\alpha}\right|\Bigg|
\frac{\mathcal{G}^{\alpha}\max\{1,C^{\alpha}\}\mathcal{O}(1)}
{\varkappa(\beta)\big[e^{\frac{(2-\beta)\alpha\pi^2}{h}}-1\big]}
(\alpha+\kappa)-E^{(0)}_T(z)\\
&+\mathcal{G}^{\alpha}\max\{1,C^{\alpha}\}(1+\kappa)\frac{\mathcal{O}(e^{-T})}{\varkappa(\beta)}\Bigg|\\
=&\frac{\mathcal{G}^{\alpha}\max\{1,C^{\alpha}\}\mathcal{O}(1)}
{\varkappa(\beta)\big[e^{\frac{(2-\beta)\alpha\pi^2}{h}}-1\big]}
+\frac{(\alpha+1)\mathcal{G}^{\alpha}\max\{1,C^{\alpha}\}}{\alpha\varkappa(\beta)}\mathcal{O}(e^{-T})\\
=&\frac{(\alpha+1)\mathcal{G}^{\alpha}\max\{1,C^{\alpha}\}}{\alpha\varkappa(\beta)}
\left(\frac{\mathcal{O}(1)}{e^{\pi\eta\sqrt{\alpha(2-\beta)N_1}}-1}
+\mathcal{O}(1)e^{-\alpha\sigma\sqrt{N_1}}\right).
\end{align*}

Thus, we arrive at the conclusions for the special case of $g(z)=1$
uniformly for $z\in S_\beta$.

\bigskip
{\bf Proof of  LPs \eqref{eq:rat} with $N_2=\mathcal{O}(\sqrt{N_1})$ for $z^\alpha g(z)$ and $z^\alpha g(z)\log z$}:
Note that $g(z)$ can be approximated simply by a polynomial $P^{(g)}_{N_2}(z)$ of degree $N_2=\mathcal{O}(\sqrt{N_1})$ satisfying $\|g(z)-P^{(g)}_{N_2}(z)\|_{C(S_{\beta})}=\mathcal{O}(e^{-T})$ (see Subsection \ref{sec3-4}). Both of $\bar{r}_N$ and $\widetilde{r}_N$ are uniformly bounded on $S_\beta$ from \eqref{rat01} and \eqref{rat01log}. Then we construct \eqref{eq: rate1} and \eqref{eq: rate2} by \eqref{LPg(z)zalpha} and \eqref{LPg(z)zalphalogz} that
\begin{align}
\left|g(z)z^\alpha-P^{(g)}_{N_2}(z)\bar{r}_N(z)\right|
&\le\|g(z)(z^\alpha-\bar{r}_N(z))\|_{C(S_{\beta})}+|\bar{r}_{N}(z)|\mathcal{O}(e^{-T}),\label{eq:errorforg(z)zalpha}\\
|g(z)z^\alpha\log{z}-P^{(g)}_{N_2}(z)\widetilde{r}_N(z)|
&\le\|g(z)(z^\alpha\log{z}-\widetilde{r}_N(z))\|_{C(S_{\beta})}+|\widetilde{r}_{N}(z)|\mathcal{O}(e^{-T}),\label{eq:errorforg(z)zalphalogz}
\end{align}
which directly leads to the desired results \eqref{eq: rate1} and \eqref{eq: rate2} in Theorem \ref{mainthm}.

In particular, for the special case where $\alpha$ is a positive integer, $g(z)z^\alpha$ can be approximated by a polynomial $P^{(g)}_{N_2}(z)$ of degree $N_2=\mathcal{O}(\sqrt{N_1})$ from Runge's theorem and the discussion in Subsection \ref{sec3-4}. For $z^\alpha\log z$, from the integral representation \eqref{eq:cint_log_gener}, the first term in \eqref{eq:cint_log_gener} vanishes, improving \eqref{eq: rate1} to \eqref{eq: rate2}.

These results collectively establish the proof of Theorem \ref{mainthm}.

\begin{remark}\label{rungeN}
 It is notable that
choosing $N_2=\mathcal{O}(\sqrt{N_1})$ is necessary according to Runge's approximation theorem (see Subsection \ref{sec3-4} for more details). For $\sqrt{N_1}/N_2\sim o(1)$, i.e.,  a larger $N_2$, the LP  theoretically gives the same convergence order as $N_2=\mathcal{O}(\sqrt{N_1})$ for $g(z)z^{\alpha}$ and $g(z)z^{\alpha}\log{z}$, but may create numerical instability. While for $N_2/\sqrt{N_1}\sim o(1)$, i.e., a smaller $N_2$, the LP cannot achieve the desired rate. See {\sc Figure} \ref{LPsLargerN2} for example. Following Herremans, Huybrechs and Trefethen \cite{Herremans2023}, for most cases in this paper the constant in front of $\sqrt{N_1}$ is chosen as $1.3$, that is, $N_2={\rm ceil}(1.3\sqrt{N_1})$.

\begin{figure}[htbp]
\centerline{\includegraphics[width=12.2cm]{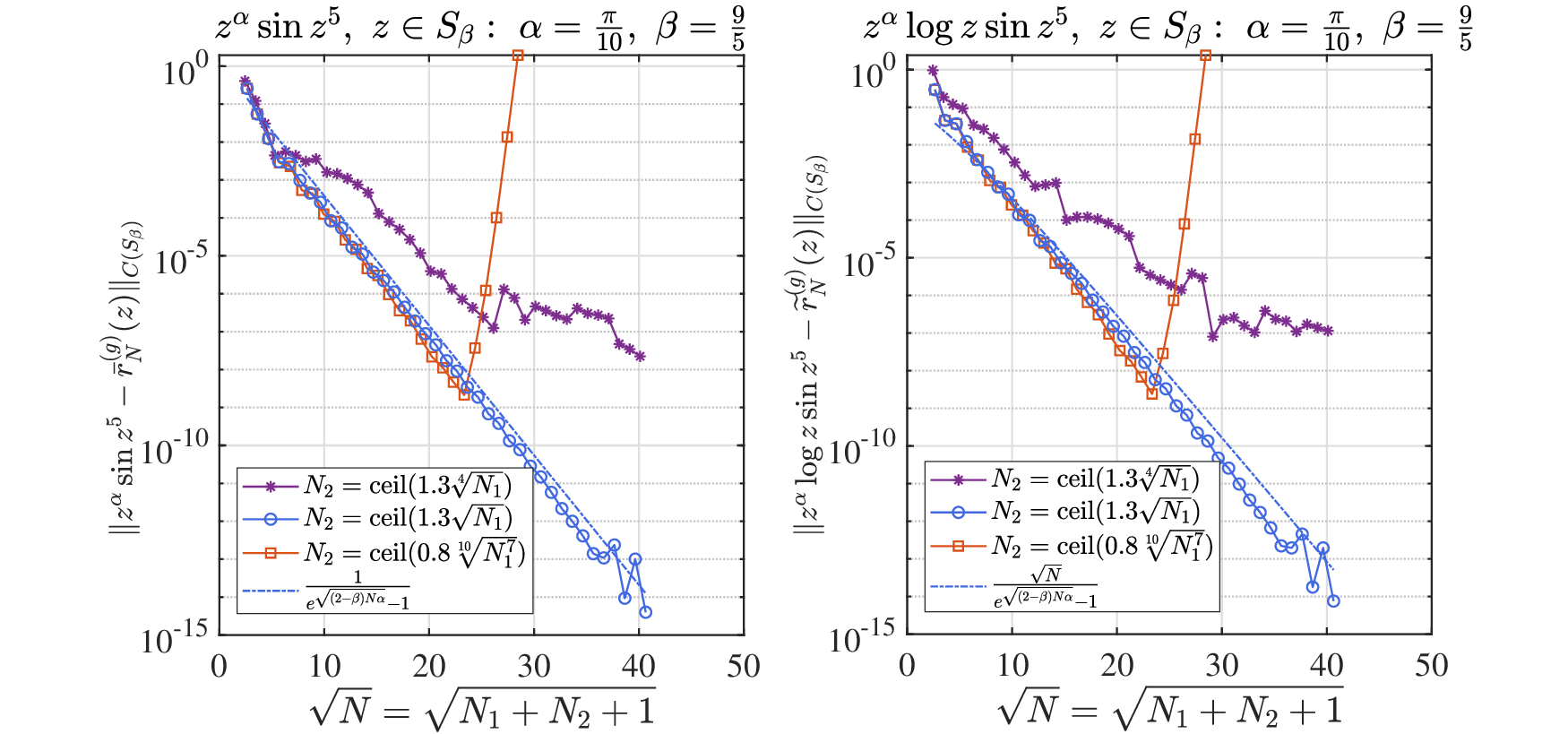}}
\caption{The comparisons of decay rates of errors of LPs for $g(z)z^{\alpha}$ and $g(z)z^{\alpha}\log{z}$ in $S_{\beta}$ with $N_2={\rm ceil}(1.3\sqrt[4]{N_1})$, ${\rm ceil}(1.3\sqrt{N_1})$ and ${\rm ceil}(0.8\sqrt[10]{N_1^7})$, respectively, where $g(z)=\sin{z^5}$.}\label{LPsLargerN2}
\end{figure}
\end{remark}

\begin{remark}
From the proof of Theorem \ref{mainthm}, we see that the convergence rates are  determined by  uniform bounds of the truncated errors $E_T^{(l)}(z)$ and $E_Q^{(l)}(z)$ ($l=0,1$) for $z\in S_\beta$, wherein both $E_T^{(l)}(z)$ and $E_T^{(l)}(z)$ attain their upper bounds at segments $z=xe^{i\frac{\beta\pi}{2}}$  and $z=xe^{-i\frac{\beta\pi}{2}}$ for $x\in (0,1]$ related to $\varkappa(\beta)$ and $B^{\mp\mathrm{sgn}(n)}_{\alpha,\sigma}$, respectively. See \eqref{eq:est1}, \eqref{eq:est2}, \eqref{constantB-} and  the proof of Theorem \ref{Quadratrue_rat_uniform}. Then if the poles \eqref{eq:uniform0} are not clustered on the exterior angle bisector of sector domain $S_{\frac{\beta+\theta}{2}}^{(j)}\subseteq S_{\beta}$ ($0\le \theta\le \beta$ and $j=1,2$), we may symmetrically extend the sector to $S_\beta$ such that the poles on the exterior angle bisector (see 1st-row of {\sc Figure} \ref{LPwithuniformclusterpole_rotated}), Theorem \ref{mainthm} on $S_{\frac{\beta+\theta}{2}}^{(j)}$ still holds 
 for $\sigma=\sigma_{\rm opt}$.
{\sc Figure} \ref{LPwithuniformclusterpole_rotated} demonstrates the optimal convergence rate
$\mathcal{O}\big(e^{-\sqrt{(2-\beta)\alpha}\pi}\big)$ with $\sigma=\sigma_{\rm opt}$.
\end{remark}

\begin{figure}[htbp]
\centerline{\includegraphics[width=16cm]{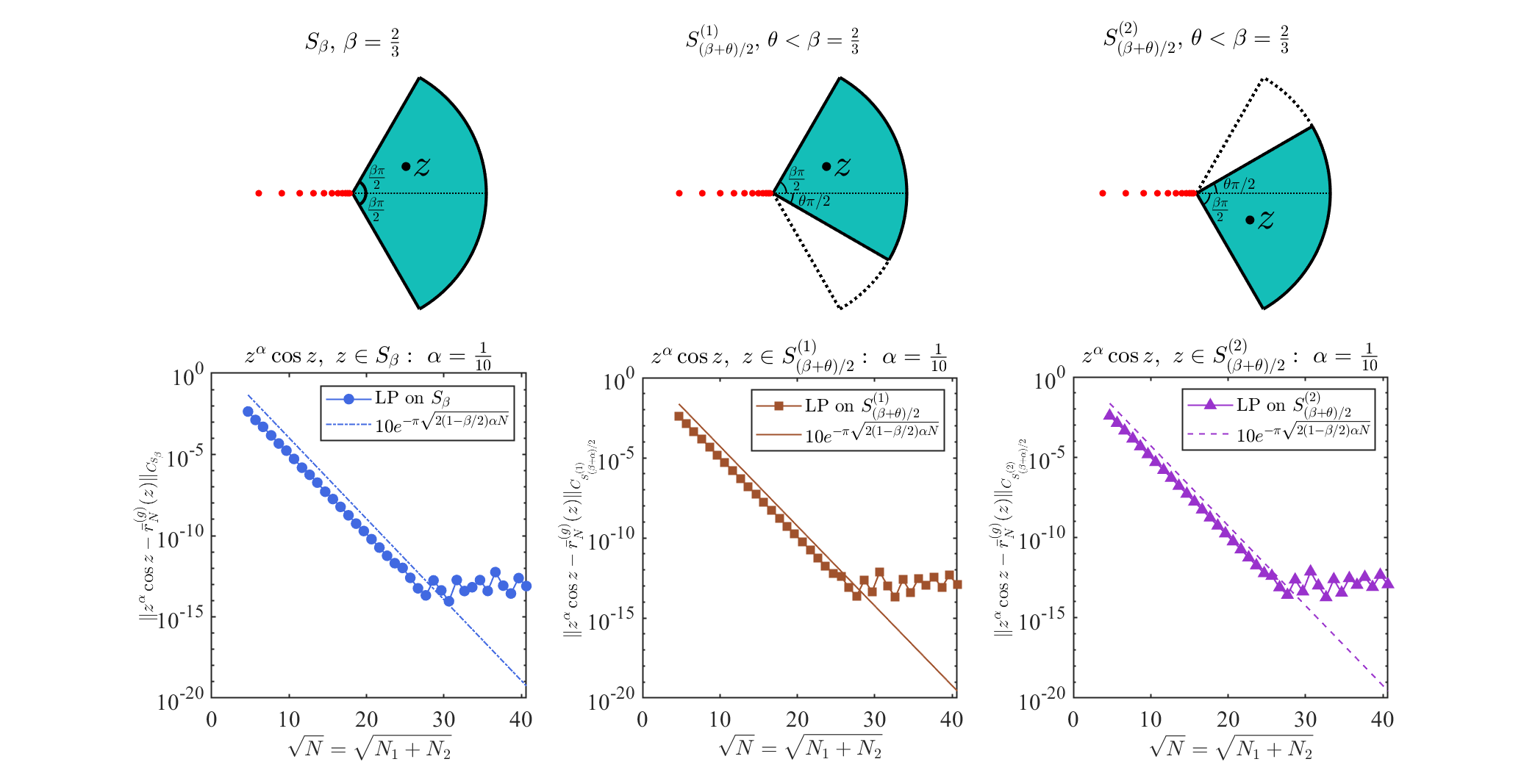}}
\centerline{\includegraphics[width=16cm]{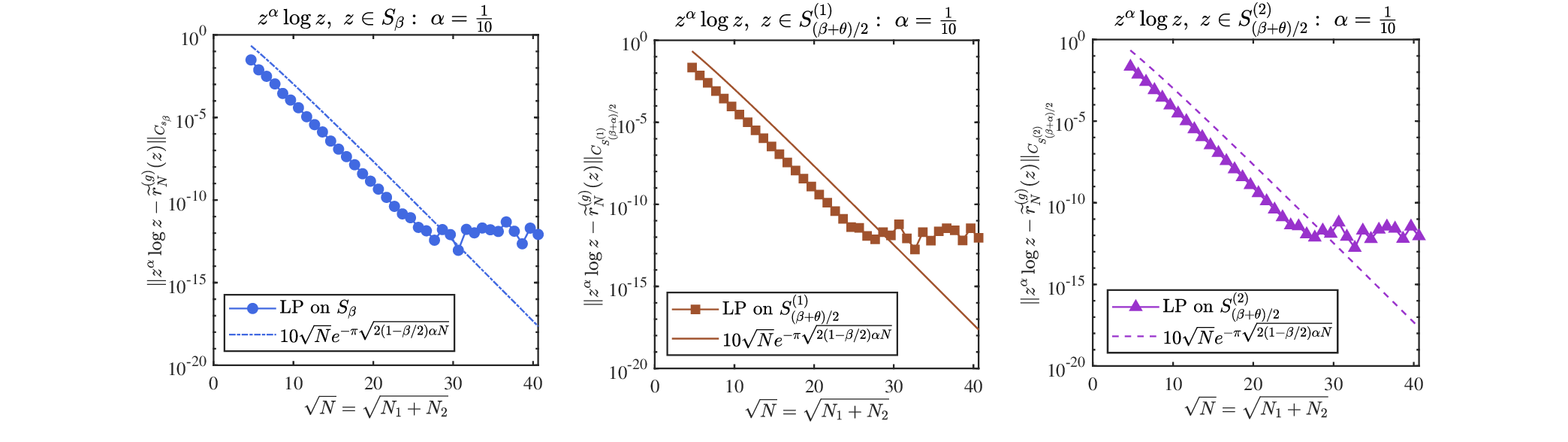}}
\caption{The convergence rates of LPs on the symmetric and asymmetric sector domains. All the cases $S_{\beta}$ and $S^{(l)}_{(\beta+\theta)/2},\ l=1,2$ achieve the theoretically predicted convergence rate \eqref{eq: rate1}.}
\label{LPwithuniformclusterpole_rotated}
\end{figure}

\begin{remark}
It is obvious that Theorem \ref{mainthm} also holds for the general sector domain with arbitrary positive radius $R>0$ and central angle $\beta\pi, \beta\in[0,2)$. 
\end{remark}

\section{Applications to conformal mappings on planar corner
domains }\label{LPapp_cornerdomain}
In the final section, we investigate a particular Laplace boundary value problem on a corner domain, which is closely related to the classical conformal mapping.
These holomorphic mappings serve as pivotal instruments for complex analysis.

Suppose $\Omega$ is a simply connected domain with a Jordan boundary, and $D$ denotes the unit disk. One aim is to seek a conformal mapping $f:\ \Omega\to D$
satisfying $f(z_0)=0$ and $f'(z_0)>0$ for some interior point $z_0\in\Omega$.
Without loss of generality, we assume $z_0=0\in\Omega$. According to Riemann mapping theorem \cite[p. 230, Theorem 1]{AhlforsComplex} and boundary correspondence theorem \cite[p. 233, Theorem 2]{AhlforsComplex}, there exists such a unique conformal mapping $f$ satisfying $f(\partial\Omega)\to\partial D$ and preserving the boundary orientation.
Moreover, the Osgood-Carathodory theorem \cite[p. 346, Theorem 16.3a]{Henricivol3} further guarantees that
$f$ admits a homeomorphic extension to $\partial\Omega$.
Thus, it follows that
\begin{align}\label{eq:g(z)}
g(z)=\log\left[\frac{f(z)}{z}\right]
\end{align}
is holomorphic in $\Omega$, with its real part $u(z)$ satisfying $u(z)=-\log|z|$ on $\partial\Omega$. The imaginary part $v(z)$ of $g(z)$, with $v(0)=0$, can be uniquely determined by the conjugate harmonic function of $u(z),\ z\in\Omega$,  as shown in \cite[pp. 370-371, Theorem 16.5a]{Henricivol3}. Additionally, from
\eqref{eq:g(z)} we have $f(z)=ze^{g(z)},\ z\in\overline\Omega$.
Consequently, the original primary objective reduces to solving the Laplace boundary problem
\begin{align}
\left\{\begin{array}{ll}
\Delta u=0,&z\in\Omega,\\
u(z)=-\log|z|,&z\in\partial\Omega.
\end{array}\right.\label{eq:Laplaceu}
\end{align}

As is introduced in Section \ref{sec:Int},
the solution $u(x,y)$ of the Laplace equation in domain $\Omega$ is the real part of a holomorphic function $f(z)$ \cite[p. 368, Theorem 6.1.2]{Asmar2018}, which
can be asymptotically represented by
\begin{equation}\label{powerseriesform}
g(z)-g(w_k)\sim\left\{\begin{array}{ll}
\sum\limits_{\displaystyle \substack{\iota,\gamma}}c^{(k)}_{\iota,\gamma}(z-w_k)^{\iota+\gamma\alpha_k}\sim (z-w_k)^{\alpha_k},\hspace{4.2cm} \alpha_k \mbox{ irrational}\\
\sum\limits_{\displaystyle \substack{\iota,\gamma,\tau}}c^{(k)}_{\iota,\gamma,\tau}(z-w_k)^{\iota+\gamma\alpha_k}[(z-w_k)^{\mu_k}\log{(z-w_k)}]^\tau\sim (z-w_k)^{\alpha_k},
\ \alpha_k \mbox{ rational}\\
\end{array}\right.
\end{equation}
uniformly in any finite sector  $S_{\beta_k}$ as $z\rightarrow w_k$,
where the terms are arranged in increasing order and $\alpha_k=1/\varphi_k$, and $\alpha_k=\frac{q_k}{\mu_k}$, $(q_k, \mu_k)=1$ if $\varphi_k$ is rational, $k=1,2,\ldots,m$. See \cite[Theorem 1]{Lehman1957DevelopmentOT} and \cite[Theorems 3, 4 and 5]{Wasow} for more details.

Naturally, the function $g(z)$ on  $\Omega$ with isolated branch points at the vertices $w_k,\ k=1,\cdots,m$ may be approximated well by
an LP approximation \eqref{LP_cornerdomain},
\begin{align*}
r_n(z)=\sum_{k=1}^m\sum_{j=0}^{N_{1,k}}\frac{a_{k,j}}{z-p_{k,j}}+\sum_{j=0}^{N_2} b_{j}z^j
\end{align*}
with lightning poles $\{p_{k,j}\}$ that are uniformly exponentially clustered with parameter $\sigma_k=\frac{\sqrt{2-\beta_k}\pi}{\sqrt{\alpha_k}}$ towards every corner $w_k$ of sector domain along the exterior bisector.
Following \cite{TREFETHEN_SERIES}, the coefficients $\{a_{k,j}\}$ and $\{b_j\}$ 
are calculated by seeking the least-squares solution of
\begin{align*}
u(z)\approx\Re(r_n(z))
=&\sum_{k=1}^m\sum_{j=0}^{N_{1,k}}\Big\{\Re(a_{k,j})\Re\left[(z-p_{k,j})^{-1}\right]
-\Im(a_{k,j})\Im\left[(z-p_{k,j})^{-1}\right]\Big\}\notag\\
&+\sum_{j=0}^{N_2} \big[\Re(b_{j})\Re(z^j)-\Im(b_j)\Im(z^j)\big]
\end{align*}
with the boundary condition in \eqref{eq:Laplaceu}, based on Arnoldi  orthogonalization algorithm \cite{BNT2021}.
For further details, see \cite{costa2020solvinglaplaceproblemsaaa,costa2023aaa,Gopal2019,Treweb}.

Consequently, the function $g(z)$ in the exponent of the conformal mapping $f(z)$ can be well approximated by the LP scheme $r_n(z)$ given in \eqref{LP_cornerdomain}, 
 yielding
\begin{align*}
f(z)=ze^{g(z)}\approx ze^{r_n(z)}=:\tilde{f}_n(z).
\end{align*}
In particular, from the power series form \eqref{powerseriesform}, the Gopal-Trefethen canonical decomposition \eqref{decompose_singularity} applied to $g$, and Theorem \ref{mainthm}, each component $g_k(z)=\frac{1}{2\pi i}\int_{\Lambda_k}\frac{g(\zeta)}{\zeta-z}\mathrm{d}\zeta$ satisfies
\[g_k(z)-r_{n_k}(z)=\mathcal{O}(1)\left(e^{-\pi\sqrt{(2-\beta_k)n_k\alpha_k}}\right),\ \ n_k\to\infty\]
and $n_k=N_{1,k}+\mathcal{O}\left(\sqrt{N_{1,k}}\right),\ k=1,\cdots,m$. 

Moreover, the analytic term $\frac{1}{2\pi i}\sum_{k=1}^m\int_{\Gamma_k}\frac{g(\zeta)}{\zeta-z}\mathrm{d}\zeta$
can be approximated by a polynomial of degree $\mathcal{O}\left(\sqrt{N_{1,k}}\right)$ with error
$$\mathcal{O}(1)\max_{1\le k\le m}\left\{e^{-\pi\sqrt{(2-\beta_k)n_k\alpha_k}}\right\}.$$ 
Combining these estimates yields a rational function $r_n$
such that
\[g(z)-r_n(z)=\mathcal{O}(1)\max_{1\le k\le m}\left\{e^{-\pi\sqrt{(2-\beta_k)n_k\alpha_k}}\right\},\ \ n_k\to\infty.\]
Therefore, the convergence rate of $\tilde{f}_n(z)$
is given by 
\begin{align*}
\left|f(z)-\tilde f_n(z)\right|=&\left|ze^{g(z)}\right|\left|1-e^{r_n(z)-g(z)}\right|
=\left|ze^{g(z)}\right|\left|1-e^{\mathcal{O}(1)\max_{1\le k\le m}\left\{e^{-\pi\sqrt{(2-\beta_k)n_k\alpha_k}}\right\}}\right|\\
=&\mathcal{O}(1)\max_{1\le k\le m}\left\{e^{-\pi\sqrt{(2-\beta_k)n_k\alpha_k}}\right\},\ n_k\to\infty.
\end{align*}

\begin{figure}[hb!]
\centerline{\includegraphics[width=12cm]{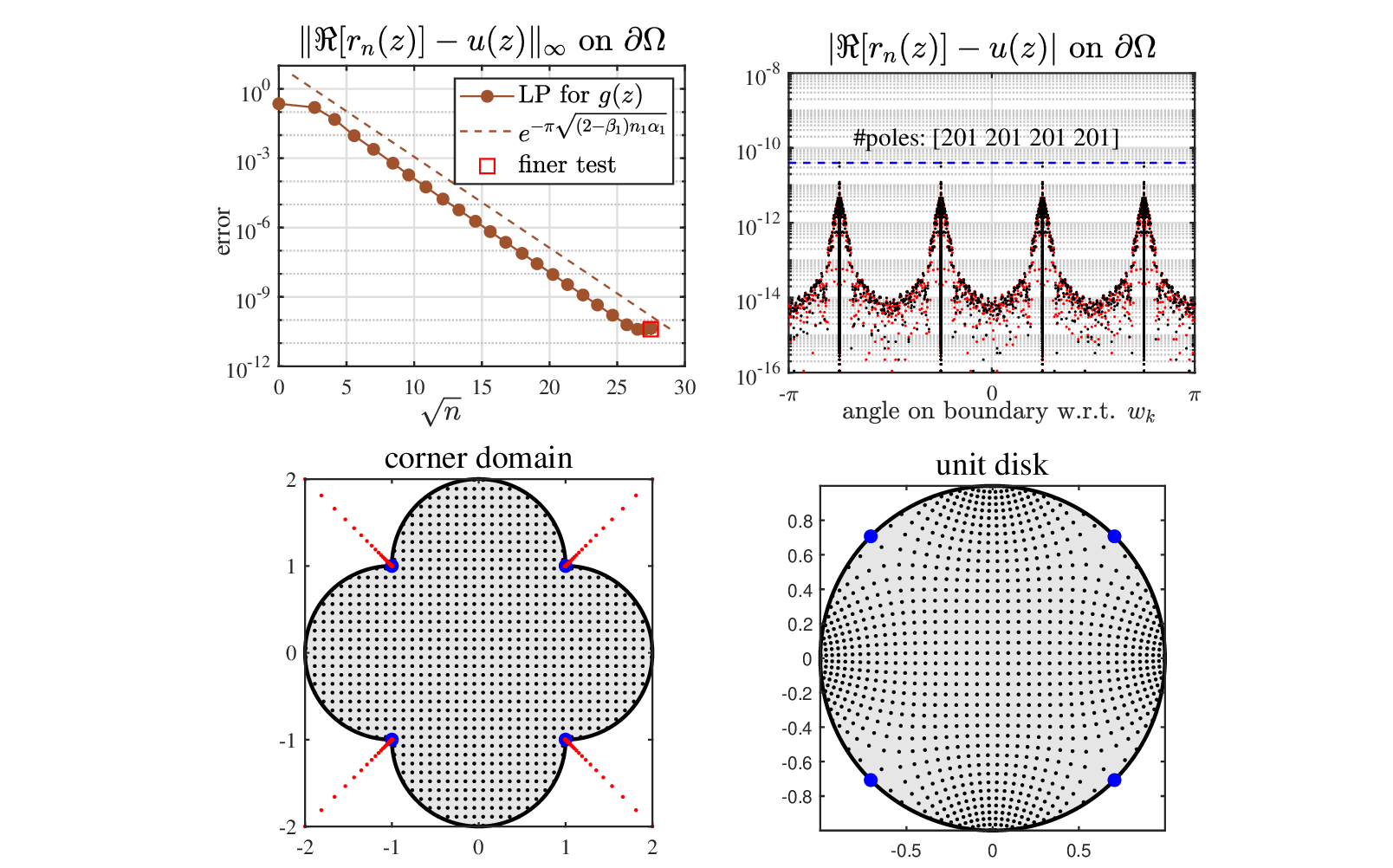}}
\caption{Error decay rate (upper left) for LP numerical solutions $\Re[r_n(z)]$s of \eqref{eq:Laplaceu} on curvy square domain with vertices $[w_1,w_2,w_3,w_4]=[-1,1,1,-1]+i[-1,-1,1,1]$, where $n=4n_1$.
The pointwise errors on $\partial D$ (upper right)
for the last $\Re[r_n(z)]$ (we use $201$ poles for each vertex) is sketched in the upper left subplot, whose abscissa indicates the rotation angle of the boundary points of $\Omega$
about its geometric center.
In the second row, the map $\widetilde f_n(z)$ conformally maps the equal spaced interior points (black) of $\Omega$ to the deformed interior points (black) of $D$, and the four corners are mapped onto the four blue points on the boundary of $D$.}
\label{ConformalCS}
\end{figure}

\begin{figure}[h!]
\centerline{\includegraphics[width=12cm]{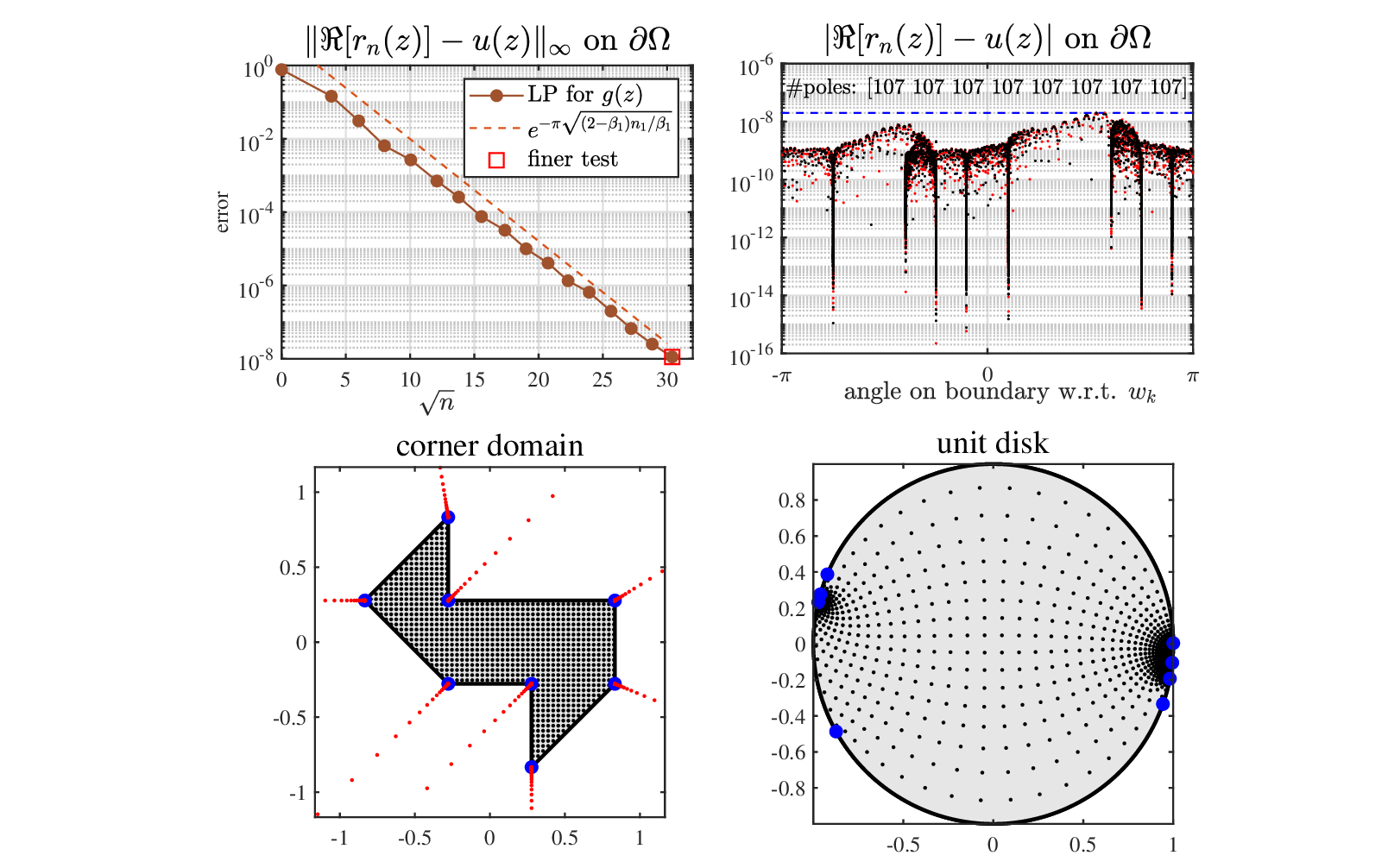}}
\caption{Error decay rate (upper left) for LP solutions $\Re[r_n(z)]$s of \eqref{eq:Laplaceu} on corner domain $\Omega$ with vertices $\frac{5}{18}\big[(-1,-1,-3,-1,1,1,3,3)+i(1,3,1,-1,-1,-3,-1,1)\big]$, where $n=8n_1$. The pointwise errors of the final $\Re[r_n(z)]$ (we use $107$ poles for each vertex) are shown in the upper-right subplot. In the second row, the conformal map $\tilde f_n(z)$ transforms equally-spaced interior points (black) of
$\Omega$ into deformed interior points (black) of $D$, while mapping the eight corners to eight blue points on $\partial D$.}
\label{conformal_complexDomain}
\end{figure}

For the numerical method for conformal map based on the LP scheme, we illustrate its superior performance by two examples on the polygon or curvy quadrilateral domains $\Omega$, see {\sc Figures} \ref{ConformalCS} and \ref{conformal_complexDomain}.
The presented {\sc Figures}  \ref{ConformalCS} and \ref{conformal_complexDomain}  exclusively illustrate the decay rate of $|\Re[r_n(z)]-u(z)|_{\infty}$ rather than $|\tilde f_n(z)-f(z)|$. This selective representation stems from a fundamental constraint: the true solution $f(z)$ for the conformal mapping from domain $\Omega$ to $D$ remains analytically unavailable, precluding direct error computation of the approximate mapping function $\widetilde f_n(z)$.
Following the methodology outlined in  \cite{Gopal2019}, the terminal red square ``{\color{red}$\square$}'' on the convergence curve denotes the boundary error measurement benchmarked against a refined computational grid-specifically, a mesh with double the resolution of that employed in the original least-squares discretization.
The pointwise errors in numerical experiments are plotted for $\Re[r_n(z)]$ that reaches the allowed largest degree, with the black and red points representing the errors on the original and finer test sample points (black and red, respectively).

\section*{Acknowledgement}
This work was supported by the National Natural Science Foundation of China (No. 12271528) and Hunan Basic Science Research Center for Mathematical Analysis (2024JC2002).

\bibliographystyle{plain}

\end{document}